\documentclass[11pt,a4paper]{article}
\usepackage[utf8]{inputenc}
\usepackage[a4paper, total={6.6in, 9.6in}]{geometry}

\usepackage[normalem]{ulem}
\usepackage[svgnames]{xcolor}
\usepackage[T1]{fontenc}
\usepackage[variablett]{lmodern}
\usepackage{amsmath}
\usepackage{amsfonts}
\usepackage{amssymb}
\usepackage{mathrsfs}
\usepackage{bbm}
\usepackage{float}
\usepackage{wrapfig}
\usepackage{enumitem}
\usepackage{tabularray}
\usepackage{tikz}
\usetikzlibrary{backgrounds,patterns,calc}
\usepackage{tikz-cd}
\usepackage{pgfplots}
\pgfplotsset{compat=newest,ticks=none}
\usepackage{graphicx}

\usepackage[backend=biber,style=alphabetic]{biblatex}
\bibliography{flimsy}

\usepackage{caption}
\usepackage{xurl}
\definecolor{mycolor}{HTML}{750000}

\newcommand\customref[2]{\hyperref[#2]{#1}}
\newcommand\iref[2]{\customref{\Cref*{#1}~\ref*{#2}}{#1}}
\newcommand\axref[1]{\customref{axiom~\ref*{#1}}{#1}}
\newcommand\Axref[1]{\customref{Axiom~\ref*{#1}}{#1}}

\newcommand{\andd}{\qquad\textnormal{and}\qquad}
\newcommand{\suchthat}[2]{\left\{#1\ \middle\vert\ \begin{matrix}#2\end{matrix}\right\}}

\newcommand{\leftl}{\mathopen{}\mathclose\bgroup\left}
\newcommand{\rightr}{\aftergroup\egroup\right}

\renewcommand{\emptyset}{\varnothing}
\renewcommand{\epsilon}{\varepsilon}
\renewcommand{\phi}{\varphi}
\renewcommand{\bar}{\overline}
\renewcommand{\hat}{\widehat}
\renewcommand{\tilde}{\widetilde}

\newcommand{\N}{\mathbb N}
\newcommand{\Z}{\mathbb Z}
\newcommand{\Q}{\mathbb Q}
\newcommand{\R}{\mathbb R}

\newcommand{\bbS}{\mathbb S}
\newcommand{\rK}{\mathrm K}
\newcommand{\rH}{\mathrm H}

\newcommand{\rT}{\mathrm T}
\newcommand{\cC}{\mathcal C}
\newcommand{\cF}{\mathcal F}

\newcommand{\sT}{\mathscr T}
\newcommand{\sE}{\mathscr E}
\newcommand{\sF}{\mathscr F}
\newcommand{\ttN}{\mathtt{N}}

\newcommand{\nbot}{\not\perp}
\newcommand{\card}[1]{\leftl|#1\rightr|} 
\newcommand{\comp}[1]{{#1}^{\complement}} 
\newcommand{\compb}[1]{\comp{\{#1\}}}

\newcommand{\lex}{\mathrm{lex}}
\newcommand{\eel}{\tilde{0}}
\newcommand{\eer}{\tilde{1}}
\newcommand{\re}{\mathrm{e}}

\newcommand{\convT}{\overset\sT\longrightarrow}
\newcommand{\convE}{\overset\sE\longrightarrow}

\newcommand{\words}{2^{<\N}}
\newcommand{\zero}{{}^\frown0}
\newcommand{\one}{{}^\frown1}

\DeclareMathOperator{\Cho}{Chords} 

\DeclareMathOperator{\BCir}{BCir}
\DeclareMathOperator{\Sepa}{Sepa}
\DeclareMathOperator{\Conn}{Conn}

\usepackage{amsthm}
\usepackage{enumitem}
\usepackage[hidelinks,unicode,naturalnames,hypertexnames=false]{hyperref}
\usepackage{thmtools}
\usepackage[english,noabbrev,sort,nameinlink]{cleveref}

\hypersetup{
    colorlinks = true,
    allcolors = mycolor,
    bookmarksdepth = 2,
}
\newcounter{mycounter}[section]

\theoremstyle{plain}
\newtheorem{theorem}[mycounter]{Theorem}
\newtheorem{corollary}[mycounter]{Corollary}
\newtheorem{proposition}[mycounter]{Proposition}
\newtheorem{lemma}[mycounter]{Lemma}

\theoremstyle{remark}
\newtheorem{remark}[mycounter]{Remark}

\theoremstyle{definition}
\newtheorem{definition}[mycounter]{Definition}
\newtheorem{example}[mycounter]{Example}

\counterwithin{equation}{section}

\Crefname{subsection}{Subsection}{Subsections}

\usepackage{enumitem}
\newlist{enumroman}{enumerate}{1}
\setlist[enumroman, 1]{label=(\roman*)}

\usepackage{titletoc}
\titlecontents{section}[0pt]{}{\contentsmargin{0pt}\large\thecontentslabel.\enspace}{\contentsmargin{0pt}\large}{\titlerule*[.5pc]{ }\contentspage}[]

\usepackage{titlesec}
\titleformat{\section}[block]{\normalfont\centering\scshape\large}{\thesection.}{1em}{}
\titleformat{\subsection}[block]{\normalfont\large}{\thesubsection.}{1em}{\bf}

\renewenvironment{abstract}{%
\par\noindent\rule{\textwidth}{1pt}
\par\noindent\textsc{Abstract.}}
{\par\noindent\rule{\textwidth}{1pt}}

\title{
  Flimsy Spaces
}

\author{
  Robin Khanfir%
  \footnote{
    Department of Mathematics and Statistics,
    McGill University.
    Email: \texttt{robin.khanfir@mcgill.ca}
  }
  \and
  Béranger Seguin%
  \footnote{
    Universität Paderborn,
    Fakultät EIM,
    Institut für Mathematik,
    Warburger Str.\ 100,
    33098 Paderborn,
    Germany.
    Email: \texttt{bseguin@math.upb.de}.
  }
}
\date{November 21, 2025}

\begin{document}

\maketitle{}

\begin{abstract}
  We study \emph{$n$-flimsy spaces}, which are the topological spaces that remain connected when removing fewer than $n$ points but become disconnected when removing exactly~$n$ points.
  We show that no such space exists for $n \geq 3$, and that the compact $2$-flimsy spaces are precisely the dense and order-complete cyclically ordered sets equipped with their order topology.
  Furthermore, we examine variants of the definition obtained by replacing connectedness by path-connectedness, where paths are either parametrized by~$[0,1]$ or by arbitrary compact linear continua.

  \bigskip

 
 \smallskip
 
\end{abstract}

{
  \hypersetup{linkcolor=black}
  \tableofcontents{}
}
\par\noindent\rule{\textwidth}{1pt}

\section{Introduction}

\subsection{Context}

A classical argument showing that~$\R$ is not homeomorphic to~$\R^d$ for~$d > 1$ goes as follows: when removing any single point, the space~$\R$ becomes disconnected, whereas~$\R^d$ stays connected.
A tempting generalization is to remove more than a single point, leading to the following definition:

\begin{definition}
\label{def:flimsy_top}
  Let~$n \in \N$.
  A topological space~$X$ is \emph{$n$-flimsy} if~$|X| > n$ and if, for any subset~$S \subseteq X$, the set $X \setminus S$ is connected whenever $\card S < n$ and is not connected whenever $\card S = n$.
\end{definition}

For instance, the line~$\R$ is $1$-flimsy, and the circle~$\bbS^1 := \R/\Z$ is $2$-flimsy.
On the other hand, when $d \geq 2$, the spaces~$\R^d$ and~$\bbS^d$ are not $n$-flimsy for any~$n$.
Flimsiness is a topological invariant of a space---in fact, it is preserved not only by homeomorphisms, but by any bijective map preserving connectedness and non-connectedness.

In the literature, $1$-flimsy spaces have been studied under the name \emph{cut-point spaces}, cf.~\cite{cutpoint}.
It has been shown that there are no compact cut-point spaces, and that there is a specific countable cut-point space that is minimal in a certain sense, the \emph{Khalimsky line}.
Other examples of $1$-flimsy spaces include any topological tree without leaves,%
\footnote{
  In particular, the number of connected components of the space obtained after removing a single point may depend on the point.
  In \Cref{only-two-components}, we will see that $2$-flimsy spaces are more constrained.
}
two copies of the topologist's sine curve glued back to back (a non-path-connected example), and the long line (more generally, any linear continuum without ends, cf.~\Cref{cor:lin-cont-1f}).

There has also been previous work on $2$-flimsy spaces (under the name \emph{cut$^*$ spaces}) and $n$-flimsy spaces (under the name \emph{cut$^{(n)}$-spaces}) in the articles \cite{penghuang,pengcao}.
There, it was shown that there are no locally compact $n$-flimsy spaces when $n \geq 3$ (a result that we improve in \Cref{thm:no-3f}) and that all $n$-flimsy spaces are Hausdorff when $n \geq 2$ (a result that we recover in \Cref{thm:2f-is-hausdorff}).
The term \emph{flimsy spaces} was coined in a 2018 \texttt{math.stackexchange} question by user ``Babelfish'', who explicitly raised the question of the existence of a $3$-flimsy space (which we answer in \Cref{thm:no-3f}).%
\footnote{
  See \url{https://math.stackexchange.com/questions/2939445/flimsy-spaces-removing-any-n-points-results-in-disconnectedness} for the original question.
  See also \url{https://math.stackexchange.com/questions/3044369/more-on-flimsy-spaces} for a version taking higher homotopy groups into account.
}

Besides the circle, examples of $2$-flimsy spaces include a non-compact variant of the Warsaw circle (\Cref{ex:non-compact-2f}), as well as the (non-path-connected) one-point compactification of the long line (more generally, any big circle, see~\Cref{def:big-circle} and \Cref{prop:big-circle-2f}).
However, we shall see that $2$-flimsy spaces are less diverse than cut-point spaces: they are all somewhat circle-like.

\paragraph{Other notions of connectedness.}

Besides the usual notion of connectedness, the definition of flimsy spaces naturally extends to any other notion expressing the idea of ``being in one piece''.
In this paper, we consider two variants.
The first is the standard notion of \emph{path-connectedness}, requiring that any two points be joined by a \emph{path}---a continuous map from the standard interval~$[0,1]$.
The second, called \emph{big-path-connectedness}, allows paths parametrized by longer intervals.
More precisely, a \emph{big path} is a continuous map whose domain is a \emph{big interval}---a totally-ordered set that is compact and connected when equipped with the order topology (\Cref{def:lin-cont}).
This terminology is borrowed from~\cite{cannon2,penrod}, where big paths are used to extend homotopy theory to spaces that are ``so large'' that paths and homotopies parametrized by $[0,1]$ do not adequately capture their geometry.

\begin{definition}
  \label{def:flimsy_path}
  Let~$n \in \N$.
  A topological space~$X$ is \emph{$n$-(big-)path-flimsy} if~$|X| > n$ and if, for any subset~$S \subseteq X$, the set $X \setminus S$ is (big-)path-connected whenever $\card S < n$ and is not (big-)path-connected whenever $\card S = n$.
\end{definition}

\subsection{Main results}

Our first main result is the following theorem, answering the original question of Babelfish:

\begin{theorem}[cf.~\Cref{thm:no-3f-cspace} and~\Cref{prop:top-is-cspace,prop:path-is-cspace,prop:continuum-is-cspace}]
  \label{thm:no-3f}
  Let $n \geq 3$ be an integer.
  There are no $n$-flimsy spaces, no $n$-big-path-flimsy spaces, and no $n$-path-flimsy spaces.
\end{theorem}

Since there are no $n$-flimsy spaces when $n>2$, and since $1$-flimsy spaces (cut-point spaces) are already well-studied, the remainder of the article is devoted to $2$-flimsy spaces and to their classification.
The first non-trivial properties of $2$-flimsy spaces that we prove are that removing any two distinct points always leaves exactly two connected components (\Cref{only-two-components}), and that the $2$-flimsy spaces are exactly the $\rT_1$ spaces in which the complement of any connected subset is connected (\Cref{flimsy-is-t1}).
The same properties hold for $2$-(big-)path-flimsy spaces, replacing the notion of connectedness everywhere appropriately.

A key fact for the classification of $2$-flimsy spaces is the following connection between all $2$-flimsy spaces and compact $2$-flimsy spaces.

\begin{theorem}[proved in \Cref{sn:2f-top}]
  \label{thm:intro-2f-top}
  Let $(X, \tau)$ be a $2$-flimsy topological space.
  Then, $X$ is Hausdorff, and there exists a unique topology~$\tau_\rK$ on~$X$ coarser than~$\tau$ such that $(X, \tau_\rK)$ is a compact $2$-flimsy space.
  Namely, $\tau_\rK$ is the topology generated by the connected components of the subsets $X\setminus \{x,y\}$ as $(x,y)$ ranges over all pairs of distinct points of~$X$.
  Furthermore, $(X, \tau)$ and $(X, \tau_\rK)$ have the same connected subsets and the same open connected subsets.
\end{theorem}

For $2$-big-path-flimsy spaces, we obtain a dual statement in which the roles of compactness and Hausdorffness are swapped.

\begin{theorem}[proved in \Cref{sn:2f-big-path}]
  \label{thm:intro-2f-big_path}
  Let $(X, \tau)$ be a $2$-big-path-flimsy topological space.
  Then, $X$ is compact, and there exists a unique topology~$\tau_\rH$ on $X$ finer than~$\tau$ such that $(X, \tau_\rH)$ is a Hausdorff $2$-big-path-flimsy space.
  Namely, $\tau_\rH$ is the topology generated by the big-path-connected components of the subsets $X\setminus \{x,y\}$ as $(x,y)$ ranges over all pairs of distinct points of~$X$.
  Furthermore, $(X, \tau)$ and~$(X, \tau_\rH)$ have the same big-path-connected subsets and the same big paths.
\end{theorem}

We have an exactly analogous statement for path-connectedness.

\begin{theorem}[proved in \Cref{sn:2f-path}]
  \label{thm:intro-2f-path}
  Let $(X, \tau)$ be a $2$-path-flimsy topological space.
  Then, $X$ is compact, and there exists a unique topology~$\tau_\rH$ on $X$ finer than~$\tau$ such that $(X, \tau_\rH)$ is a Hausdorff $2$-path-flimsy space.
  Namely, $\tau_\rH$ is the topology generated by the path-connected components of the subsets $X\setminus \{x,y\}$ as $(x,y)$ ranges over all pairs of distinct points of~$X$.
  Furthermore, $(X, \tau)$ and $(X, \tau_\rH)$ have the same path-connected subsets and the same paths.
\end{theorem}

The next step is to classify compact $2$-flimsy spaces and Hausdorff $2$-(big-)path-flimsy spaces.
This is done by relating them to order theory.
The following theorem vastly generalizes \cite[Theorem~28.14]{willard}, which deals with the special case of metric spaces:

\begin{theorem}[proved in \Cref{sn:equi}]
  \label{thm:intro-classification}
  Let~$X$ be a topological space.
  Then, the following assertions are equivalent.
  Furthermore, if they hold, a subset of~$X$ is connected if and only if it is big-path-connected.
  \begin{enumroman}
    \item
      \label{item:intro-classification-compact-top}
      The space $X$ is compact and $2$-flimsy.
    \item
      \label{item:intro-classification-hausdorff-bp}
      The space $X$ is Hausdorff and $2$-big-path-flimsy.
    \item
      \label{item:intro-classification-top-bp}
      The space $X$ is $2$-flimsy and $2$-big-path-flimsy.
    \item
      \label{item:intro-classification-big-circle}
      The space $X$ is a big circle (cf.~\Cref{def:big-circle}).
  \end{enumroman}
\end{theorem}

Although \Cref{thm:intro-classification} gives a classification of compact $2$-flimsy spaces and of Hausdorff $2$-big-path-flimsy spaces in terms of the order-theoretic notion of linear continuum (cf.~\Cref{def:lin-cont,def:big-circle}), it should be noted that the list of linear continua depends heavily on the specific axioms chosen for set theory, cf.~\cite{jechsuslin}.

In contrast, the catalog of Hausdorff $2$-path-flimsy spaces is reduced to a single reference object:

\begin{theorem}[cf.~\Cref{thm:2pf-is-compact}]
  \label{thm:hausdorff-2pf-is-circle}
  Any Hausdorff $2$-path-flimsy space is homeomorphic to the standard circle~$\bbS^1$.
\end{theorem}

The assumptions of our theorems are minimal, and the three different notions of $2$-flimsiness are generally independent of one another.
Indeed, \Cref{ex:non-compact-2f} exhibits a non-compact $2$-flimsy space (a variant of the Warsaw circle), which is thus neither $2$-big-path-flimsy nor $2$-path-flimsy.
Moreover, we prove in \Cref{ex:non-hausdorff-2bpf} that keeping only the co-countable open subsets in the topology of $\bbS^1$ yields a non-Hausdorff $2$-big-path-flimsy space, which is thus not $2$-flimsy.
While it turns out that this particular example is also $2$-path-flimsy (see \Cref{ex:non-hausdorff-2pf}), we obtain many $2$-big-path-flimsy spaces that are not $2$-path-flimsy using \Cref{thm:intro-classification,thm:hausdorff-2pf-is-circle}, such as the big circles that come from extended long lines (see also \Cref{lexico_order_prop}).
Constructing a $2$-path-flimsy space that is not $2$-big-path-flimsy is more complex.
This is the purpose of \Cref{sn:2f-path-not-big-path}, which does so by endowing~$\bbS^1$ with a highly pathological topology defined using a map that, loosely speaking, visits every place of the circle at every moment.
The construction of that map relies on a generalization of Bernstein sets for products of topological spaces (\Cref{bernstein_rectangle}), which may be of independent interest.

\subsection{Connectivity spaces, informally}

In fact, a significant part of our work studies \Cref{def:flimsy_top,def:flimsy_path} in a unified way, using an axiomatic notion of connected subsets that encompasses connectedness (\Cref{prop:top-is-cspace}), big-path-connectedness (\Cref{prop:continuum-is-cspace}), and path-connectedness (\Cref{prop:path-is-cspace}).
This notion of \emph{connectivity spaces} may be of independent interest---we present it in detail in \Cref{subsn:c-spaces}.
It consists of the four following axioms (see \Cref{def:c-space} for a more formal description):
\begin{itemize}
  \item
    any singleton is connected;
  \item
    any union of non-disjoint connected subsets is connected;
  \item
    if removing a point~$x$ from a connected subset~$A$ disconnects it, then~$A$ remains connected upon removing any of the connected components of~$A \setminus \{x\}$;
  \item
    if $A$ and $B$ are two connected subsets whose union is also connected, then there are two possibilities: either there is a point~$x \in A \cup B$ such that both $A \cup \{x\}$ and $B \cup \{x\}$ are connected (a ``button''), or there is a connected subset $S \subseteq A\cup B$ whose intersection with $A$ or with $B$ is not connected (a ``seam'').
    (This is a weak converse to the second axiom.)
\end{itemize}

\paragraph{Illustration of the axioms.}

The following figure illustrates the third axiom:
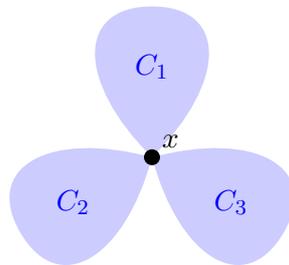
\begin{figure}[H]
  \begin{center}
    \begin{tikzpicture}[scale=2]
      \foreach \a in {1,2,3} {
        \pgfmathsetmacro\angle{120*(\a-1)}
        \fill[blue!20,rotate=\angle]
          (0,0)
            .. controls (0.5,0.5) and (0.5,1) .. (0,1)
            .. controls (-0.5,1) and (-0.5,0.5) .. (0,0)
          -- cycle;
          \node[blue] at ({\angle+90}:.6) {$C_\a$};
      }

      \fill[black] (0,0) circle (1.5pt) node[above right] {$x$};
    \end{tikzpicture}
    \caption{In this example, removing $x$ yields three components $C_1$, $C_2$ and $C_3$, and indeed removing~$C_i$ (for any $i \in \{1,2,3\}$) does not disconnect the figure.}
    \label{fig-flower}
  \end{center}
\end{figure}

The next two figures illustrate the two situations that the fourth axiom allows:
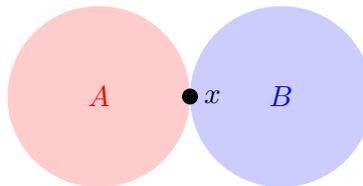
\begin{figure}[H]
  \begin{center}
    \begin{tikzpicture}
      \fill[red!20] (0,0) circle (1.2);
      \node[red] at (0,0) {$A$};
      \fill[blue!20] (2.4,0) circle (1.2);
      \node[blue] at (2.4,0) {$B$};
      \fill[black] (1.2,0) circle (3pt) node[right=1.5pt] {$x$};
    \end{tikzpicture}
    \caption{A ``button'' situation.}
    \label{fig-button}
  \end{center}
\end{figure}
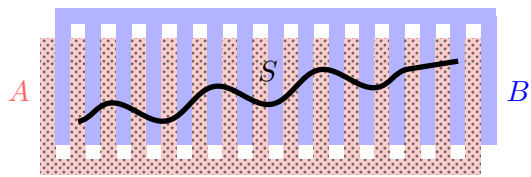
\begin{figure}[H]
  \begin{center}
    \begin{tikzpicture}
      \def\boxwidth{6}
      \def\boxheight{2}
      \def\nstripes{30}

      \foreach \i in {0,...,\numexpr\nstripes-1} {
        \pgfmathsetmacro{\x}{\i * \boxwidth/\nstripes}
        \ifodd\i
          \fill[blue!30] (\x,3/2*\boxwidth/\nstripes) rectangle (\x+\boxwidth/\nstripes,\boxheight);
        \else
          \begin{scope}
            \fill[red!20] (\x,0) rectangle ++(\boxwidth/\nstripes,\boxheight-3/2*\boxwidth/\nstripes);
            \pgfsetfillpattern{crosshatch dots}{black!60}
            \fill (\x,0) rectangle ++(\boxwidth/\nstripes,\boxheight-3/2*\boxwidth/\nstripes);
          \end{scope}
        \fi
      }

      \begin{scope}
        \fill[red!20] (0,-\boxwidth/\nstripes/2) rectangle ++(\boxwidth-\boxwidth/\nstripes,\boxwidth/\nstripes);
        \pgfsetfillpattern{crosshatch dots}{black!60}
        \fill (0,-\boxwidth/\nstripes/2) rectangle ++(\boxwidth-\boxwidth/\nstripes,\boxwidth/\nstripes);
      \end{scope}
      \fill[blue!30] (\boxwidth/\nstripes,\boxheight-\boxwidth/\nstripes/2) rectangle (\boxwidth,\boxheight+\boxwidth/\nstripes/2);

      \draw[black,line width=2pt,decorate,decoration={snake,amplitude=5pt,segment length=40pt}]
        (0.5,0.6) -- (5.5,1.4);

      \node[red!60, left] at (0,1) {$A$};
      \node[blue, right] at (\boxwidth,1) {$B$};
      \node[black,above] at (3,1) {$S$};
    \end{tikzpicture}
  \caption{
    A ``seam'' situation.
    The red (dotted) region~$A$ consists of the bottom horizontal edge together with the red vertical stripes, and the blue (non-dotted) region~$B$ consists of the top edge together with the blue vertical stripes.
    If the red and blue stripes alternate ``densely often'' and we are thinking in terms of (big-)path-connectedness, then it is impossible to find a button as in \Cref{fig-button}, but a seam~$S$ does indeed exist.
  }
  \end{center}
\end{figure}

\paragraph{Comparison of connectivity spaces with other notions.}

Previous definitions for connectivity spaces have been proposed and studied in~\cite{borger-cspace,serra,connectology,dugowson} among others; see~\cite{stadler} for a general survey.
However, our specific collection of axioms seems original: it is weaker than that of a ``connectology'' as defined in~\cite{connectology} (our third axiom~\ref{ax:complement-connected} is weaker than their axiom~(iv), and our fourth axiom~\ref{ax:button-zip} is weaker than their axiom~(iii)), but stronger than that of an ``integral connectivity space'' as defined in \cite{stadler} (where our third and fourth axioms are not mentioned).
The reason for our choices is to ensure that topological spaces equipped with their (big-)path-connected subsets are connectivity spaces, while retaining the ability to establish the key properties of flimsy spaces (see also \Cref{rmk:rational_circle}).

\subsection{Outline}

In \Cref{sn:prelim}, we recall definitions of standard order-theoretic structures, the connections between them, and some of their elementary properties related to flimsiness; we also define big circles and big-path-connected spaces.
In \Cref{sn:props-2f}, we introduce our axiomatic notion of connectivity spaces, we study $2$-flimsy connectivity spaces, and we show that there are no $n$-flimsy connectivity spaces when $n \geq 3$ (\Cref{thm:no-3f-cspace}).
In \Cref{sn:2f-top}, we specialize the results of \Cref{sn:props-2f} to the usual notion of connectedness to prove \Cref{thm:intro-2f-top}.
In \Cref{sn:2f-big-path} and \Cref{sn:2f-path}, we similarly deal with big-path-connectedness and path-connectedness, and we prove \Cref{thm:intro-2f-big_path,thm:intro-2f-path,thm:hausdorff-2pf-is-circle}.
In \Cref{sn:equi}, we unify various ``circle-like structures'' by explaining how $2$-flimsy connectivity spaces can be equivalently described through the lens of topology or of order theory, and we prove \Cref{thm:intro-classification}.
Finally, in \Cref{sn:2f-path-not-big-path}, we construct a $2$-path-flimsy space which is not $2$-big-path-flimsy.

\subsection{Acknowledgements}

The first author was supported by the Natural Sciences and Engineering
Research Council of Canada (NSERC) via a Banting postdoctoral fellowship [BPF-198443].
The second author was supported by the Deutsche Forschungsgemeinschaft (DFG, German Research Foundation) --- Project-ID 491392403 --- TRR 358.
The authors are grateful to Fabian Gundlach, Raphaël Ruimy, and Rolf Suabedissen for helpful discussions.

\section{Preliminaries}
\label{sn:prelim}

\subsection{Linear continua, big intervals, big paths}
\label{sn:prelim_linear}

A (strictly) ordered set $(X,<)$ is a \emph{linear order} if the strict order~$<$ is total, \emph{dense} if $\card X \geq 2$ and if whenever $x<y$ there exists $z \in X$ with $x < z < y$, and \emph{order-complete} if any non-empty subset~$S \subseteq X$ with an upper bound
admits a supremum.
We say that a linear order~$(X,<)$ is \emph{without ends} if~$X$ has neither a minimum nor a maximum, and \emph{with ends} if~$X$ has both a minimum and a maximum.
The \emph{open intervals} of~$X$ are the sets of the following forms: $(x,y) := \suchthat{z \in X}{x<z<y}$, $(-\infty, y) :=  \suchthat{z \in X}{z<y}$, $(x,+\infty) := \suchthat{z \in X}{x<z}$, $(-\infty,+\infty) := X$.
The \emph{order topology of~$(X,<)$} is the topology on~$X$ generated by the open intervals.

\begin{definition}
  \label{def:lin-cont}
  A topological space~$X$ is a \emph{linear continuum} if its topology is the order topology associated with some dense order-complete linear order on~$X$.
  A \emph{big interval} is a linear continuum associated with a dense order-complete linear order with ends.
\end{definition}

Big intervals generalize the standard interval~$[0,1] \subseteq \R$.
As such, big intervals induce a generalized notion of path-connectedness as follows:

\begin{definition}
  \label{defn:conticonn}
  Let~$X$ be a topological space.
  For any points~$x,y \in X$, a \emph{big path} (on~$X$) from~$x$ to~$y$ is a continuous map $\gamma \colon L \to X$, where~$L$ is a big interval, mapping the minimum and the maximum of~$L$ to~$x$ and~$y$, respectively.
  The space~$X$ is \emph{big-path-connected} if for any two points~$x,y \in X$, there is a big path from~$x$ to~$y$.
  (Note that~$\emptyset$ is big-path-connected.)
\end{definition}

Let us briefly discuss some properties of big-path-connectedness (see also \Cref{sn:2f-big-path}).
The notion of two points being joinable by a big path defines an equivalence relation on points of~$X$---transitivity follows from the possibility of concatenating two big intervals~$L_1$ and~$L_2$ by identifying the maximum of~$L_1$ with the minimum of~$L_2$.
Moreover, for any continuous map $f \colon X \to Y$ and any big path $\gamma \colon L \to X$ on~$X$, the composition $f \circ \gamma \colon L \to Y$ is a big path on~$Y$.
It follows that the image of a big-path-connected space under a continuous map is also big-path-connected.

The following proposition gathers some basic properties of linear continua.

\begin{proposition}
  \label{prop:continuum-facts}
  Let~$X$ be the linear continuum associated to some dense order-complete linear order~$<$ on~$X$.
  Then, the following properties hold:
  \begin{enumroman}
    \item
      \label{item:continuum-hausdorff}
      The space $X$ is Hausdorff.
    \item
      \label{item:continuum-connected-subsets}
      The connected subsets of~$X$ are exactly the \emph{intervals}\footnote{
        The \emph{intervals} are the subsets of one of the forms $(a,b)$, $[a,b)$, $(a,b]$, $[a,b]$ where $a,b \in X \cup \{\pm\infty\}$.
      } of $(X,<)$.
      In particular, $X$ is connected.
    \item
      \label{item:continuum-connected-open-subsets}
      The open connected subsets of~$X$ are exactly the open intervals of $(X,<)$.
    \item
      \label{item:continuum-compact-subsets}
      The compact connected subsets of~$X$ are exactly the \emph{closed intervals}, i.e., the subsets of the form $[a,b] := \suchthat{z \in X}{a \leq z \leq b}$ for some $a,b \in X$.
      In particular, $X$ is compact if and only if $(X, <)$ has both a minimum and a maximum, i.e., if~$X$ is a big interval.
  \end{enumroman}
\end{proposition}

\begin{proof}
  These facts are classical; cf.~\cite[Theorem~24.1, Theorem~27.1]{munkres}.
  The proofs are straightforward adaptations of the corresponding proofs in the classical special case~$X = \R$.
\end{proof}

These facts allow us to relate big-path-connectedness and connectedness:

\begin{corollary}
  \label{cor:bigpathconn-is-conn}
  Any big-path-connected space is connected, and any linear continuum is big-path-connected.
  In a linear continuum, a subset is big-path-connected if and only if it is connected.
\end{corollary}

\begin{proof}
  By \iref{prop:continuum-facts}{item:continuum-connected-subsets}, any big interval is connected, so the image of any big path is connected as the continuous image of a big interval.
  Therefore, any big-path-connected space is connected: it is the union of the images of all the big paths starting from a common (arbitrary) basepoint.

  Now, consider a linear continuum $X$.
  For any $a,b \in X$ with $a < b$, the interval~$[a,b]$ is a big interval, so the identity map $[a,b] \to [a,b]$ is a big path joining~$a$ to~$b$.
  Thus, every linear continuum is big-path-connected.

  The last point follows from the first two: we have shown that all big-path-connected subsets are connected, and by \iref{prop:continuum-facts}{item:continuum-connected-subsets}, any connected subset of a linear continuum is an interval; hence, it is itself a linear continuum, so it is big-path-connected.
\end{proof}

Finally, from the description of the (big-path-)connected subsets of linear continua (\iref{prop:continuum-facts}{item:continuum-connected-subsets} and \Cref{cor:bigpathconn-is-conn}), we easily deduce the following:

\begin{corollary}
  \label{cor:lin-cont-1f}
  Any linear continuum without ends is $1$-flimsy and $1$-big-path-flimsy.
\end{corollary}

\subsection{Big circles}

In the same way that big intervals generalize the segment~$[0,1]$, we can generalize the standard circle by identifying the two ends of a big interval:

\begin{definition}
  \label{def:big-circle}
  A \emph{big circle} is a topological space obtained by identifying the minimum and the maximum of a big interval, equipped with the quotient topology.
\end{definition}

Big circles are the flagship examples of $2$-(big-path-)flimsy spaces:

\begin{proposition}
  \label{prop:big-circle-2f}
  Any big circle is Hausdorff, compact, $2$-flimsy and $2$-big-path-flimsy.
\end{proposition}

\begin{proof}
  Let $X$ be a big circle.
  By definition, there exists a dense order-complete linear order $(\tilde X, <)$ with ends (denoted $\eel=\min \tilde{X}$ and $\eer= \max \tilde{X}$) such that $X$ is the quotient of $\tilde X$ identifying~$\eel$ and~$\eer$ (we denote by~$\re$ their common image in $X$).
  In particular, since~$\tilde X$ is compact, connected, and big-path-connected by~\Cref{prop:continuum-facts} and \Cref{cor:bigpathconn-is-conn}, so is its image~$X$ under the continuous surjection $\pi \colon \tilde X \twoheadrightarrow X$.
  Similarly, the images of the intervals $(\eel,\eer)$, $[\eel,x)$, and $(x,\eer]$ under~$\pi$ are connected and big-path-connected for all $x\in (\eel,\eer)$.
  Thus, the sets $X\setminus\{\re\}=\pi\bigl((\eel,\eer)\bigr)$ and $X\setminus\{\pi(x)\}=\pi\bigl([\eel,x)\bigr)\cup \pi\bigl((x,\eer]\bigr)$ are connected and big-path-connected.

  Now, observe that if an open subset $\tilde{U}$ of $\tilde{X}$ either avoids or contains $\{\eel,\eer\}$, then $\pi^{-1}(\pi(\tilde{U}))=\tilde{U}$, so $\pi(\tilde{U})$ is open in~$X$.
  This observation implies that $X$ inherits the Hausdorff property from~$\tilde{X}$.
  Moreover, for any $x,y\in \tilde{X}$ with $\eel<x<y<\eer$, it implies that the subsets $\pi\bigl((x,y)\bigr)$ and $\pi\bigl((-\infty,x)\cup (y,+\infty)\bigr)$, and $\pi\bigl((\eel,x)\bigr)$ and $\pi\bigl((x,\eer)\bigr)$ are all open in~$X$.
  When $\pi(x)\neq\pi(y)$, it follows that $X\setminus\{\pi(x),\pi(y)\}$ is not connected as the disjoint union of two non-empty open subsets.
  It is not big-path-connected either, thanks to \Cref{cor:bigpathconn-is-conn}.
\end{proof}

\subsection{Cyclic orders, separation relations}

The main results of this article deal with relating $2$-flimsy spaces to certain order-theoretic structures.
Since a $2$-flimsy space (just like the ordinary circle) has neither a canonical distinguished ``point at infinity'' nor a canonical ``sense'' (or orientation), it is more natural to use a notion of order that is less conventional than linear orders, namely \emph{separation relations}.
In this subsection, we recall known facts about these relations and explain how they are related to linear orders.

\paragraph{Cyclic orders.}

Let~$\tilde X$ be a linear order with ends (denoted $\eel$ and $\eer$), and let~$X$ be the quotient set obtained by identifying $\eel$ and $\eer$ into a single point, denoted by $\re$.
Then, $X$ is equipped with a \emph{cyclic order}~$[-,-,-]$, induced by the order of~$\tilde X$: the ternary relation $[x,y,z]$ must be understood as meaning ``when rotating positively around $X$ starting from $x$, we encounter $y$ before $z$''.
Concretely, if $x,y,z$ are distinct points of $X \setminus \{\re\}$, we have
\begin{equation}
  \label{linear-to-cyclic}
  [x,y,z]
  \quad\textnormal{when}\quad
  \bigl(
    x<y<z
  \bigr)
  \textnormal{ or }
  \bigl(
    z<x<y
  \bigr)
  \textnormal{ or }
  \bigl(
    y<z<x
  \bigr).
\end{equation}
\begin{figure}[H]
  \begin{center}
    \begin{tabular}{ccc}
      \begin{tikzpicture}
        \def\shift{5}

        \def\radius{1}
        \def\anglex{90}
        \def\angley{210}
        \def\anglez{-30}

        \newcommand{\printcircle}[6]{
          \begin{scope}[xshift=#1*\shift cm]
            \begin{scope}[yshift=2cm]
              \draw[thick] (-1,0) -- (1,0);
              \node[anchor=south] at (-.75,0) {#4};
              \fill (-.75,0) circle (1.5pt);
              \node[anchor=south] at (0,0) {#5};
              \fill (0,0) circle (1.5pt);
              \node[anchor=south] at (.75,0) {#6};
              \fill (.75,0) circle (1.5pt);
            \end{scope}

            \def\angleinf{#2}
            \path (\anglex:\radius) coordinate (x);
            \path (\angley:\radius) coordinate (y);
            \path (\anglez:\radius) coordinate (z);
            \path (\angleinf:\radius) coordinate (inf);

            \draw[thick, color=black] (x) arc (\anglex:\anglex+360:\radius);

            \node[anchor=south] at (x) {$x$};
            \fill (x) circle (1.5pt);
            \node[anchor=east] at (y) {$y$};
            \fill (y) circle (1.5pt);
            \node[anchor=west] at (z) {$z$};
            \fill (z) circle (1.5pt);
            \node[anchor=#3,red] at (inf) {$\re$};
            \fill[red] (inf) circle (1.5pt);
          \end{scope}
        }
        
        \printcircle{0}{30}{west}{$x$}{$y$}{$z$}
        \printcircle{1}{-90}{north}{$z$}{$x$}{$y$}
        \printcircle{2}{150}{east}{$y$}{$z$}{$x$}
      \end{tikzpicture}
    \end{tabular}
  \end{center}
  \caption{The three possible meanings of $[x,y,z]$ when $x,y,z \in X\setminus\{\re\}$.}
\end{figure}
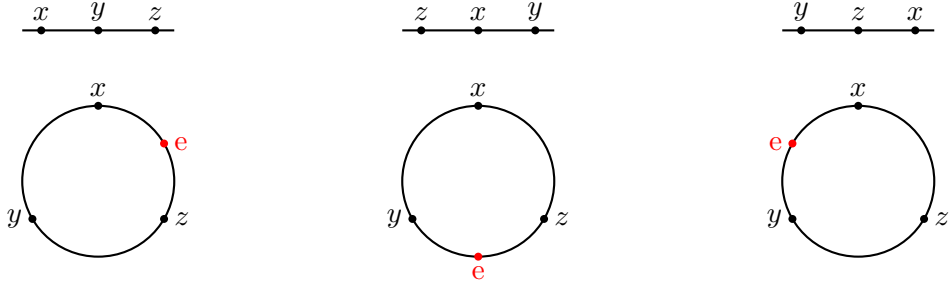
\noindent
The definition when one of $x$, $y$, or $z$ is $\re$ is analogous: the three relations $[x, y, \re]$, $[y, \re, x]$ and $[\re, x, y]$ are all equivalent to $x<y$.
The notion of cyclic orders can be axiomatized without reference to linear orders (cf.~\cite[\nopp 3.]{huntington}), and conversely any cyclic order induces a linear order by picking an arbitrary point $\re$ and by cutting at that point.
More precisely:

\begin{theorem}[{{\cite[Paragraphs~1.2 and~2.1]{huntington}}}]
  \label{thm:linord-cycord}
  Any (total) cyclic order is induced from a linear order with ends using the procedure described in \Cref{linear-to-cyclic}.
  Moreover, two linear orders with ends induce the same cyclic order if and only if they are cyclic rearrangements of each other.%
  \footnote{
    Let~$X$ be a linear order with minimum $y^+$ and maximum $y^-$.
    A \emph{cyclic rearrangement of~$X$} is a linear order of the form $\{x^+\} \sqcup (x, y^-) \sqcup \{y\} \sqcup (y^+, x) \sqcup \{x^-\}$ for some point~$x \in X \setminus \{y^+, y^-\}$ (the order on the disjoint union is the concatenation of the respective orders, where summands are read from left to right)---the two ends~$y^+$ and~$y^-$ have been merged into a single point~$y$, and the point~$x$ has been split into two ends~$x^+$ and~$x^-$.
  }
\end{theorem}

\paragraph{Separation relations.}

Cyclic orders come in pairs: if $[-,-,-]$ is a cyclic order, the \emph{opposite cyclic order}~$[-,-,-]^\textnormal{op}$ is defined by the equivalence $[x,y,z]^\textnormal{op} \Leftrightarrow [z,y,x]$.
In order to specify a cyclic order without choosing a sense (that is, a pair of two opposite cyclic orders), we use the notion of \emph{separation relations} (also called \emph{point-pair separations}).
We refer to \cite[\nopp 4.]{huntington} for a reference on these relations.
We interpret separation relations in terms of ``intersecting chords'', as this will be a more natural way to think about them when relating them to $2$-flimsy spaces---consequently, the notation ``$ABCD$'' in \cite{huntington} corresponds to $\{A,C\} \bot \{B,D\}$ in our framework.

\begin{definition}
  Let $X$ be a set.
  A \emph{chord} $\{a,b\}$ is an unordered pair of distinct points $a,b \in X$.
  We denote by $\Cho(X)$ the set of all chords of~$X$.
\end{definition}

\begin{definition}
  \label{def:sep_rel}
  A \emph{separation relation} on $X$ is a binary relation $\bot$ on $\Cho(X)$ such that:
  \begin{enumerate}[label=(S\arabic*)]
    \item
      \label{ax:separation-symmetric}
      $\bot$ is a symmetric relation: $\{a,b\} \bot \{c,d\} \Leftrightarrow \{c,d\} \bot \{a,b\}$.
    \item
      \label{ax:separation-strict}
      For all $a \in X$ and $b,c \in X \setminus \{a\}$, we have $\{a,b\} \nbot \{a,c\}$.
    \item
      \label{ax:separation-total}
      If $a,b,c,d$ are four distinct points of~$X$, exactly one of the following holds:
      \[
        \{a,b\} \bot \{c,d\},
        \qquad\qquad
        \{a,c\} \bot \{b,d\},
        \qquad\qquad
        \{a,d\} \bot \{b,c\}.
      \]
      For example, in the drawings below, only the chords $\{a,c\}$ and $\{b,d\}$ intersect.
      
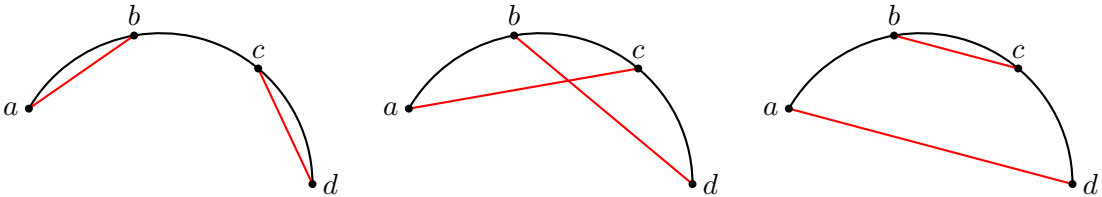
\begin{center}
      \begin{tikzpicture}
        \def\radius{2}
        \def\shift{5}
        \def\angleA{150}
        \def\angleB{100}
        \def\angleC{50}
        \def\angleD{0}

        \begin{scope}
            \draw[thick] (\angleA:\radius) arc (\angleA:\angleD:\radius);
            
            \path (\angleA:\radius) coordinate (a1);
            \path (\angleB:\radius) coordinate (b1);
            \path (\angleC:\radius) coordinate (c1);
            \path (\angleD:\radius) coordinate (d1);

            \draw[red, thick] (a1) -- (b1);
            \draw[red, thick] (c1) -- (d1);

            \foreach \pt in {a1,b1,c1,d1} {
                \fill (\pt) circle (1.5pt);
            }
            \node[anchor=east] at (a1) {$a$};
            \node[anchor=south] at (b1) {$b$};
            \node[anchor=south] at (c1) {$c$};
            \node[anchor=west] at (d1) {$d$};
        \end{scope}

        \begin{scope}[xshift=\shift cm]
            \draw[thick] (\angleA:\radius) arc (\angleA:\angleD:\radius);

            \path (\angleA:\radius) coordinate (a2);
            \path (\angleB:\radius) coordinate (b2);
            \path (\angleC:\radius) coordinate (c2);
            \path (\angleD:\radius) coordinate (d2);

            \draw[red, thick] (a2) -- (c2);
            \draw[red, thick] (b2) -- (d2);

            \foreach \pt in {a2,b2,c2,d2} {
                \fill (\pt) circle (1.5pt);
            }
            \node[anchor=east] at (a2) {$a$};
            \node[anchor=south] at (b2) {$b$};
            \node[anchor=south] at (c2) {$c$};
            \node[anchor=west] at (d2) {$d$};

        \end{scope}

        \begin{scope}[xshift=2*\shift cm]
            \draw[thick] (\angleA:\radius) arc (\angleA:\angleD:\radius);

            \path (\angleA:\radius) coordinate (a3);
            \path (\angleB:\radius) coordinate (b3);
            \path (\angleC:\radius) coordinate (c3);
            \path (\angleD:\radius) coordinate (d3);

            \draw[red, thick] (a3) -- (d3);
            \draw[red, thick] (b3) -- (c3);
            \foreach \pt in {a3,b3,c3,d3} {
                \fill (\pt) circle (1.5pt);
            }
            \node[anchor=east] at (a3) {$a$};
            \node[anchor=south] at (b3) {$b$};
            \node[anchor=south] at (c3) {$c$};
            \node[anchor=west] at (d3) {$d$};

        \end{scope}
      \end{tikzpicture}
      \end{center}
    \item
      \label{ax:separation-transitive}
      For any five distinct points $a,b,c,d,e$ of~$X$, either none or exactly two of the following hold:
      \[
        \{a,b\} \bot \{c,d\},
        \qquad
        \{a,b\} \bot \{d,e\},
        \qquad
        \{a,b\} \bot \{c,e\}.
      \]
      Below, we depict on the left a situation where none of the properties hold, and on the right a situation where both $\{a,b\}\bot\{c,e\}$ and $\{a,b\}\bot\{d,e\}$ hold.
      
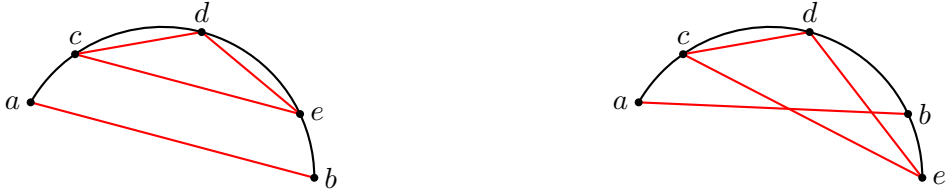
\begin{center}
      \begin{tikzpicture}
        \def\radius{2}
        \def\shift{8}
        \def\angleA{150}
        \def\angleB{0}
        \def\angleC{125}
        \def\angleD{75}
        \def\angleE{25}

        \begin{scope}
          \draw[thick] (\angleA:\radius) arc (\angleA:\angleB:\radius);
          
          \path (\angleA:\radius) coordinate (a1);
          \path (\angleB:\radius) coordinate (b1);
          \path (\angleC:\radius) coordinate (c1);
          \path (\angleD:\radius) coordinate (d1);
          \path (\angleE:\radius) coordinate (e1);

          \draw[red, thick] (a1) -- (b1);
          \draw[red, thick] (c1) -- (d1);
          \draw[red, thick] (d1) -- (e1);
          \draw[red, thick] (c1) -- (e1);

          \foreach \pt in {a1,b1,c1,d1,e1} {
              \fill (\pt) circle (1.5pt);
          }
          \node[anchor=east] at (a1) {$a$};
          \node[anchor=west] at (b1) {$b$};
          \node[anchor=south] at (c1) {$c$};
          \node[anchor=south] at (d1) {$d$};
          \node[anchor=west] at (e1) {$e$};
        \end{scope}

        \begin{scope}[xshift=\shift cm]
          \draw[thick] (\angleA:\radius) arc (\angleA:\angleB:\radius);
          
          \path (\angleA:\radius) coordinate (a1);
          \path (\angleB:\radius) coordinate (b1);
          \path (\angleC:\radius) coordinate (c1);
          \path (\angleD:\radius) coordinate (d1);
          \path (\angleE:\radius) coordinate (e1);

          \draw[red, thick] (a1) -- (e1);
          \draw[red, thick] (c1) -- (d1);
          \draw[red, thick] (d1) -- (b1);
          \draw[red, thick] (c1) -- (b1);

          \foreach \pt in {a1,b1,c1,d1,e1} {
              \fill (\pt) circle (1.5pt);
          }
          \node[anchor=east] at (a1) {$a$};
          \node[anchor=west] at (e1) {$b$};
          \node[anchor=south] at (c1) {$c$};
          \node[anchor=south] at (d1) {$d$};
          \node[anchor=west] at (b1) {$e$};
        \end{scope}
      \end{tikzpicture}
      \end{center}
  \end{enumerate}
\end{definition}

Any cyclic order $[-,-,-]$ on a set~$X$ induces a natural separation relation~$\bot$ on~$X$, defined as follows (cf.~\cite[Paragraph~1.5]{huntington}):
\begin{equation}
  \label{cyclic-to-seprel}
  \{a,b\} \bot \{c,d\}
  \quad\textnormal{when}\quad
  \Big(
    [a,c,b] \textnormal{ and } [b,d,a]
  \Big)
  \textnormal{ or }
  \Big(
    [a,d,b] \textnormal{ and } [b,c,a]
  \Big).
\end{equation}

\begin{theorem}[{{\cite[Paragraph~4.2]{huntington}}}]
  \label{thm:seprel-cycord}
  Any separation relation is induced from a cyclic order using the procedure described in \Cref{cyclic-to-seprel}.
  Moreover, two cyclic orders induce the same separation relation if and only if they are equal or opposite.
\end{theorem}

\subsection{Intervals, denseness, order-completeness of separation relations.}

Let $(X, \bot)$ be a set equipped with a separation relation.

\begin{definition}
  \label{def:open-interval-seprel}
  An \emph{open interval} of~$(X,\bot)$ is a set of one of the two forms
  \[
    I_{\not \ni c}(a,b)
    :=
    \suchthat{
      d \in X\setminus \{a,b,c\}
    }{
      \{a,b\} \bot \{c,d\}
    }\quad
    \text{ or }
    \quad 
    I_{\ni c}(a,b)
    :=
    \Bigl(X\setminus \{a,b\}\Bigr)
    \setminus
    I_{\not\ni c}(a,b),
  \]
  where $a,b,c \in X$ are distinct.
  We call $I_{\not \ni c}(a,b)$ and $I_{\ni c}(a,b)$ the \emph{open intervals between $a$ and $b$}.
\end{definition}

For distinct $a,b,c\in X$, note that $c\notin I_{\not\ni c}(a,b)$ and $c\in I_{\ni c}(a,b)$.
Moreover, \axref{ax:separation-transitive} implies that for any $d\in I_{\not\ni c}(a,b)$ we have $I_{\ni c}(a,b)=I_{\not\ni d}(a,b)$.

\begin{definition}
  \label{def:order-top_seprel}
  The \emph{order topology of~$(X,\bot)$} is the topology on~$X$ generated by the open intervals.
\end{definition}

\begin{definition}
\label{def:dense-complete_seprel}
  We say that $(X, \bot)$ is \emph{dense} if $\card X \geq 2$ and if for any $\{a,b\} \in \Cho(X)$ there is a $\{c,d\} \in \Cho(X)$ such that $\{a,b\} \bot \{c,d\}$.
  By \axref{ax:separation-transitive}, assuming that $\card X \geq 3$ (which is a necessary condition), this is equivalent to all open intervals being non-empty.
  
  We say that $(X,\bot)$ is \emph{order-complete} if, for any open interval $I$, the union of any non-empty chain of open intervals contained in $I$ is an open interval.
\end{definition}

Observing what intervals become under the equivalences of \Cref{thm:linord-cycord,thm:seprel-cycord}, we obtain:

\begin{proposition}
  \label{prop:intervals_from_linear_to_separation}
  Assume that $\card X \geq 3$.
  Let $(\tilde X, <)$ be a linear order with ends inducing the separation relation $\bot$ on~$X$ via the equivalences of \Cref{thm:linord-cycord,thm:seprel-cycord}, and let $\pi \colon \tilde X \twoheadrightarrow X$ be the natural surjection identifying both ends of~$\tilde X$.
  Then, for any $a,b\in\tilde{X}$ such that $a<b$ and $\pi(a)\neq\pi(b)$, the open intervals of $(X,\bot)$ between $\pi(a)$ and $\pi(b)$ are exactly $\pi\bigl((a,b)\bigr)$ and $X\setminus \pi\bigl([a,b]\bigr)$.
\end{proposition}

\begin{corollary}
  \label{prop:from_linear_to_separation_intervals}
  Assume that $\card{X}\geq 3$ and that $(\tilde X, <)$ is a linear order with ends inducing the separation relation~$\bot$ on~$X$ via the equivalences of \Cref{thm:linord-cycord,thm:seprel-cycord}.
  Then:
  \begin{itemize}
    \item
      $(X,\bot)$ is dense if and only if the linear order $(\tilde X, <)$ is dense.
    \item
      $(X,\bot)$ is order-complete if and only if the linear order $(\tilde X, <)$ is order-complete.
  \end{itemize}
\end{corollary}

\begin{proof}
    The first point readily follows from \Cref{prop:intervals_from_linear_to_separation}.
    For the second point, observe that for any subset~$S$ of $\tilde{X}$, for any $s_0\in S$, and for any $s_+\in\tilde{X}$ with $s_0\leq s_+$, the set~$S$ admits~$s_+$ as its supremum if and only if $\bigcup_{s\in S}[s_0,s)=[s_0,s_+)$.
    Likewise, for $s_-\in\tilde{X}$ with $s_-\leq s_0$, $S$ admits~$s_-$ as its infimum (namely, $s_-$ is the supremum of the set of lower bounds of $S$) if and only if $\bigcup_{s\in S}(s,s_0]=(s_-,s_0]$.
    Then, the second point follows easily from \Cref{prop:intervals_from_linear_to_separation}.
\end{proof}

We recall a standard result about the comparison of Hausdorff and compact topologies.
This fact is used several times throughout the paper.

\begin{lemma}
  \label{lattice_Hausdorff-compact}
  Let $X$ be a set and let $\tau_1,\tau_2$ be two topologies on $X$ such that $(X,\tau_1)$ is Hausdorff and~$(X,\tau_2)$ is compact.
  If $\tau_1\subseteq \tau_2$, then $\tau_1=\tau_2$.
\end{lemma}

\begin{proof}
  Since $\tau_1\subseteq \tau_2$, the identity map of~$X$ is a continuous bijection from the compact space~$(X, \tau_2)$ to the Hausdorff space~$(X, \tau_1)$.
  Hence, it is a homeomorphism by \cite[Theorem~17.14]{willard}, which means that $\tau_1 = \tau_2$.
\end{proof}

Using this fact, we conclude this section with a characterization of big circles:

\begin{theorem}
  \label{cor:seprel-big-circ}
  A topological space $(X,\tau)$ is a big circle (\Cref{def:big-circle}) if and only if there is a dense and order-complete separation relation~$\bot$ on~$X$ such that~$\tau$ is the order topology of~$(X,\bot)$.
\end{theorem}

\begin{proof}
  Using the definition of big circles on the one hand, and the equivalences of \Cref{thm:linord-cycord,thm:seprel-cycord} together with \Cref{prop:from_linear_to_separation_intervals} on the other hand, this reduces to the following claim:
  for any dense and order-complete linear order~$(\tilde X,<)$, with ends~$\tilde 0$ and~$\tilde 1$, letting~$X$ be the quotient of~$\tilde X$ obtained by identifying~$\tilde 0$ and~$\tilde 1$ and letting~$\bot$ be the dense order-complete separation relation on~$X$ induced by~$<$ (cf.~\Cref{linear-to-cyclic,cyclic-to-seprel} and \Cref{prop:from_linear_to_separation_intervals}), the order topology~$\tau_1$ of~$(X,\bot)$ (\Cref{def:order-top_seprel}) coincides with the quotient topology~$\tau_2$ on~$X$ induced by the order topology of~$(\tilde X, {<})$.
  We shall prove that $\tau_1 = \tau_2$ by using \Cref{lattice_Hausdorff-compact}.

  First, we know from \Cref{prop:big-circle-2f} that the big circle $(X,\tau_2)$ is compact.

  Second, if $x,y\in X$ are distinct, then by density there exists $\{w,z\}\in\Cho(X)$ such that $\{x,y\}\bot\{w,z\}$, which lets us observe that $x\in I_{\ni x}(w,z)$ and $y\in I_{\not\ni x}(w,z)$ by \axref{ax:separation-symmetric}.
  Since the open intervals $I_{\ni x}(w,z)$ and $I_{\not\ni x}(w,z)$ of $(X,\bot)$ are disjoint, it follows that $(X,\tau_1)$ is Hausdorff.

  By \Cref{lattice_Hausdorff-compact}, the claim then reduces to verifying that $\tau_1 \subseteq \tau_2$, i.e., that all open intervals of~$(X,\bot)$ belong to~$\tau_2$.
  This is a consequence of \Cref{prop:intervals_from_linear_to_separation}.
  Indeed, for any $a,b\in\tilde{X}$ such that $a<b$, $(a,b)$ avoids the ends of $(\tilde{X},<)$ so the inverse image under the natural projection $\pi \colon \tilde X \twoheadrightarrow X$ of $\pi\bigl((a,b)\bigr)$ is~$(a,b)$, which is open.
  Depending on whether $\{a,b\}$ contains one of the ends of $(\tilde{X},<)$, the inverse image of $X\setminus \pi\bigl([a,b]\bigr)$ is either $(-\infty,a)\cup (b,+\infty)$ or $(\eel,a)$ or $(b,\eer)$, which are all open.
\end{proof}

\section{Flimsy connectivity spaces}
\label{sn:props-2f}

In this section, we prove key properties of $2$-flimsy spaces, using our axiomatic notion of connected subsets.
Notably, we show that:
\begin{description}
  \item[(\Cref{only-two-components})]
    The complement of any pair of two distinct points of a $2$-flimsy space has exactly two connected components.
  \item[(\Cref{flimsy-is-t1})]
    $2$-flimsy spaces can be equivalently characterized as those $\rT_1$ spaces in which the complement of any connected subset is connected.
  \item[(\Cref{thm:no-3f-cspace})]
    There are no $n$-flimsy spaces for any integer $n \geq 3$.
  \item[(\Cref{cor:list-of-connected-parts})]
    The connected subsets of a $2$-flimsy space $X$ are exactly the following: $\emptyset$, $X$, singletons and their complements, and subsets of the form $C$, $C \cup \{x\}$, $C \cup \{x,y\}$ for some pair of distinct points $x,y \in X$ and some connected component $C$ of $X \setminus \{x,y\}$.
\end{description}

\subsection{Connectivity spaces}
\label{subsn:c-spaces}

Whether a topological space $X$ is $n$-flimsy only depends on the set of its connected subsets, which is less information than the full topological structure.
In order to prove things in more generality, we reason just in terms of connected subsets---this will allow us to treat the case of connectedness and (big-)path-connectedness in a unified manner.
To this end, we use the following axioms:

\begin{definition}
  \label{def:c-space}
  A \emph{connectivity space} is a pair $(X, \cC)$ where~$X$ is a set and~$\cC$ is a set of parts of~$X$ satisfying the following properties (we say that~$\cC$ is a \emph{connectivity} on~$X$).
  \begin{enumerate}[label=\bfseries(C\arabic*)]
    \item
      \label{ax:singleton-connected}
      For any $x \in X$, we have $\{x\} \in \cC
      $.
    \item
      \label{ax:union-connected}
      For any $C \subseteq \cC$, if $\bigcap C \neq \emptyset$, then $\bigcup C \in \cC$.
      (In particular, $\emptyset \in \cC$.)
    \item
      \label{ax:complement-connected}
      For any $Y \in \cC$ and any $x,y \in Y$, the subset $C := \bigcup \suchthat{Z \in \cC}{y \in Z \textnormal{ and } Z \subseteq Y \setminus \{x\}}$ satisfies $Y \setminus C \in \cC$.
    \item
      \label{ax:button-zip}
      For any non-empty $A,B\in\cC$ such that $A\cup B\in\cC$, at least one of the following holds:
      \begin{enumroman}
        \item
          \label{ax:button-zip-button}
          there exists $x\in A\cup B$ such that $A\cup\{x\}\in\cC$ and $B\cup\{x\}\in \cC$; or
        \item
          \label{ax:button-zip-seam}
          there exists $S\in\cC$ with $S \subseteq A\cup B$ such that $S\cap A \notin \cC$ or $S\cap B\notin \cC$.
      \end{enumroman}
  \end{enumerate}
\end{definition}

We shall see in \Cref{prop:top-is-cspace} (resp.\ in \Cref{prop:continuum-is-cspace}, in \Cref{prop:path-is-cspace}) that any topological space together with the set of its connected subsets (resp.\ its big-path-connected subsets, its path-connected subsets) is a connectivity space.
\Axref{ax:button-zip} is not used in \Cref{subsn:prop-2f-cspace,subsn:no-3f}---in particular, there are no $3$-flimsy connectivity spaces even if one does not require \axref{ax:button-zip}.

Let~$(X, \cC)$ be a connectivity space.
We say that a subset $Y \subseteq X$ is \emph{$\cC$-connected} when $Y \in \cC$.
Any subset $Y \subseteq X$ is equipped with a canonical structure of connectivity subspace, namely $(Y, \, \cC \cap \mathcal P(Y))$.
\Axref{ax:union-connected} implies that~$Y$ can be partitioned into its (well-defined) maximal $\cC$-connected subsets, which we call \emph{the $\cC$-components of~$Y$}.
With that terminology, \axref{ax:complement-connected} says that if $C$ is a $\cC$-component of $Y \setminus \{x\}$, and if~$Y$ is $\cC$-connected, then $Y \setminus C$ is $\cC$-connected.
For any $y \in Y$, we define $C_{\ni y}(Y) := \bigcup \suchthat{ Z \in \cC }{ y \in Z \textnormal{ and } Z \subseteq Y }$, which is the $\cC$-component of~$Y$ containing $y$, and we let $C_{\not\ni y}(Y) := Y \setminus C_{\ni y}(Y)$ be its complement (which is not $\cC$-connected in general).

\begin{lemma}
  \label{lem:compo_subset}
  Let $(X,\mathcal{C})$ be a connectivity space, and consider two subsets $Z \subseteq Y \subseteq X$.
  Any $\cC$-component of~$Y$ that is contained in~$Z$ is a $\cC$-component of~$Z$.
\end{lemma}

\begin{proof}
  This follows directly from the definitions (a $\cC$-component is a maximal $\cC$-connected subset, and any $\cC$-connected subset of $Z$ is in particular a $\cC$-connected subset of $Y$).
\end{proof}

\begin{definition}
  \label{def:c-space-T1}
  A connectivity space $(X, \cC)$ is \emph{$\rT_1$} when for any distinct $x,y \in X$, we have~$\{x,y\} \notin \cC$.
\end{definition}

\begin{definition}
\label{def:c-space_flimsy}
  Let $n \in \N$.
  A connectivity space $(X, \cC)$ is \emph{$n$-flimsy} if $\card X > n$ and if, for any subset~$S \subseteq X$, we have $(X \setminus S) \in \cC$ whenever $\card S < n$ and $(X \setminus S) \notin \cC$ whenever $\card S = n$.
\end{definition}

\subsection{Properties of $2$-flimsy connectivity spaces}
\label{subsn:prop-2f-cspace}

In this subsection, we fix a $2$-flimsy connectivity space~$(X,\cC)$.
For any subset~$S \subseteq X$, the notation~$\comp S$ always refers to the complement~$X \setminus S$ of~$S$ in~$X$.

\begin{proposition}
  \label{only-two-components}
  For any distinct~$x, y \in X$, the subset~$\compb{x,y}$ has exactly two $\cC$-components.
\end{proposition}

\begin{proof}
  Since~$X$ is $2$-flimsy, by definition, $\compb{x,y}$ has at least two $\cC$-components.
  Pick two distinct $\cC$-components~$C_1$ and~$C_2$ of~$X$, two arbitrary points $u_1 \in C_1$ and $u_2 \in C_2$ (in particular, $u_1 \neq u_2$), and let $A := \comp{(\{x,y\}\cup C_1\cup C_2)}$.
  Our goal is to show that~$A$ is empty.
  We reason by contradiction by assuming that~$A$ is non-empty.

  \begin{figure}[H]
    \begin{center}
      \begin{tikzpicture}
        \def\radius{2}
        \def\anglex{90}
        \def\angley{-90}
        \def\angleuone{0}
        \def\angleutwo{180}

        \path (\anglex:\radius) coordinate (x);
        \path (\angley:\radius) coordinate (y);
        \path (\angleuone:\radius) coordinate (u1);
        \path (\angleutwo:\radius) coordinate (u2);

        \draw[thick, color=red] (y) arc (\angley:\anglex:\radius);
        \draw[thick, color=blue] (x) arc (\anglex:\angley+360:\radius);
        \draw[ForestGreen, thick] (x) -- (y);
        \draw[ForestGreen, thick] (x) to[bend right] (y);
        \draw[ForestGreen, thick] (x) to[bend left] (y);

        \node[anchor=south] at (x) {$x$};
        \fill (x) circle (1.5pt);
        \node[anchor=north] at (y) {$y$};
        \fill (y) circle (1.5pt);
        \fill [color=red] (u1) circle (1.5pt);
        \node[anchor=west, color=red] at (u1) {$u_1$};
        \fill [color=blue] (u2) circle (1.5pt);
        \node[anchor=east, color=blue] at (u2) {$u_2$};

        \node[anchor=south, color=red] at (30:\radius+.2) {$C_1$};
        \node[anchor=south, color=blue] at (150:\radius+.2) {$C_2$};
        \node[anchor=east, color=ForestGreen] at (150:.7) {$A$};
      \end{tikzpicture}
    \end{center}
    \caption{A sketch of the situation of the proof of \Cref{only-two-components}.}
  \end{figure}
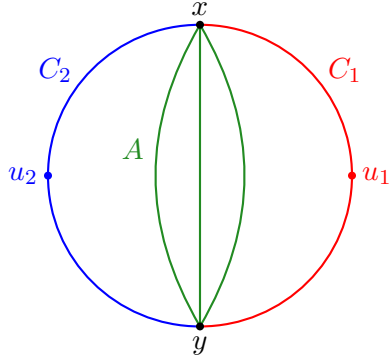

  We are in fact going to show that the subsets~$\compb{u_1} \setminus C_2$ and~$\compb{u_2} \setminus C_1$ are both $\cC$-connected.
  This implies the claim by the following argument: since these two sets both contain $x$, \axref{ax:union-connected} implies that their union $\compb{u_1,u_2}=(\compb{u_1} \setminus C_2) \cup (\compb{u_2} \setminus C_1)$ is $\cC$-connected, contradicting the fact that $X$ is $2$-flimsy.
  Since the two cases are identical by symmetry, we have reduced to showing that the set $\compb{u_1} \setminus C_2 = A \cup \{x,y\} \cup (C_1 \setminus \{u_1\})$ is $\cC$-connected.

  The subset $C_2 \subseteq \compb{u_1,x}$ is $\cC$-connected, so it is contained in a $\cC$-component $D$ of $\compb{u_1,x} = \compb{u_1}\setminus\{x\}$.
  Note that since the space $\compb{u_1}$ is $\cC$-connected by definition of $2$-flimsy connectivity spaces, its subset $\compb{u_1}\setminus D$ is $\cC$-connected by \axref{ax:complement-connected}.

  If $y\notin D$, then $D \subseteq \compb{x,y}$ is $\cC$-connected and contains $C_2$, so $D = C_2$ by maximality of $C_2$.
  Then, the subset $\compb{u_1}\setminus C_2=\compb{u_1}\setminus D$ is $\cC$-connected, completing the proof.

  Further, we assume that $y \in D$.
  By \axref{ax:complement-connected}, the set $D\setminus C_2$ is $\cC$-connected: indeed, $C_2$ is a $\cC$-component of $\compb{x,y}$ by definition, and $D\setminus\{y\}$ is a subset of~$\compb{x,y}$ which contains $C_2$, so $C_2$ is a $\cC$-component of~$D\setminus\{y\}$ by \Cref{lem:compo_subset}.

  We now prove that $A\cup\{y\}$ is $\cC$-connected.
  The space $\compb{x}$ is $\cC$-connected and $C_1$ is a $\cC$-component of $\compb{x,y} =\compb{x}\setminus \{y\}$, so $Y := \compb{x} \setminus C_1 = \{y\} \cup C_2 \cup A$ is $\cC$-connected by \axref{ax:complement-connected}.
  Moreover, $C_2$ is a $\cC$-component of~$Y \setminus \{y\}=C_2\cup A\subseteq \compb{x,y}$ by \Cref{lem:compo_subset}, so $Y \setminus C_2 = A \cup \{y\}$ is $\cC$-connected by \axref{ax:complement-connected}.
  The same arguments show that $A \cup \{x\}$ is $\cC$-connected.
  
  Therefore, $A\cup\{y\}$ is a $\cC$-connected subset of~$\compb{u_1,x}$ which intersects $D$ at $y$, so $A\cup\{y\}\subseteq D$ by maximality of~$D$.
  In fact, it holds that $A \subseteq D\setminus C_2$ because $A$ and $C_2$ are disjoint.
  Hence, we can write $(D\cup\{x\})\setminus C_2=(D\setminus C_2)\cup (A\cup \{x\})$.
  The intersection of the two $\cC$-connected sets~$D\setminus C_2$ and~$A\cup \{x\}$ is non-empty because they both contain~$A$, so their union $(D\cup\{x\})\setminus C_2$ is $\cC$-connected by \axref{ax:union-connected}.
  Using \axref{ax:union-connected} again, $\compb{u_1} \setminus C_2$ is then $\cC$-connected as the union of the two $\cC$-connected sets~$\compb{u_1}\setminus D$ and~$(D\cup\{x\})\setminus C_2$, which both contain~$x$.
  We have proved that $\compb{u_1} \setminus C_2$ is $\cC$-connected, completing the proof.
\end{proof}

An immediate consequence of \Cref{only-two-components} is the following fact:

\begin{corollary}
  \label{cor:comp-not-containing}
  For any three distinct points $x,y,z \in X$, there is a unique $\cC$-component of $\compb{x,y}$ that does not contain~$z$.
  In other words, $C_{\not\ni z}(\compb{x,y})$ is non-empty and $\cC$-connected---we can call it \emph{the} $\cC$-component of~$\compb{x,y}$ not containing~$z$.
\end{corollary}

\begin{lemma}
  \label{Ccupx}
  For any distinct $x,y\in X$ and any $\cC$-component $C$ of~$\compb{x,y}$, the sets $C\cup\{x\}$ and $C \cup \{x,y\}$ are $\cC$-connected.
\end{lemma}

\begin{proof}
  Let $C' := \compb{x,y}\setminus C$.
  \Cref{only-two-components} shows that $C'$ is a $\cC$-component of~$\compb{x,y}=\compb{y}\setminus \{x\}$.
  Since the space $\compb{y}$ is $\cC$-connected, \axref{ax:complement-connected} implies that $\compb{y}\setminus C'=C\cup\{x\}$ is $\cC$-connected.
  The same argument shows that $C \cup \{y\}$ is $\cC$-connected.
  The $\cC$-connectedness of~$C \cup \{x,y\}$ then follows using \axref{ax:union-connected} as $C$ is non-empty.
\end{proof}

\begin{proposition}
  \label{A^c}
  For any $\cC$-connected subset $A \subseteq X$, its complement $\comp{A}$ is also $\cC$-connected.
\end{proposition}

\begin{proof}
  We assume that $\card{A}\geq 2$ and $|\comp{A}|\geq 2$, as the result follows from the definition of $2$-flimsy connectivity spaces and from \axref{ax:singleton-connected} otherwise.
  Fix arbitrary elements $a\in A$ and $\xi \in \comp{A}$.
  We begin by proving the identity
  \begin{equation}
    \label{A^c_tool}
    A
    =
    \bigcap_{x\in \comp{A}\setminus\{\xi\}}
      C_{\ni a}(\compb{\xi, x}).
  \end{equation}
  Indeed, let $B$ be the right-hand side.
  We have $A\subseteq B$ because, for any $x\in \comp{A}\setminus\{\xi\}$, the set~$A$ is a $\cC$-connected subset of~$\compb{\xi,x}$ containing~$a$, hence is contained in $C_{\ni a}(\compb{\xi, x})$.
  Conversely, an element~$x \in \comp A$ either equals $\xi$ (which does not lie in $B$), or belongs to $\comp{A}\setminus\{\xi\}$ and then $x \notin C_{\ni a}(\compb{\xi, x})$, so~$x \notin B$.
  Thus, $\comp{A}\subseteq \comp{B}$, completing the proof of \Cref{A^c_tool}.
  
  Taking complements, \Cref{A^c_tool} implies that
  \[
    \comp{A}=
      \bigcup_{x\in \comp{A}\setminus\{\xi\}}
        C_{\not\ni a}(\compb{\xi, x})
        \cup
        \{\xi,x\}.
  \]
  By \Cref{cor:comp-not-containing} and \Cref{Ccupx}, all the summands in this union are $\cC$-connected.
  Since they all contain~$\xi$, their union~$\comp A$ is $\cC$-connected by \axref{ax:union-connected}.
\end{proof}

This yields a characterization of $2$-flimsy connectivity spaces:

\begin{corollary}
  \label{flimsy-is-t1}
  A connectivity space $(X,\cC)$ with at least $3$ points is $2$-flimsy if and only if it is a $\rT_1$ space whose collection of $\cC$-connected subsets is closed under complement.
\end{corollary}

\begin{proof}
  By \Cref{A^c}, in a $2$-flimsy space, complements of $\cC$-connected subsets are $\cC$-connected, so we can show the equivalence under that hypothesis.
  We can then take complements in the definition of flimsy spaces: the space~$X$ is $2$-flimsy if and only if~$\emptyset$ is $\cC$-connected (which is automatically true by \axref{ax:union-connected}), $\{x\}$ is $\cC$-connected for any $x \in X$ (which is automatically true by \axref{ax:singleton-connected}), and $\{x,y\}$ is not $\cC$-connected for any two distinct $x,y \in X$ (which is \Cref{def:c-space-T1} of $\rT_1$~spaces).
\end{proof}

Note also the following corollary:

\begin{corollary}
  \label{intersection-connected}
  Let $A, B$ be two $\cC$-connected subsets of~$X$.
  If $A \cup B \neq X$, then $A \cap B$ is $\cC$-connected.
\end{corollary}

\begin{proof}
  By \Cref{A^c}, both $\comp{A}$ and $\comp{B}$ are $\cC$-connected.
  Moreover, their intersection is non-empty, so their union $\comp{A}\cup\comp{B}$ is $\cC$-connected by \axref{ax:union-connected}.
  Using \Cref{A^c} again, we find that $A\cap B$ is $\cC$-connected.
\end{proof}

Finally, we prove two lemmas that will be key to the proof of the forthcoming \Cref{thm:no-3f-cspace}:

\begin{lemma}
  \label{if-you-touch-edges-you-are-component}
  Let $x,y$ be distinct points of~$X$, and let $C$ be a $\cC$-connected subset of $\compb{x,y}$ such that $C \cup \{x,y\}$ is $\cC$-connected.
  Then, $C$ is a $\cC$-component of $\compb{x,y}$.
\end{lemma}

\begin{proof}
  Since $C \cup \{x,y\}$ is $\cC$-connected, so is its complement $C'$ by \Cref{A^c}.
  We have $C \sqcup C' = \compb{x,y}$, where~$C$ and~$C'$ are disjoint and are both $\cC$-connected, so $C$ and $C'$ are exactly the two $\cC$-components of~$\compb{x,y}$ (cf.~\Cref{only-two-components}).
\end{proof}

\begin{lemma}
  \label{intersect-intervals}
  Let $a,b,c$ be three distinct points of~$X$.
  We have the identity
  \[
    C_{\ni c}\bigl(\compb{a,b}\bigr)
    \cap
    C_{\ni b}\bigl(\compb{a,c}\bigr)
    =
    C_{\not\ni a}\bigl(\compb{b,c}\bigr).
  \]
\end{lemma}

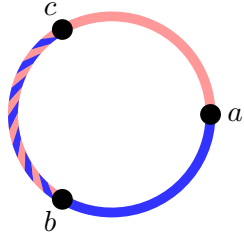
\begin{figure}[H]
  \begin{center}
    \begin{tikzpicture}[scale=1.3]
      \def\rinner{.95}
      \def\router{1.05}

      \def\angA{0}
      \def\angB{240}
      \def\angC{120}

      \fill[red!40] (\angA:\rinner) arc (\angA:\angB:\rinner) -- (\angB:\router) arc (\angB:\angA:\router) -- cycle;
      \fill[blue!80] (\angB:\rinner) arc (\angB:\angA+360:\rinner) -- (\angA+360:\router) arc (\angA+360:\angB:\router) -- cycle;

      \def\nstripes{15}
      \foreach \i in {0,...,\numexpr\nstripes-1\relax} {
        \pgfmathsetmacro{\a}{\angC-5 + \i*(120/\nstripes)}
        \pgfmathsetmacro{\b}{\angC-5 + (\i+1)*(120/\nstripes)}
        \pgfmathparse{int(mod(\i,2)==0 ? 1 : 0)}
        \ifnum\pgfmathresult=1
          \fill[red!40] (\a:\rinner) arc (\a:\b:\rinner) -- (\b+10:\router) arc (\b+10:\a+10:\router) -- cycle;
        \else
          \fill[blue!80] (\a:\rinner) arc (\a:\b:\rinner) -- (\b+10:\router) arc (\b+10:\a+10:\router) -- cycle;
        \fi
      }

      \foreach \ang/\name in {\angA/a,\angB/b,\angC/c} {
        \fill[black] (\ang:1) circle (3pt);
        \node at (\ang:1.25) {$\name$};
      }
    \end{tikzpicture}
    \caption{
      An illustration of \Cref{{intersect-intervals}}.
      The red region (including the striped region) is $C_{\ni c}(\compb{a,b})$, the blue region (including the striped region) is $C_{\ni b}(\compb{a,c})$, and their intersection (the striped region) is indeed $C_{\not\ni a}(\compb{b,c})$.
    }
  \end{center}
\end{figure}

\begin{proof}
  Taking complements (and ignoring the points~$b$ and~$c$ which do not belong to either side), the claim reduces to showing:
  \[
    C_{\not\ni c}(\compb{a,b})
    \cup
    C_{\not\ni b}(\compb{a,c})
    \cup
    \{a\}
    =
    C_{\ni a}(\compb{b,c}).
  \]
  By \Cref{cor:comp-not-containing} and \Cref{Ccupx}, both $C_1 := C_{\not\ni c}(\compb{a,b}) \cup \{a\}$ and $C_2 := C_{\not\ni b}(\compb{a,c}) \cup \{a\}$ are $\cC$-connected, and since they both contain~$a$, their union $C_1 \cup C_2$ (contained in $\compb{b,c}$) is $\cC$-connected by \axref{ax:union-connected}.
  Similarly, \Cref{Ccupx} implies that $C'_1 := C_{\not\ni c}(\compb{a,b}) \cup \{a,b\}$ and $C'_2 := C_{\not\ni b}(\compb{a,c}) \cup \{a,c\}$ are $\cC$-connected, so $C'_1 \cup C'_2 = (C_1 \cup C_2) \cup \{b,c\}$ is $\cC$-connected by \axref{ax:union-connected}.
  By \Cref{if-you-touch-edges-you-are-component}, $C_1 \cup C_2$ is then a $\cC$-component of~$\compb{b,c}$, and since it contains~$a$ this concludes the proof.
\end{proof}

\subsection{There are no $3$-flimsy connectivity spaces}
\label{subsn:no-3f}

We now prove \Cref{thm:no-3f} for general connectivity spaces:

\begin{theorem}
  \label{thm:no-3f-cspace}
  For any $n \geq 3$, there are no $n$-flimsy connectivity spaces.
\end{theorem}

\begin{proof}
  Since removing any point from an $n$-flimsy space (for $n \geq 3$) yields an $(n-1)$-flimsy space, it suffices to prove the claim for $n = 3$.
  
  Let $(X,\cC)$ be a $3$-flimsy connectivity space, and let $x,y,t,s$ be four distinct points of~$X$.
  Let $D := C_{\ni s}(\compb{x,y,t}) \cap C_{\ni t}(\compb{x,y,s})$.
  Applying \Cref{intersect-intervals} in the two $2$-flimsy spaces~$\compb{x}$ and~$\compb{y}$ shows respectively that $D = C_{\not\ni y}(\compb{x,t,s})$ and $D = C_{\not\ni x}(\compb{y,t,s})$.
  We have shown that~$D$ is a $\cC$-component of $\compb{x,t,s} = \compb{t} \setminus \{x,s\}$, so $D\cup\{x\}$ is $\cC$-connected by \Cref{Ccupx} applied in the $2$-flimsy space $\compb t$.
  But this contradicts the maximality of~$D$ among the $\cC$-connected subsets of~$\compb{y,t,s}$.
\end{proof}

\subsection{$\cC$-connected subsets of $2$-flimsy connectivity spaces}

In this section, we characterize all the $\cC$-connected subsets of an arbitrary $2$-flimsy connectivity space~$(X,\cC)$ in terms of the $\cC$-components of the complements of pairs of distinct points of~$X$.

\begin{proposition}
  \label{prop:list-of-connected-parts}
  Let $(X,\cC)$ be a $2$-flimsy connectivity space, and let $C\in\cC$.
  If $\card{C}\geq 2$ and $|\comp{C}|\geq 2$, then there exist distinct $x,y\in X$ such that $C\cap \compb{x,y}$ is a $\cC$-component of $\compb{x,y}$.
\end{proposition}

\begin{proof}
  If there are two distinct $x,y\in \comp{C}$ such that $C\cup\{x,y\}$ is $\cC$-connected, then the set $C\cap \compb{x,y} = C$ is a $\cC$-component of $\compb{x,y}$ by \Cref{if-you-touch-edges-you-are-component}.
  Similarly, if there are two distinct $x,y\in C$ such that $\comp{C}\cup\{x,y\}$ is $\cC$-connected, then $\comp{C}$ is a $\cC$-component of $\compb{x,y}$, and then so is $C\cap\compb{x,y}$ by \Cref{cor:comp-not-containing}.
  Thus, we can assume in the following that there is at most one $x'\in \comp{C}$ such that~$C\cup\{x'\}$ is $\cC$-connected, and at most one $y'\in C$ such that $\comp{C}\cup\{y'\}$ is $\cC$-connected.
  Since~$\card{C}\geq 2$ and $|\comp{C}|\geq 2$, we can then pick elements $u\in C$ and $v\in\comp{C}$ such that $C\cup\{v\}$ and $\comp{C}\cup\{u\}$ are not $\cC$-connected.
  We denote by $D_x$ and $D_y$ the two $\cC$-components of $\compb{u,v}$.

  The set $D_x\cup\{u\}$ is $\cC$-connected by \Cref{Ccupx}.
  Moreover, neither $C$ nor $D_x\cup\{u\}$ contains~$v$, so \Cref{intersection-connected} implies that $A:= C\cap (D_x\cup\{u\})=(C\cap D_x)\cup\{u\}$ is $\cC$-connected.
  Symmetrically, we get that $B:= \comp{C}\cap (D_x\cup\{v\})=(\comp{C}\cap D_x)\cup\{v\}$ is $\cC$-connected.
  Furthermore, observe that $A\cup B=D_x\cup\{u,v\}$ is $\cC$-connected by \Cref{Ccupx}.

  Now, we apply \axref{ax:button-zip} to~$A$ and~$B$.
  In the conclusion of the axiom, case~\ref{ax:button-zip-seam} is impossible: for any $\cC$-connected subset $S \subseteq D_x\cup \{u,v\}$, neither $A$ nor $B$ nor $S$ meets $D_y$ (which is non-empty as a $\cC$-component of $\compb{u,v}$), so $S\cap A$ and $S\cap B$ are $\cC$-connected by \Cref{intersection-connected}.
  We are therefore in case~\ref{ax:button-zip-button}: there exists an element $x \in D_x\cup \{u,v\}$ such that $(C\cap D_x)\cup\{u,x\}$ and $(\comp{C}\cap D_x)\cup \{v,x\}$ are $\cC$-connected.
  By symmetry, there also exists $y \in D_y\cup \{u,v\}$ such that $(C\cap D_y)\cup\{u, y\}$ and $(\comp{C}\cap D_y)\cup \{v,y\}$ are $\cC$-connected.
  Taking the respective unions, it follows from \axref{ax:union-connected} that the sets $(C\cap\compb{u,v})\cup\{u,x,y\}=C\cup\{x,y\}$ and $(\comp{C}\cap\compb{u,v})\cup\{v,x,y\}=\comp{C}\cup\{x,y\}$ are $\cC$-connected.

  We now show that $x \neq y$.
  Indeed, if $x=y$, then $x$ would belong to both $D_x\cup\{u,v\}$ and $D_y\cup\{u,v\}$, so it would have to be equal to either~$u$ or~$v$.
  One of the sets $C\cup\{v\}$ or $\comp{C}\cup\{u\}$ would then be $\cC$-connected, contradicting the choice of~$u$ and~$v$.

  Taking the complements of $C\cup\{x,y\}$ and $\comp{C}\cup\{x,y\}$ and using \Cref{A^c}, we get that $C\cap\compb{x,y}$ and~$\comp{C}\cap\compb{x,y}$ are $\cC$-connected.
  Since they are two disjoint $\cC$-connected subsets of~$\compb{x,y}$ that cover it, they are the two $\cC$-components of~$\compb{x,y}$ (cf.~\Cref{only-two-components}).
\end{proof}

\begin{corollary}
  \label{cor:list-of-connected-parts}
  Let $(X,\cC)$ be a connectivity space.
  If $(X,\cC)$ is $2$-flimsy then the elements of $\cC$ are exactly the following subsets of~$X$: $\emptyset$, $X$, singletons and their complements, and subsets of the form~$C$, $C \cup \{x\}$ or $C \cup \{x,y\}$ for some distinct points $x,y \in X$ and some $\cC$-component $C$ of $\compb{x,y}$.
\end{corollary}

\begin{proof}
  By axioms~\ref{ax:singleton-connected} and~\ref{ax:union-connected}, singletons and $\emptyset$ are $\cC$-connected.
  By definition of $2$-flimsy spaces (\Cref{def:c-space_flimsy}), $X$ and complements of singletons are $\cC$-connected.
  For any distinct $x,y \in X$ and any $\cC$-component $C$ of~$\compb{x,y}$, the sets $C$, $C\cup\{x\}$ and $C\cup\{x,y\}$ are $\cC$-connected by definition of $\cC$-components and by 
  \Cref{Ccupx}.
  Conversely, if $D$ is a $\cC$-connected subset, we deduce from \Cref{prop:list-of-connected-parts} that either $\card D \leq 1$, or $\card{\comp{D}} \leq 1$, or there exist distinct $x,y\in X$ and a $\cC$-component~$C$ of~$\compb{x,y}$ such that $C \subseteq D \subseteq C\cup\{x,y\}$.
\end{proof}

\begin{remark}
  \label{rmk:rational_circle}
  \Cref{prop:list-of-connected-parts} and \Cref{cor:list-of-connected-parts} need not hold when $(X,\cC)$ does not satisfy \axref{ax:button-zip}.
  Indeed, consider the following counterexample (due to Fabian Gundlach): $X = \Q/\Z$, and~$\cC$ consists of intersections of connected subsets of $\R/\Z$ (for the usual topology) with~$\Q/\Z$.
  Then, $(X, \cC)$ satisfies axioms \ref{ax:singleton-connected}, \ref{ax:union-connected}, and \ref{ax:complement-connected}, and $(1,\sqrt 2) \cap \Q/\Z$ is a $\cC$-connected subset of~$X$, but it is not of the desired form as $\sqrt 2 \notin \Q/\Z$.
\end{remark}

A useful consequence of \Cref{cor:list-of-connected-parts} is that distinct $2$-flimsy connectivities are incomparable:

\begin{corollary}
  \label{cor:comparison_connectivities}
  Let $X$ be a set, and let $\cC_1,\cC_2$ be two connectivities on~$X$ such that $(X,\cC_1)$ and~$(X,\cC_2)$ are $2$-flimsy.
  If $\cC_1\subseteq \cC_2$, then $\cC_1=\cC_2$.
\end{corollary}

\begin{proof}
  Thanks to \Cref{cor:list-of-connected-parts}, it suffices to prove that for any distinct $x,y\in X$, the $\cC_1$-components of $\compb{x,y}$ coincide with the $\cC_2$-components of $\compb{x,y}$.
  By \Cref{only-two-components}, $\compb{x,y}$ has exactly two $\cC_1$-components.
  These two subsets partition $\compb{x,y}$ and are $\cC_2$-connected (as $\cC_1\subseteq \cC_2$).
  Again by \Cref{only-two-components}, it follows that these two subsets are in fact the two $\cC_2$-components of $\compb{x,y}$.
\end{proof}

\subsection{The separation relation associated to a $2$-flimsy connectivity space}
\label{sn:from_connectivity_to_separation}

Let $(X, \cC)$ be a $2$-flimsy connectivity space.
In order to relate~$X$ to order-theoretic structures, we must construct such a structure using only the $\cC$-connected subsets of~$X$.
To this end, we equip~$X$ with a relation~$\bot_\cC$ (which will be shown to be a dense order-complete separation relation in \Cref{prop:2f-seprel}) in the following way:

\begin{definition}
  \label{def:2f-seprel}
  If $\{a,b\}$ and $\{c,d\}$ are chords of~$X$, with $a,b,c,d$ four distinct points, we say that $\{a,b\} \bot_\cC \{c,d\}$ if $c$ and $d$ are in distinct $\cC$-components of~$X\setminus\{a,b\}$.
\end{definition}

By definition of~$\bot_\cC$, for any three distinct points $x,y,z\in X$, the open interval $I_{\not\ni z}(x,y)=\suchthat{w\in \compb{x,y,z}}{\{x,y\}\bot_\cC\{w,z\}}$ of $(X,\bot_\cC)$ is exactly the $\cC$-component of~$\compb{x,y}$ that does not contain $z$ (well-defined by \Cref{cor:comp-not-containing}).

\begin{lemma}
  \label{lem:four-points}
  Let $a,b,c,d$ be four distinct points of~$X$.
  The following implications hold:
  \begin{enumroman}
    \item
      \label{separation_step1}
      if $\{a,c\} \bot_\cC \{b,d\}$, then $\{a,b\} \nbot_\cC \{c,d\}$;
    \item
      \label{separation_step2}
      if $\{a,b\} \nbot_\cC \{c,d\}$ and $\{a,d\} \nbot_\cC \{b,c\}$, then $\{b,d\} \bot_\cC \{a,c\}$;
    \item
      \label{separation_step3}
      if $\{a,c\}\bot_\cC \{b,d\}$, then $\{b,d\}\bot_\cC\{a,c\}$.
  \end{enumroman}
\end{lemma}

\begin{proof}
  We show each point separately.
  \begin{enumroman}
    \item
      Let $C := C_{\ni d}(\compb{a,c})$.
      The set $C\cup\{c\}$ is $\cC$-connected by \Cref{Ccupx}, contains both $c$ and $d$, and is contained in $\compb{a,b}$, so $c$ and $d$ are in the same $\cC$-component of~$\compb{a,b}$, i.e., $\{a,b\}\nbot_\cC \{c,d\}$.
    \item
      That $\{a,b\}\nbot_\cC \{c,d\}$ means that $c\in C_{\ni d}(\compb{a,b})$.
      Likewise, $\{a,d\}\nbot_\cC \{b,c\}$ means that $c\in C_{\ni b}(\compb{a,d})$, and so the intersection $C := C_{\ni d}(\compb{a,b})\cap C_{\ni b}(\compb{a,d})$ contains $c$.
      But \Cref{intersect-intervals} asserts that $C = C_{\not\ni a}(\compb{b,d})$, so $c \in C$ means that $\{b,d\} \bot_\cC \{a,c\}$.
    \item
      Use \ref{separation_step1} twice, then use \ref{separation_step2}.
      \qedhere
  \end{enumroman}
\end{proof}

\begin{lemma}
  \label{lemma:open-avoids-edges}
  Let $v,w,x,y,z\in X$ be such that $v,w,x$ are distinct and such that $v,y,z$ are distinct.
  If $C_{\ni v}(\compb{w,x}) \subseteq C_{\ni v}(\compb{y,z}) \cup \{y,z\}$, then $y \notin C_{\ni v}(\compb{w,x})$.
\end{lemma}

Informally, \Cref{lemma:open-avoids-edges} means that if an interval~$A$ ``without ends'' is contained in an interval~$B$ ``with ends'', then the ends of $B$ cannot belong to $A$, i.e., $A$ is contained in $B$ minus its ends.
This fact will be used in \Cref{prop:2f-seprel} to prove that $\bot_\cC$ is an order-complete separation relation.

\begin{proof}
  We dualize the claim by taking complements: we have $C_{\not\ni v}(\compb{y,z}) \subseteq C_{\not\ni v}(\compb{w,x}) \cup \{w,x\}$, and we must show that $y \in C_{\not\ni v}(\compb{w,x}) \cup \{w,x\}$.
  We reason by contradiction by assuming the contrary, which amounts to assuming that $C_{\not\ni v}(\compb{w,x}) \cup \{w\} \subseteq \compb{x,y}$.
  
  The two sets $C_{\not\ni v}(\compb{y,z}) \cup \{y\}$ and $C_{\not\ni v}(\compb{w,x}) \cup \{w,x\}$ are $\cC$-connected by \Cref{cor:comp-not-containing} and \Cref{Ccupx}, and $C_{\not\ni v}(\compb{y,z}) \subseteq C_{\not\ni v}(\compb{w,x}) \cup \{w,x\}$ by hypothesis, so their union is $C_{\not\ni v}(\compb{w,x}) \cup \{w,x,y\}$, and it is $\cC$-connected by \axref{ax:union-connected}.
  By \Cref{cor:comp-not-containing} and \Cref{Ccupx}, $C_{\not\ni v}(\compb{w,x})\cup\{w\}$ is $\cC$-connected too.
  Since this is a subset of $\compb{x,y}$ by assumption, it follows from \Cref{if-you-touch-edges-you-are-component} that~$C_{\not\ni v}(\compb{w,x})\cup\{w\}$ is a $\cC$-component of $\compb{x,y}$.
  Then, \Cref{Ccupx} implies that $C_{\not\ni v}(\compb{w,x})\cup\{w,y\}$ is $\cC$-connected.
  Applying \Cref{if-you-touch-edges-you-are-component} again, we see that~$C_{\not\ni v}(\compb{w,x})$ is a $\cC$-component of~$\compb{w,y}$, so $C_{\not\ni v}(\compb{w,x})\cup\{y\}$ is $\cC$-connected by \Cref{Ccupx}.
  This contradicts the maximality of~$C_{\not\ni v}(\compb{w,x})$ among the $\cC$-connected subsets of~$\compb{w,x}$.
\end{proof}

\begin{proposition}
  \label{prop:2f-seprel}
  The relation $\bot_\cC$ is a dense order-complete separation relation on $X$.
\end{proposition}

\begin{proof}
  We first check the axioms of \Cref{def:sep_rel}: \axref{ax:separation-strict} is clear from the definition of $\bot_\cC$, \axref{ax:separation-symmetric} follows directly from \iref{lem:four-points}{separation_step3}, and \axref{ax:separation-total} follows from points~\ref{separation_step1} and~\ref{separation_step2} of \Cref{lem:four-points}.
  For \axref{ax:separation-transitive}, consider five distinct points $a,b,c,d,e$ of~$X$, and recall from \Cref{only-two-components} that $X\setminus \{a,b\}$ has exactly two $\cC$-components: either $c,d,e$ all lie in the same component, or two of them lie in the same component and the third one does not, which gives the two cases.
  That the separation relation $\bot_\cC$ is dense follows again from \Cref{only-two-components}: given $\{a,b\}\in \Cho(X)$, the set~$\compb{a,b}$ has two $\cC$-components, so we can pick $c$ in one and $d$ in the other, and then $\{a,b\}\bot_\cC\{c,d\}$.
  
  Now, consider a non-empty chain $\{U_i\}$ of open intervals of $(X,\bot_\cC)$ contained in an open interval~$I$.
  Then,  $U := \bigcup U_i$ is $\cC$-connected by \axref{ax:union-connected}, and is contained in $I$ (so $|\comp{U}|\geq |\comp{I}|\geq 2$).
  If~$U$ is not an open interval, then $\card{U}\geq \card{U_i}+1\geq 2$, and by \Cref{prop:list-of-connected-parts} there must be distinct points $x,y,z \in X$ such that $U$ is either of the form $C_{\ni z}(\compb{x,y}) \cup \{x\}$ or $C_{\ni z}(\compb{x,y}) \cup \{x,y\}$.
  But then, by \Cref{lemma:open-avoids-edges}, none of the open intervals~$U_i \subseteq U$ contain~$x$, contradicting $x \in U$.
\end{proof}

\section{Flimsy topological spaces}
\label{sn:2f-top}

In this section, we focus on flimsy topological spaces, i.e., we specialize the results about flimsy connectivity spaces (from \Cref{sn:props-2f}) to the case where~$\cC$ is the set of connected subsets of a topological space.
The main result of this section is \Cref{thm:intro-2f-top} from the introduction.

Throughout the section, we always adopt the convention that~$\emptyset$ is connected.

\begin{proposition}
  \label{prop:top-is-cspace}
  Let~$X$ be a topological space, and let~$\cC$ be the set of connected subsets of~$X$.
  Then, $(X, \cC)$ is a connectivity space.
  Moreover:
  \begin{itemize}
    \item
      $X$ is a $\rT_1$ topological space $\Longleftrightarrow$ $(X, \cC)$ is a $\rT_1$ connectivity space (cf.~\Cref{def:c-space-T1}).
    \item
      For any $n \in \N$,
      $X$ is an $n$-flimsy topological space $\Longleftrightarrow$ $(X, \cC)$ is an $n$-flimsy connectivity space.
  \end{itemize}
\end{proposition}

\begin{proof}
  \Axref{ax:singleton-connected} clearly holds.
  For \axref{ax:union-connected}, see \cite[Theorem~26.7a]{willard}.
  \Axref{ax:complement-connected} follows directly from \cite[Chapter~XVI.3, Theorem~4]{kuratowski}.
  
  We now check \axref{ax:button-zip}---in fact, we show that the case~\ref{ax:button-zip-button} always holds.
  Let $A$ and $B$ be non-empty connected subsets of $X$, and let $\bar{A}$ and~$\bar{B}$ be their respective closures in $A \cup B$.
  These two sets are closed in $A\cup B$ and cover $A\cup B$, so $\bar{A}\cap\bar{B} \neq \emptyset$ as $A\cup B$ is connected.
  Let $x\in \bar{A}\cap\bar{B}$.
  We have $A \subseteq A\cup\{x\}\subseteq \bar{A}$ and $B \subseteq B\cup\{x\}\subseteq \bar{B}$.
  Since $A$ and $B$ are connected, it follows from~\cite[Theorem~26.8]{willard} that $A\cup\{x\}$ and $B\cup\{x\}$ are connected.

  The equivalences of the notions of $\rT_1$ spaces and of the notions of $n$-flimsy spaces follow easily from the definitions.
\end{proof}

\Cref{prop:top-is-cspace} implies that the results of \Cref{sn:props-2f} apply to flimsy topological spaces.
In particular, by \Cref{thm:no-3f-cspace}, there are no $n$-flimsy topological spaces when $n \geq 3$.

In what follows, we fix a $2$-flimsy topological space $(X,\tau)$.
Let $(X, \cC)$ be the corresponding $2$-flimsy connectivity space (\Cref{prop:top-is-cspace}), let $\bot=\bot_\cC$ be the corresponding order-complete dense separation relation on~$X$ (\Cref{def:2f-seprel}, \Cref{prop:2f-seprel}), and let~$\tau_\rK$ be the order topology of~$(X,\bot)$ (\Cref{def:order-top_seprel}).

We prove that~$\tau$ is finer than~$\tau_\rK$, i.e., that the open intervals of~$(X, \bot)$ are open subsets of~$X$:

\begin{lemma}
  \label{C_open}
  Let $x, y \in X$ be distinct.
  If $C$ is a connected component of~$\compb{x,y}$, then~$C$ is open
  in~$X$, $C\cup\{x,y\}$ is closed and not open
  in~$X$, and $C\cup\{x\}$ is neither open nor closed in~$X$.
\end{lemma}

\begin{proof}
  By \Cref{only-two-components}, the sets $C$ and $D := \compb{x,y}\setminus C$ are the two connected components of~$\compb{x,y}$, hence they are both closed in $\compb{x,y}$, so $C$ is open and closed in $\compb{x,y}$.
  Since $X$ is~$\rT_1$ by \Cref{flimsy-is-t1} and \Cref{prop:top-is-cspace}, the subset~$\{x,y\}$ is closed in $X$, so~$\compb{x,y}$ is open in $X$.
  Therefore, $C$ is open in~$X$, and so is $D$ by symmetry.
  Hence, $C \cup \{x,y\} = X \setminus D$ is closed, and as a proper subset of the connected space~$X$, it cannot also be simultaneously open.
  Finally, $C\cup\{x\}$ cannot be closed because otherwise~$C$ would be a non-empty open and closed proper subset of the space~$\compb{x}$, which is connected as $X$ is $2$-flimsy.
  Similarly, if $C\cup \{x\}$ were open, then~$C\cup\{x\}$ would be open and closed in the connected space~$\compb{y}$.
\end{proof}

We prove the following result, which was already obtained in \cite[Theorem~15]{pengcao}:

\begin{proposition}
  \label{thm:2f-is-hausdorff}
  The $2$-flimsy space $X$ is Hausdorff.
\end{proposition}

\begin{proof}
  Let $x, y$ be two distinct points of~$X$.
  By density of~$\bot$, pick $\{a,b\}$ such that $\{x,y\} \bot \{a,b\}$.
  By symmetry of $\bot$, we have $\{a,b\} \bot \{x,y\}$, so $C_{\ni x}(\compb{a,b})$ and $C_{\ni y}(\compb{a,b})$ are two disjoint open subsets of~$X$ (by~\Cref{C_open}) containing respectively $x$ and $y$.
\end{proof}

We can now show \Cref{thm:intro-2f-top} from the introduction:

\begin{proof}[Proof of \Cref{thm:intro-2f-top}]
  First, the space~$(X,\tau)$ is Hausdorff by \Cref{thm:2f-is-hausdorff} (or \cite[Theorem~15]{pengcao}).
  Second, recall from \Cref{cor:seprel-big-circ} that $(X,\tau_\rK)$ is a big circle, so it is compact and $2$-flimsy by \Cref{prop:big-circle-2f}.
  Third, observe from \Cref{def:open-interval-seprel,def:2f-seprel} that for any distinct $x,y\in X$, the open intervals of $(X,\bot)$ between $x$ and $y$ are exactly the $\cC$-components of~$\compb{x,y}$, so the topology~$\tau_\rK$ is indeed generated by the connected components (with respect to the topology~$\tau$) of the subsets $\compb{x,y}$ for distinct $x,y\in X$.
  Combined with \Cref{C_open}, this shows that~$\tau_\rK$ is coarser than~$\tau$.
  
  Let~$\cC_\rK$ be the set of connected subsets of $(X,\tau_\rK)$.
  Since $\tau_\rK\subseteq \tau$, the identity of~$X$ is a continuous map from $(X,\tau)$ to $(X,\tau_\rK)$, which implies that $\cC\subseteq \cC_\rK$, and so $\cC=\cC_\rK$ by \Cref{cor:comparison_connectivities}.
  In other words, $(X,\tau)$ and $(X,\tau_\rK)$ have the same connected subsets.
  It follows that, for any distinct $x,y\in X$, the connected components of $\compb{x,y}$ are the same for the two topologies $\tau$ and $\tau_\rK$.
  Then, recalling that a $2$-flimsy topological space is always $\rT_1$ and connected on the one hand, and combining \Cref{cor:list-of-connected-parts} with \Cref{C_open} on the other hand, we get that $(X,\tau)$ and $(X,\tau_\rK)$ have the same open connected subsets.

  To show uniqueness, let $\tau'$ be a topology on~$X$ coarser than~$\tau$ such that $(X,\tau')$ is a compact $2$-flimsy space.
  Since $\tau'\subseteq \tau$, any connected subset of $(X,\tau)$ is connected in $(X,\tau')$.
  Then, by \Cref{cor:comparison_connectivities}, $(X,\tau)$ and $(X,\tau')$ have in fact the same connected subsets.
  Hence, for any distinct~$x,y\in X$, the connected components of $\compb{x,y}$ are the same for the two topologies~$\tau$ and~$\tau'$.
  Applying the conclusions of the first paragraph of the proof to the space~$(X,\tau')$, we obtain $\tau_\rK\subseteq \tau'$.
  Since~$(X,\tau')$ is compact, and since $(X,\tau_\rK)$ is Hausdorff by \Cref{thm:2f-is-hausdorff}, \Cref{lattice_Hausdorff-compact} ensures that~$\tau'=\tau_\rK$.
\end{proof}

\begin{example}
  \label{ex:non-compact-2f}
  We now give an example of a non-compact $2$-flimsy topological space.
  We define a function $\gamma \colon (0,2] \to \R^3$ as follows: for any $s\in(0,2]$, we let
  \[
    \gamma(s)
    =
    \begin{cases}
        \bigl(
          s,
          \sin \tfrac{2\pi}{s},
          0
        \bigr)
        & \text{if } s \in (0,1],
        \\
        \bigl(
          2-s,
          0,
          (s-1)(2-s)
        \bigr)
        & \text{if } s \in (1,2].
    \end{cases}
  \]

  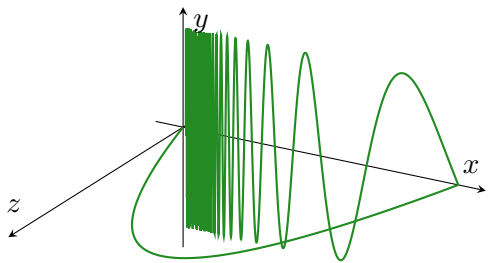
\begin{figure}[H]
    \begin{center}
      \begin{tikzpicture}
        \begin{axis}[
            view={120}{30},
            axis lines=center,
            xlabel={$z$},
            ylabel={$x$},
            zlabel={$y$},
            domain=0.01:1,
            samples=2000,
            samples y=0,
            enlargelimits=true,
            grid=major,
        ]
          \addplot3[
            thick,
            ForestGreen,
            domain=0.01:1,
          ] (
            {0},
            {x},
            {sin(deg(2*pi/x))}
          );

          \addplot3[
            thick,
            ForestGreen,
            domain=1:2,
          ] (
            {(x-1)*(2-x)},
            {2-x},
            {0}
          );
        \end{axis}
      \end{tikzpicture}
    \end{center}
    \caption{An illustration of \Cref{ex:non-compact-2f}.}
  \end{figure}

  The map~$\gamma$ is continuous and injective on $(0,2]$.
  Let~$X$ be the image of~$\gamma$, equipped with the subspace topology~$\tau$ inherited from the Euclidean topology of~$\R^3$.
  In particular, the space~$X$ is Hausdorff.
  However, $X$ is not compact since it is not closed in~$\R^3$ (its closure contains the segment $\{0\} \times [-1,1] \times \{0\}$).
  We leave it to the reader to check that~$X$ is $2$-flimsy.

  \Cref{thm:intro-2f-top} states that there is a unique topology $\tau_\rK$ on~$X$ coarser than~$\tau$ such that $(X,\tau_\rK)$ is compact and $2$-flimsy.
  In fact, here, $(X,\tau_\rK)$ is homeomorphic to the standard circle~$\bbS^1 = \R/\Z$.
  Indeed, the functions $s\in (0,2]\longmapsto \gamma(s)\in X$ and $\varphi\colon s\in(0,2]\longmapsto \frac s2 \in\bbS^1$ are bijections, so the set $\tau_\rK := \suchthat{\gamma\circ\varphi^{-1}(U)}{U\text{ open in }\bbS^1}$ is a topology on~$X$ such that~$(X,\tau_\rK)$ is homeomorphic to~$\bbS^1$, and it is easy to check that $\tau_\rK \subseteq \tau$.
\end{example}  

\section{Big-path-flimsy topological spaces}
\label{sn:2f-big-path}

In this section, we study big-path-flimsy topological spaces (\Cref{def:flimsy_path}), which correspond to the notion of big-path-connectedness defined in \Cref{defn:conticonn}.

\begin{proposition}
  \label{prop:continuum-is-cspace}
  Let~$X$ be a topological space, and let $\cC$ be the set of all big-path-connected subsets of~$X$, including~$\emptyset$.
  Then, $(X, \cC)$ is a connectivity space.
  Moreover:
  \begin{itemize}
    \item
      $X$ is a $\rT_1$ topological space $\Longleftrightarrow$ $(X, \cC)$ is a $\rT_1$ connectivity space.
    \item
      For any $n \in \N$,
      $X$ is an $n$-big-path-flimsy topological space $\Longleftrightarrow$ $(X, \cC)$ is an $n$-flimsy connectivity space.
  \end{itemize}
\end{proposition}

\begin{proof}
  \Axref{ax:singleton-connected} clearly holds.
  We verify axioms~\ref{ax:union-connected}, \ref{ax:complement-connected}, and \ref{ax:button-zip}:
  \begin{description}
    \item[\ref{ax:union-connected}]
      Let $x \in \bigcap C$, and let $y, z \in \bigcup C$.
      Pick $Y, Z \in C$ such that $y \in Y$ and $z \in Z$.
      By hypothesis, there are big paths from $y$ to $x$ in $Y$, and from $x$ to $z$ in $Z$.
      Concatenating these big paths yields a big path from $y$ to $z$ in $\bigcup C$.
    \item[\ref{ax:complement-connected}]
      It suffices to show that $x$ can be connected by a big path in $Y \setminus Z$ to any element $y \in C$, where~$C$ is any big-path-connected component of $Y \setminus \{x\}$ besides $Z$ (in particular, $x \neq y$).
      Since $Y$ is big-path-connected, there is a big path $\gamma \colon L \to Y$ from $x$ to $y$.
      Let $t_0 := \sup\suchthat{t \in L}{\gamma(t) \in \overline{\{x\}}}$, so that $\gamma(t_0) \in \overline{\{x\}}$ but $\gamma(t^*) \notin \overline{\{x\}}$ for all $t^* > t_0$.
      For any $t^* \in ( t_0,\max L)$, we have $\gamma(t^*) \in C$, since the restricted big path $\gamma|_{[t^*, \max L]}$ avoids $x$ and $C$ is a big-path-connected component of~$Y \setminus \{x\}$.
      If $\gamma(t_0) = x$, then the restricted big path $\gamma|_{[t_0, \max L]}$ defines a big path from $x$ to $y$ contained in $C \cup \{x\} \subseteq Y \setminus Z$, and we are done.
      Otherwise, if $\gamma(t_0)\neq x$, then $\gamma|_{[t_0, \max L]}$ is a big path from $\gamma(t_0)$ to $y$, which is contained in $C$ as it avoids~$x$.
      It then suffices to construct a big path from $x$ to $\gamma(t_0)$ contained in $C \!\cup\! \{x\}$.
      But as $\gamma(t_0) \! \in\!  \overline{\{x\}}$, we can simply take
      \[
        \left\lbrace
        \begin{matrix}
          t \in [0,\frac12) & \mapsto & x \\
          t \in [\frac12,1] & \mapsto & \gamma(t_0).
        \end{matrix}
        \right.
      \]
    \item[\ref{ax:button-zip}]
      The proof of \axref{ax:button-zip} is postponed to \Cref{sec:proof_C4}.
  \end{description}
  The equivalence of the notions of $\rT_1$ spaces and the equivalence between $X$ being $n$-big-path-flimsy and $(X, \cC)$ being $n$-flimsy follow easily from the definitions.
\end{proof}

\Cref{prop:continuum-is-cspace} implies that the results of \Cref{sn:props-2f} apply to big-path-flimsy topological spaces.
In particular, by \Cref{thm:no-3f-cspace}, there are no $n$-big-path-flimsy topological spaces when $n \geq 3$.

\begin{proposition}
\label{thm:2-big-path-f-is-compact}
  If~$X$ is a $2$-big-path-flimsy space, then there exist a big interval $L$ and a continuous surjection $\varphi \colon L \twoheadrightarrow X$.
  In particular, $X$ is compact as the continuous image of a compact space.
\end{proposition}

\begin{proof}
  Let $x,y\in X$ be distinct.
  \Cref{prop:continuum-is-cspace,only-two-components} ensure that $\compb{x,y}$ has exactly two big-path-connected components, which we denote by~$C_1$ and~$C_2$.
  We know from \Cref{Ccupx} that~$C_1\cup\{x,y\}$ and $C_2\cup\{x,y\}$ are big-path-connected.
  Hence, there exist big paths $\gamma_1 \colon L_1 \to C_1\cup\{x,y\}$ and $\gamma_2 \colon L_2 \to C_2\cup\{x,y\}$ such that $\gamma_1(\min L_1)=\gamma_2( \max L_2)=x$ and $\gamma_1(\max L_1) = \gamma_2(\min L_2) = y$.

  Then, define a big path~$\gamma$ from~$x$ to~$x$ on~$X$ by concatenating~$\gamma_1$ and~$\gamma_2$.
  We want to show that~$\gamma$ is surjective, i.e., that $\gamma_1(L_1)\cup\gamma_2(L_2)=X$.
  As they are continuous images of the big-path-connected spaces~$L_1$ and~$L_2$, the subsets $\gamma_1(L_1)$ and $\gamma_2(L_2)$ are both big-path-connected.
  By \Cref{prop:continuum-is-cspace,A^c}, $\comp{\gamma_1(L_1)}$ and $\comp{\gamma_2(L_2)}$ are also big-path-connected.
  Next, observe that $\gamma_1(L_1)\cap\gamma_2(L_2)=\{x,y\}$ because~$C_1$ and~$C_2$ are disjoint.
  Since $X$ is $2$-big-path-flimsy, $\compb{x,y}=\comp{\gamma_1(L_1)}\cup\comp{\gamma_2(L_2)}$ is not big-path-connected, so the sets $\comp{\gamma_1(L_1)}$ and $\comp{\gamma_2(L_2)}$ are disjoint, which means that $\gamma_1(L_1)\cup\gamma_2(L_2)=X$.
  The last part of the claim follows directly from \iref{prop:continuum-facts}{item:continuum-compact-subsets}.
\end{proof}

We are now going to prove \Cref{thm:intro-2f-big_path} from the introduction, after setting some notation.
Let~$(X,\tau)$ be a $2$-big-path-flimsy topological space, let $(X,\cC)$ be the corresponding $2$-flimsy connectivity space (\Cref{prop:continuum-is-cspace}), and let $\bot=\bot_\cC$ be the corresponding order-complete dense separation relation on~$X$ (\Cref{def:2f-seprel}, \Cref{prop:2f-seprel}).
We also define some other topologies on~$X$: the topology $\tau_\rH$ is the order topology of~$(X,\bot)$ (\Cref{def:order-top_seprel}), and~$\tau_\cF$ is the final topology on~$X$ induced by the family of all big paths on $(X,\tau)$, i.e.,
\begin{equation}
  \label{eqn:def-tau-cF}
  \tau_\cF
  :=
  \suchthat{
    V\subseteq X
  }{
    \gamma^{-1}(V)\text{ is open for all big paths }\gamma \colon L \to (X,\tau)
  }.
\end{equation}

\begin{proof}[Proof of \Cref{thm:intro-2f-big_path}]
  We have already shown that the space $(X,\tau)$ is compact in \Cref{thm:2-big-path-f-is-compact}.
  We also know from \Cref{cor:seprel-big-circ} that $(X,\tau_\rH)$ is a big circle, so it is Hausdorff and $2$-big-path-flimsy by \Cref{prop:big-circle-2f}.
  Next, using \Cref{def:open-interval-seprel,def:2f-seprel}, we see that, for any distinct $x,y\in X$, the open intervals of $(X,\bot)$ between $x$ and $y$ are exactly the $\cC$-components of~$\compb{x,y}$.
  Thus, 
  \begin{equation}
    \label{inter-fact_intro-2f-big_path}
    \begin{matrix}
      \textnormal{
        $\tau_\rH$ is generated by the big-path-connected components of the subsets~$\compb{x,y}$,
      }
      \\
      \textnormal{
        as~$x$ and~$y$ range over pairs of distinct points of~$X$.
      }
    \end{matrix}
  \end{equation}

  In order to prove that $\tau_\rH$ is finer than~$\tau$ and that the topological spaces~$(X,\tau)$ and~$(X,\tau_\rH)$ have the same big paths, we only need to show that $\tau_\rH$ coincides with the topology $\tau_\cF$ from \Cref{eqn:def-tau-cF}.
  Indeed, first, it is clear that $\tau \subseteq \tau_\cF$ as big paths are continuous, so any big path on $(X,\tau_\cF)$ is also continuous with respect to~$\tau$.
  Conversely, for any big path $\gamma$ on $(X,\tau)$ and any $U\in \tau_\cF$, the set~$\gamma^{-1}(U)$ is open in~$L$ by definition of~$\tau_\cF$, so~$\gamma$ is also continuous with respect to~$\tau_\cF$.

  In order to prove that $\tau_\rH=\tau_\cF$, we want to apply \Cref{lattice_Hausdorff-compact}.
  We have already proved that~$(X,\tau_\rH)$ is Hausdorff.
  Since $(X,\tau)$ and $(X,\tau_\cF)$ have the same big paths, their big-path-connected subsets coincide, so $(X,\tau_\cF)$ is $2$-big-path-flimsy and thus compact by \Cref{thm:2-big-path-f-is-compact}.
  It only remains to show that $\tau_\rH\subseteq \tau_\cF$.
  Equivalently, we must prove that, for any distinct $x,y\in X$, if~$C$ is a big-path-connected component of $\compb{x,y}$, then $C \in \tau_\cF$.
  Fix such a~$C$ and a big path $\gamma \colon L \to (X,\tau)$, and consider a $t\in \gamma^{-1}(C)$.
  We have $\compb{x,y}\in\tau$, since $(X,\tau)$ is $\rT_1$ by \Cref{flimsy-is-t1} and \Cref{prop:continuum-is-cspace}, and $\gamma(t) \in C \subseteq \compb{x,y}$, so by continuity of $\gamma$ and by definition of the order topology, there is an open interval~$I$ of~$L$ such that $t\in I$ and $\gamma(I)\subseteq \compb{x,y}$.
  \Cref{cor:bigpathconn-is-conn} implies that $\gamma(I)$ is big-path-connected, so it is contained in the big-path-connected component $C$ of $\gamma(t)$ in $\compb{x,y}$, so~$I$ is an open neighborhood of the (arbitrary) point $t \in \gamma^{-1}(C)$, and it is contained in $\gamma^{-1}(C)$.
  This shows that~$\gamma^{-1}(C)$ is open in~$L$ and thus $C \in \tau_\cF$ as desired.

  By \Cref{lattice_Hausdorff-compact}, we have $\tau_\rH=\tau_\cF$.
  In particular,
  \begin{equation}
    \label{inter2-fact_intro-2f-big_path}
    \tau\subseteq \tau_\rH,
  \end{equation}
  and the big paths on $(X,\tau)$ and $(X,\tau_\rH)$ coincide, so $(X,\tau)$ and $(X,\tau_\rH)$ have the same big-path-connected subsets.

  To prove the uniqueness of~$\tau_\rH$, let $\tau'$ be a topology on $X$ finer than $\tau$ such that $(X,\tau')$ is a Hausdorff $2$-big-path-flimsy space.
  Since $\tau\subseteq \tau'$, any big path on $(X,\tau')$ is a big path on~$(X,\tau)$, and so any big-path-connected subset of~$(X,\tau')$ is also big-path-connected in~$(X,\tau)$.
  By \Cref{cor:comparison_connectivities}, $(X,\tau)$ and $(X,\tau')$ must then in fact have the same big-path-connected subsets.
  Thus, for any distinct $x,y\in X$, the big-path-connected components of $\compb{x,y}$ are the same for the two topologies~$\tau$ and~$\tau'$.
  Applying~\eqref{inter-fact_intro-2f-big_path} and~\eqref{inter2-fact_intro-2f-big_path} to~$(X,\tau')$ entails that $\tau'\subseteq\tau_\rH$.
  Since~$(X,\tau')$ is Hausdorff and since~$(X,\tau_\rH)$ is compact by \Cref{thm:2-big-path-f-is-compact}, \Cref{lattice_Hausdorff-compact} implies that $\tau'=\tau_\rH$.
\end{proof}

\begin{example}
  \label{ex:non-hausdorff-2bpf}
  We now give an example of a non-Hausdorff $2$-big-path-flimsy topological space.
  Let~$\bbS^1 := \R/\Z$ be the standard circle, and let $\mathcal E$ be its usual Euclidean topology.
  Then,
  \[
    \tau :=
    \suchthat{
      U\in\mathcal{E}
    }{
      \bbS^1\setminus U\text{ is at most countable}
    }
    \cup
    \{\varnothing\}
  \]
  is a topology on~$\bbS^1$ that is coarser than~$\mathcal{E}$.
  Any two non-empty elements of $\tau$ have co-countable intersection, so $(\bbS^1,\tau)$ is not Hausdorff.
  Below, we prove that a big path $\gamma \colon L \to \bbS^1$ is continuous with respect to~$\tau$ if and only if it is continuous with respect to~$\mathcal{E}$.
  Thus, the space~$(\bbS^1,\tau)$ is $2$-big-path-flimsy because~$(\bbS^1,\mathcal{E})$ is (for example by \Cref{prop:big-circle-2f}).
  In this case, the unique Hausdorff topology~$\tau_\rH$ finer than~$\tau$ such that $(\bbS^1,\tau_\rH)$ is Hausdorff and $2$-big-path-flimsy, whose existence is ensured by \Cref{thm:intro-2f-big_path}, will hence simply the standard Euclidean topology~$\mathcal E$.

  Consider a big interval~$L$ and a map $\gamma \colon L \to \bbS^1$.
  If~$\gamma$ is continuous with respect to~$\mathcal{E}$, then it is continuous with respect to~$\tau$ since $\tau\subseteq \mathcal{E}$.
  Conversely, suppose that~$\gamma$ is continuous with respect to~$\tau$ and, by contradiction, assume that there exist $U\in\mathcal{E}$ and $\ell\in \gamma^{-1}(U)$ such that~$\ell$ is not in the interior of $\gamma^{-1}(U)$.
  Reversing the order of $L$ if needed, we can assume that $\ell$ is in the closure of $[\min L,\ell)\cap \gamma^{-1}(\bbS^1\setminus U)$.
  For any $s \in [\min L, \ell)$, it follows that $\gamma\bigl([s,\ell]\bigr)\setminus U$ is non-empty---denote by $F_s$ the closure of $\gamma\bigl([s,\ell]\bigr)\setminus U$ in~$(\bbS^1,\mathcal{E})$.
  Observe that for all $s<s'<\ell$, we have $\emptyset\neq F_{s'}\subseteq F_s\subseteq \bbS^1\setminus U$.
  By compactness of~$(\bbS^1, \mathcal E)$, there exists $x\in \bigcap_{s<\ell}F_s$.
  Note that $x\notin U$, so $x\neq \gamma(\ell)$.

  Let $s \in [\min L, \ell)$.
  Since $x\in F_s$ belongs to the closure of $\gamma([s,\ell])$, and since $(\bbS^1, \mathcal E)$ is a sequential space, we can find a sequence of points $t_{s,n}\in[s,\ell]$, indexed by integers $n\in\N$, such that $\gamma(t_{s,n})\longrightarrow x$ with respect to $\mathcal{E}$.
  Since $[s, \ell]$ is sequentially compact,%
  \footnote{
    \label{bigint-seqcomp}
    Any big interval is sequentially compact: as in the real numbers (with the same proof), any sequence either has a non-decreasing subsequence (converging to its supremum) or a decreasing subsequence (converging to its infimum).
  }
  replacing $(t_{s,n})$ by a subsequence if needed, we can assume that $t_{s,n} \longrightarrow t_s$ for some $t_s \in [s,\ell]$.

  We are going to show that $\gamma(t_s) = x$ for all $s < \ell$.
  This will yield the desired contradiction, as the closed subsets $[s,\ell] \cap \gamma^{-1}(\{x\})$ are then non-empty for all $s < \ell$ (they contain the respective~$t_s$), so their intersection in the (compact) big interval $L$ is also non-empty, meaning that $\gamma(\ell) = x$, contradicting $\ell \in \gamma^{-1}(U)$.

  Let $\mathrm M := \suchthat{n \in \N}{\gamma(t_{s,n}) = \gamma(t_s)}$.
  If $\mathrm M$ is infinite, replacing $(t_{s,n})$ by a subsequence if needed, we can assume that $\gamma(t_{s,n}) = \gamma(t_s)$ for all $n \in \N$; then, the fact that $\gamma(t_{s,n})\longrightarrow x$ with respect to~$\mathcal{E}$ implies (because~$(\bbS^1, \mathcal E)$ is Hausdorff) that $\gamma(t_{s,n}) = x$ for all $n$ large enough, and in particular $\gamma(t_s) = x$.
  We now assume that $\mathrm M$ is finite.
  Replacing $(t_{s,n})$ by a subsequence if needed, we can assume that $\gamma(t_{s,n}) \neq \gamma(t_s)$ for all $n \in \N$.
  Since $(\bbS^1, \mathcal E)$ is Hausdorff and $\gamma(t_{s,n}) \longrightarrow x$, the set
  \[
    V:=\suchthat{\gamma(t_{s,n})}{n \in \N} \cup \{x\}
  \]
  is closed for~$\mathcal E$, and hence closed for~$\tau$ as it is at most countable.
  Since $\gamma$ is continuous with respect to~$\tau$ and $t_{s,n} \longrightarrow t_s$, we must have $\gamma(t_s) \in V$ because $\gamma(t_{s,n}) \in V$ for all $n \in \N$.
  Since $\gamma(t_{s,n}) \neq \gamma(t_s)$ for all $n \in \N$, this can only mean that $\gamma(t_s) = x$.
  As explained above, this concludes the proof.
\end{example}

\section{Path-flimsy topological spaces}
\label{sn:2f-path}

In this section, we study path-flimsy topological spaces (\Cref{def:flimsy_path}), which correspond to the following standard notion of path-connectedness (we consider~$\emptyset$ to be path-connected):

\begin{definition}
\label{def:path-connectedness}
  Let~$X$ be a topological space.
  For any points~$x,y \in X$, a \emph{path} from~$x$ to~$y$ is a continuous map $f \colon [0,1] \to X$ such that $f(0) = x$ and $f(1) = y$.
  The space~$X$ is \emph{path-connected} if for any two points~$x,y \in X$, there is a path from~$x$ to~$y$.
\end{definition}

Most of the arguments of this section are straightforwardly adapted from those of \Cref{sn:2f-big-path} by replacing ``big paths'' and ``big-path-connectedness'' with ``paths'' and ``path-connectedness'', respectively.
For the sake of brevity, we omit the details of such adaptations and rather focus on the few points that require new proofs.

\begin{proposition}
  \label{prop:path-is-cspace}
  Let $X$ be a topological space, and let $\cC$ be the set of all path-connected subsets of~$X$, including~$\emptyset$.
  Then, $(X, \cC)$ is a connectivity space.
  Moreover:
  \begin{itemize}
    \item
      $X$ is a $\rT_1$ topological space $\Longleftrightarrow$ $(X, \cC)$ is a $\rT_1$ connectivity space.
    \item
      For any $n \in \N$,
      $X$ is an $n$-path-flimsy topological space $\Longleftrightarrow$ $(X, \cC)$ is an $n$-flimsy connectivity space.
  \end{itemize}
\end{proposition}

\begin{proof}
  The proofs of axioms~\ref{ax:singleton-connected}, \ref{ax:union-connected}, and~\ref{ax:complement-connected} are as in \Cref{prop:continuum-is-cspace}, replacing all big intervals involved by $[0,1]$, remarking that the concatenation of two paths is a path (the concatenation of~$[0,1]$ with itself is order-isomorphic to~$[0,1]$), and that the restriction of a path to any segment is also a path (if $0\leq s<t\leq 1$, then the subinterval $[s,t]$ is order isomorphic to $[0,1]$).
  The details are left to the reader.

  The equivalence of the notions of $\rT_1$ spaces and the equivalence between $X$ being $n$-path-flimsy and $(X, \cC)$ being $n$-flimsy follow easily from the definitions.

  We now focus on the proof of \axref{ax:button-zip}.
  Pick non-empty sets $A,B \in \cC$ such that $A \cup B \in \cC$.
  We assume that~$A$ and~$B$ are disjoint, as the claim is clear otherwise (case~\ref{ax:button-zip-button} holds for any $x \in A \cap B$).
  Pick arbitrary points $a \in A$ and $b\in B$, let $\hat{\gamma}$ be a path from~$a$ to~$b$ in~$A\cup B$, and let
  \[
    s_A
    :=
    \sup\suchthat{
      t\in[0,1]
    }{
      \hat{\gamma}(t)\in A
    }
    \andd
    x
    :=
    \hat{\gamma}(s_A).
  \]
  If $x\in A$, then $\hat\gamma|_{[s_A, 1]}$ is (after reparametrization) a path from $\hat{\gamma}(s_A)=x$ to $\hat{\gamma}(1)=b$ in $B \cup \{x\}$, so $B \cup \{x\}$ is path-connected, and then case~\ref{ax:button-zip-button} holds as $A = A\cup\{x\}$ is also path-connected.
  Therefore, for the rest of the proof, we assume that $x\in B$.
  In particular, $B\cup\{x\}=B$ is path-connected, so if we show that $A\cup\{x\}$ is also path-connected, then case~\ref{ax:button-zip-button} holds, concluding the proof.

  Without loss of generality, replacing $b$ by $x$ and $\gamma$ by the (reparametrized) restricted path $\hat{\gamma}|_{[0,s_A]}$, we may assume that $s_A = 1$, which implies that there is a strictly increasing sequence $(s_n)_{n\geq 0}$ of elements of $\gamma^{-1}(A)\subseteq [0,1)$ with $s_0=0$ and $s_n \longrightarrow 1$.
  For each $n\geq 0$, let $S_n = \gamma\bigl([s_n,s_{n+1}]\bigr)$, which is a path-connected subset of $A\cup B$.
  If $S_n \cap A$ is not path-connected for some $n \geq 0$, then we are in case~\ref{ax:button-zip-seam}, finishing the proof.
  Hence, we further assume that $S_n\cap A$ is path-connected for all $n \geq 0$, so there are continuous maps $\gamma_n \colon [s_n,s_{n+1}] \to S_n\cap A$ such that $\gamma_n(s_n) = \gamma(s_n)$ and $\gamma_n(s_{n+1}) = \gamma(s_{n+1})$.
  Now, define a map $\gamma^*\colon [0,1] \to A\cup \{x\}$ by setting
  \[
    \gamma^*(t)
    :=
    \begin{cases}
      \gamma_n(t) & \textnormal{if } t \in [s_n, s_{n+1}) \textnormal{ with } n \geq 0, \\
      x & \textnormal{if } t=1,
    \end{cases}
  \]
  for any $t\in[0,1]$.
  By construction, the map $\gamma^*$ is continuous on $[0,1)$.
  We have reduced to showing that~$\gamma^*$ is continuous at~$1$, as this implies that $\gamma^*$ is a path from $a$ to $x$ in $A \cup \{x\}$, and then $A \cup \{x\}$ is path-connected.
  So, let $U$ be an open subset of $A\cup B$ containing $\gamma^*(1) = x = \gamma(1)$.
  By continuity of~$\gamma$, there exists $N \geq 0$ such that $\gamma\bigl([s_N,1]\bigr)\subseteq U$.
  We have
  \[
    \gamma^*\bigl([s_N,1]\bigr)
    =
    \{x\} \cup
    \bigcup_{n\geq N}
      \gamma_n
        \bigl([s_n,s_{n+1})\bigr)
    \subseteq
    \{x\} \cup
    \bigcup_{n\geq N}
      S_n
    =
    \{x\} \cup
    \bigcup_{n\geq N}
      \gamma
        \bigl([s_n,s_{n+1}]\bigr)
    =
    \gamma
      \bigl([s_N,1]\bigr)
    \subseteq
    U,
  \]
  which implies that $\gamma^*$ is continuous at $1$.
  As explained above, this concludes the proof.
\end{proof}

\Cref{prop:path-is-cspace} implies that the results of \Cref{sn:props-2f} apply to path-flimsy topological spaces.
In particular, by \Cref{thm:no-3f-cspace}, there are no $n$-big-path-flimsy topological spaces when $n \geq 3$.

\begin{theorem}
  \label{thm:2pf-is-compact}
  If $X$ is a $2$-path-flimsy space, then there exists a continuous surjection $\varphi \colon \bbS^1 \twoheadrightarrow X$.
  In particular, $X$ is compact.
  Furthermore, if $X$ is also Hausdorff, then it is homeomorphic to $\bbS^1$.
\end{theorem}

\begin{proof}
  Here, we view $\bbS^1$ as the quotient topological space of $[0,1]$ obtained by identifying the points~$0$ and~$1$.
  Denote by $q \colon [0,1]\twoheadrightarrow \bbS^1$ the associated quotient map.

  Let $x,y\in X$ be distinct.
  \Cref{prop:path-is-cspace,only-two-components} ensure that $\compb{x,y}$ has exactly two path-connected components, which we denote by~$C_1$ and~$C_2$.
  We know from \Cref{Ccupx} that~$C_1\cup\{x,y\}$ and~$C_2\cup\{x,y\}$ are path-connected, so there exist paths $\gamma_1 \colon [0,1] \to C_1\cup\{x,y\}$ and $\gamma_2 \colon [0,1] \to C_2\cup\{x,y\}$ such that $\gamma_1(0)=\gamma_2(1)=x$ and $\gamma_1(1) = \gamma_2(0) = y$.
  Then, define another path~$\gamma \colon [0,1] \to X$ by setting for all $t \in [0,1]$,
  \[
    \gamma(t) =
    \begin{cases}
      \gamma_1(2t)   & \textnormal{if } t \leq 1/2,\\
      \gamma_2(2t-1) & \textnormal{if } t \geq 1/2.
    \end{cases}
  \]
  Note that $\gamma(0) = \gamma(1)=x$.
  Thus, there is a continuous map $\varphi \colon \bbS^1 \to X$ such that $\gamma = \varphi \circ q$, and its image is $\varphi(\bbS^1) = \gamma([0,1]) = \gamma_1([0,1]) \cup \gamma_2([0,1])$.
  Following the same arguments as in the proof of \Cref{thm:2-big-path-f-is-compact} (replacing \Cref{prop:continuum-is-cspace} with \Cref{prop:path-is-cspace}), we show that $\gamma_1\bigl([0,1]\bigr)\cup\gamma_2\bigl([0,1]\bigr)=X$, so $\varphi$ is a continuous surjection $\bbS^1 \twoheadrightarrow X$.
  That $X = \varphi(\bbS^1)$ is compact then follows from the fact that~$\bbS^1$ is compact.

  Now, assume that $X$ is Hausdorff.
  The sets $C_1\cup\{x,y\}$ and $C_2\cup\{x,y\}$ are then Hausdorff and path-connected, so they are arcwise connected by \cite[Theorem~31.6]{willard} (see also \cite{pathinj,pathinj-zorn}), meaning that any two distinct points are connected via an \emph{injective} path.
  Thus, we can assume without loss of generality that $\gamma_1$ and $\gamma_2$ are injective, and we want to show that $\varphi$ is injective.
  Assume that there are $0\leq s<t\leq 1$ such that $\gamma(s)=\gamma(t)$.
  Since $\gamma_1$ and $\gamma_2$ are injective, we must have $s<1/2<t$.
  Thus, $\gamma(s)\neq \gamma(1/2)=y$ and $\gamma(s)\in (C_1\cup\{x,y\})\cap(C_2\cup\{x,y\})=\{x,y\}$, so $\gamma(s)=\gamma(t)=x$.
  By injectivity of $\gamma_1$ and $\gamma_2$, we have $s=0$ and $t=1$, and thus $q(s)=q(t)$.
  This proves that~$\varphi$ is injective.
  Hence, $\varphi$ is a continuous bijection from the compact space~$\bbS^1$ to the Hausdorff space~$X$, so~$\varphi$ is a homeomorphism by \cite[Theorem~17.14]{willard}.
\end{proof}

Finally, we prove \Cref{thm:intro-2f-path} from the introduction.
Let $(X,\tau)$ be a $2$-path-flimsy topological space, let $(X,\cC)$ be the corresponding $2$-flimsy connectivity space (cf.~\Cref{prop:path-is-cspace}), let~$\bot = \bot_\cC$ be the corresponding order-complete dense separation relation on $X$ (cf.~\Cref{def:2f-seprel}, \Cref{prop:2f-seprel}), let $\tau_\rH$ be the order topology of $(X,\bot)$, and let~$\tau_\cF$ be the final topology on~$X$ induced by the paths on~$(X,\tau)$, i.e.,
\[
  \tau_\cF
  :=
  \suchthat{
    V\subseteq X
  }{
    \gamma^{-1}(V)\text{ is open for all paths }\gamma \colon [0,1] \to (X,\tau)
  }.
\]

\begin{proof}[Proof of \Cref{thm:intro-2f-path}]
  The compactness of $(X,\tau)$ is ensured by \Cref{thm:2pf-is-compact}.
  Next, recall from \Cref{cor:seprel-big-circ} that $(X,\tau_\rH)$ is a big circle, so it is in particular Hausdorff by \Cref{prop:big-circle-2f}.
  The same observations as in the proof of \Cref{thm:intro-2f-big_path} (in \Cref{sn:2f-big-path}) show that $\tau_\rH$ is generated by the path-connected components of the subsets $\compb{x,y}$ as $(x,y)$ ranges over all pairs of distinct points of~$X$.
  
  In contrast to the proof of \Cref{thm:intro-2f-big_path}, here we prove that $(X,\tau)$ and $(X,\tau_\rH)$ have the same paths before showing that $(X,\tau_\rH)$ is $2$-path-flimsy.
  Nevertheless, we still adopt the same strategy, namely, showing that $\tau_\rH=\tau_\cF$.
  Indeed, the same arguments as in the proof of \Cref{thm:intro-2f-big_path} yield that the paths on $(X,\tau)$ and $(X,\tau_\cF)$ coincide.
  In particular, $(X,\tau_\cF)$ is $2$-path-flimsy and thus compact by \Cref{thm:2pf-is-compact}.
  Moreover, replacing \Cref{prop:continuum-is-cspace} with \Cref{prop:path-is-cspace} in the proof of \Cref{thm:intro-2f-big_path}, we get that $\tau_\rH\subseteq \tau_\cF$.
  Since we have already shown that~$(X,\tau_\rH)$ is Hausdorff and that~$(X,\tau_\cF)$ is compact, \Cref{lattice_Hausdorff-compact} implies that $\tau_\rH = \tau_\cF$.
  Therefore, $(X,\tau)$ and $(X,\tau_\rH)$ have the same paths and thus the same path-connected subsets.
  This implies that~$(X,\tau_\rH)$ is $2$-path-flimsy.
  
  For the two remaining points, namely the inclusion $\tau \subseteq \tau_\cF = \tau_\rH$ and the uniqueness of $\tau_\rH$, the proofs are just as in the proof of \Cref{thm:intro-2f-big_path} (using \Cref{thm:2pf-is-compact} instead of \Cref{thm:2-big-path-f-is-compact}).
  \end{proof}

\begin{example}
  \label{ex:non-hausdorff-2pf}
  In \Cref{ex:non-hausdorff-2bpf}, we have described a non-Hausdorff topological space $(X,\tau)$ that has the same big paths as the standard circle.
  In particular, they have the same \emph{paths}, so~$(X,\tau)$ is also a non-Hausdorff $2$-path-flimsy space.
  Moreover, the topology $\tau_\rH$ whose existence is ensured by \Cref{thm:intro-2f-path} is again the standard topology on $\bbS^1$---this is consistent with \Cref{thm:hausdorff-2pf-is-circle}.
\end{example}

\section{A triple equivalence}
\label{sn:equi}

Previously, we have discussed three types of circle-like structures: the topological \emph{big circles} (\Cref{def:big-circle}), the order-theoretic \emph{separation relations} (\Cref{def:sep_rel}), and the \emph{$2$-flimsy connectivity spaces} (\Cref{subsn:c-spaces}).
Moreover, we have drawn several links between these frameworks.
In this section, we sum up these connections and show that those different notions are in fact equivalent.
We then apply this equivalence to show \Cref{thm:intro-classification}.

Fix a set~$X$.
We introduce the three following sets:
$\BCir(X)$ is the set of all topologies~$\tau$ on~$X$ such that~$(X,\tau)$ is a big circle (\Cref{def:big-circle}), $\Sepa(X)$ is the set of all separation relations~$\bot$ on~$X$ such that~$(X,\bot)$ is dense and order-complete (\Cref{def:sep_rel,def:dense-complete_seprel}), and $\Conn(X)$ is the set of all connectivities~$\cC$ on~$X$ such that~$(X,\cC)$ is a $2$-flimsy connectivity space (\Cref{def:c-space,def:c-space_flimsy}).

Any big circle is a $2$-flimsy topological space (\Cref{prop:big-circle-2f}), so by \Cref{prop:top-is-cspace}, we obtain a map $\Upsilon \colon \BCir(X) \to \Conn(X)$ mapping a topology~$\tau$ to the set of all connected subsets of~$(X,\tau)$.
(We show in the proof of \Cref{thm:intro-classification} that $\Upsilon(\tau)$ is also the set of all big-path-connected subsets of $(X,\tau)$.)
Similarly, by \Cref{prop:2f-seprel}, we obtain a map $\Phi \colon \Conn(X) \to \Sepa(X)$ mapping a connectivity~$\cC$ to the separation relation $\bot_\cC$ associated with~$(X,\cC)$ (\Cref{def:2f-seprel} and \Cref{prop:2f-seprel}).
Finally, by \Cref{cor:seprel-big-circ}, we obtain a map $\Psi \colon \Sepa(X) \to \BCir(X)$ mapping a separation relation $\bot$ to the order topology of~$(X,\bot)$ (\Cref{def:order-top_seprel}).

\begin{lemma}
  \label{lem:diagram}
  Let $x,y\in X$ be two distinct points.
  Then:
  \begin{itemize}
    \item
      For any topology~$\tau\in\BCir(X)$, the $\Upsilon(\tau)$-components of $\compb{x,y}$ are exactly the connected components of $\compb{x,y}$ with respect to~$\tau$;
    \item
      For any connectivity~$\cC\in\Conn(X)$, the open intervals of $(X,\Phi(\cC))$ between $x$ and $y$ are exactly the $\cC$-components of $\compb{x,y}$;
    \item
      For any separation relation~$\bot\in\Sepa(X)$, the connected components of $\compb{x,y}$ with respect to the topology~$\Psi(\bot)$ are exactly the open intervals of $(X,\bot)$ between~$x$ and~$y$.
  \end{itemize}
\end{lemma}

\begin{proof}
  The first point is by definition of $\Upsilon(\tau)$.
  The second point is easily observed from \Cref{def:open-interval-seprel,def:2f-seprel}.
  To prove the third point, let us fix $z\in\compb{x,y}$.
  By \Cref{def:open-interval-seprel,def:order-top_seprel}, the two open intervals of $(X,\bot)$ between~$x$ and~$y$ partition $\compb{x,y}$ and are open in $(X,\Psi(\bot))$.
  Moreover, they are both non-empty as~$(X,\bot)$ is dense.
  Thus, they each contain at least one connected component of $\compb{x,y}$ for the topology $\Psi(\tau)$.
  But since the connectivity space $(X,\Upsilon(\Psi(\tau)))$ is $2$-flimsy, \Cref{only-two-components} together with the first point imply that $\compb{x,y}$ has exactly two connected components with respect to the topology~$\Psi(\tau)$, and these two components must therefore coincide with the two open intervals between~$x$ and~$y$.
\end{proof}

Using \Cref{lem:diagram}, we are going to prove that the sets $\Conn(X)$, $\Sepa(X)$ and $\BCir(X)$ are in canonical bijection.
More precisely, in the following diagram, any loop around the triangle composes to the corresponding identity map:
\[\begin{tikzcd}[row sep=2.5em]
  & \Conn(X) \arrow[dr,"\Phi"] \\
  \BCir(X) \arrow[ur,"\Upsilon"]  && \Sepa(X) \arrow[ll,"\Psi"]
\end{tikzcd}\]

\begin{theorem}
  \label{thm:diagram}
  The following facts hold:
  \begin{enumroman}
    \item
      For any $\tau \in \BCir(X)$, we have $\Psi(\Phi(\Upsilon(\tau))) = \tau$.
    \item
      For any $\cC \in \Conn(X)$, we have $\cC = \Upsilon(\Psi(\Phi(\cC)))$.
    \item
      For any $\bot \in \Sepa(X)$, we have $\bot = \Phi(\Upsilon(\Psi(\bot)))$.
  \end{enumroman}
\end{theorem}

\begin{proof}
  ~
  \begin{enumroman}
    \item
      Recall from \Cref{prop:big-circle-2f} that any big circle is $2$-flimsy and compact, so by \Cref{thm:intro-2f-top}, its topology is generated by the connected components of the complements of all pairs of distinct points.
      Using all three points of \Cref{lem:diagram}, we see that for distinct $x,y\in X$, the connected components of $\compb{x,y}$ coincide for the two big circles $(X,\tau)$ and $(X,\Psi\circ\Phi\circ\Upsilon(\tau))$.
      Thus, $\Psi\circ\Phi\circ\Upsilon(\tau)=\tau$.
    \item
      Similarly, \Cref{lem:diagram} implies that for any distinct $x,y\in X$ the two components of $\compb{x,y}$ are the same for the two $2$-flimsy connectivity spaces~$(X,\cC)$ and $(X,\Upsilon\circ \Psi\circ \Phi(\cC))$.
      Then, the identity $\Upsilon\circ \Psi\circ \Phi(\cC)=\cC$ is a direct consequence of \Cref{cor:list-of-connected-parts}.
    \item
      To prove the identity $\Phi\circ\Upsilon\circ \Psi(\bot)=\bot$, we need to show that for any four distinct points $a,b,c,d$ of~$X$, we have $\{a,b\}\bot_{\Upsilon\circ\Psi(\bot)} \{c,d\}$ if and only if $\{a,b\}\bot \{c,d\}$ (for non-distinct quadruples of points, this is ensured by \axref{ax:separation-strict}).
      By \Cref{def:2f-seprel} and \Cref{lem:diagram}, this amounts to showing that $\{a,b\}\bot\{c,d\}$ if and only if $c$ and $d$ are in distinct open intervals of $(X,\bot)$ between~$a$ and $b$.
      But since $I_{\not \ni c}(a,b) = \suchthat{x \in \compb{a,b,c}}{\{a,b\} \bot \{c,x\}}$ is one of these two open intervals and does not contain $c$, this is clear.
      \qedhere
  \end{enumroman}
\end{proof}

We now prove \Cref{thm:intro-classification} from the introduction, giving equivalent characterizations of~$\BCir(X)$.

\begin{proof}[Proof of \Cref{thm:intro-classification}]
  The implication \ref{item:intro-classification-big-circle}$\Rightarrow$\ref{item:intro-classification-top-bp} is stated in \Cref{prop:big-circle-2f}.
  The implications \ref{item:intro-classification-top-bp}$\Rightarrow$\ref{item:intro-classification-compact-top} and \ref{item:intro-classification-top-bp}$\Rightarrow$\ref{item:intro-classification-hausdorff-bp} follow directly from \Cref{thm:intro-2f-big_path,thm:intro-2f-top}.
  Now, let us denote by~$\cC$ the set of connected subsets of~$X$ and by~$\cC_{\mathrm{bp}}$ the set of its big-path-connected subsets.
  
  If \ref{item:intro-classification-compact-top} holds, then $(X,\cC)$ is a $2$-flimsy connectivity space by \Cref{prop:top-is-cspace}, and the uniqueness part of \Cref{thm:intro-2f-top} implies that the topology of~$X$ is generated by the connected components of the subsets $\compb{x,y}$ for $\{x,y\} \in \Cho(X)$.
  This means that the topology of~$X$ is $\Psi(\Phi(\cC))$ (cf.~\Cref{def:order-top_seprel} and \Cref{lem:diagram}), so~$X$ is a big circle by \Cref{cor:seprel-big-circ}.
  We have shown \ref{item:intro-classification-compact-top}$\Rightarrow$\ref{item:intro-classification-big-circle}.
  
  The implication \ref{item:intro-classification-hausdorff-bp}$\Rightarrow$\ref{item:intro-classification-big-circle} is proved in the same way, using \Cref{prop:continuum-is-cspace} and \Cref{thm:intro-2f-big_path} to prove that the topology of~$X$ is $\Psi(\Phi(\cC_{\mathrm{bp}})) \in \BCir(X)$.
  This concludes the proof of the equivalence.
  
  Finally, if \ref{item:intro-classification-top-bp} holds, then both $(X,\cC)$ and $(X,\cC_{\mathrm{bp}})$ are $2$-flimsy connectivity spaces.
  Moreover, we have $\cC_{\mathrm{bp}}\subseteq \cC$ by \Cref{cor:bigpathconn-is-conn}.
  By \Cref{cor:comparison_connectivities}, we must then have $\cC=\cC_{\mathrm{bp}}$, i.e., a subset of~$X$ is connected if and only if it is big-path-connected.
\end{proof}

\section{An example of a $2$-path-flimsy space that is not $2$-big-path-flimsy}
\label{sn:2f-path-not-big-path}

In this section, we construct a $2$-path-flimsy topological space $(X,\sT)$ that is not $2$-big-path-flimsy.
The main theorem is \Cref{thm:counter-ex}.

\subsection{Strategy}

\paragraph{Heuristic discussion.}

By \Cref{thm:intro-2f-path,thm:hausdorff-2pf-is-circle}, we know that any $2$-path-flimsy space admits a finer topology making it homeomorphic to the standard circle~$\bbS^1$ with its Euclidean topology~$\sE$.
Hence, we can look without loss of generality for examples whose underlying set~$X$ is~$\bbS^1$, and for a coarser topology $\sT \subseteq\sE$ on~$\bbS^1$ such that~$(\bbS^1,\sT)$ and~$(\bbS^1, \sE)$ share the same path-connected subsets (in particular, $(\bbS^1,\sT)$ is $2$-path-flimsy), but also such that $(\bbS^1,\sT)$ is not $2$-big-path-flimsy.

Since any path-connected subset is also big-path-connected, $X := (\bbS^1,\sT)$ has to be big-path-connected, and so does $X \setminus \{x\}$ for any $x \in X$.
Hence, the fact that $X$ is not $2$-big-path-flimsy can only mean that there are two distinct points $x,y\in X$ such that $X\setminus\{x,y\}$ is big-path-connected (but necessarily \emph{not path-connected} at the same time).
Therefore, there must be a big path $\gamma \colon L \to X\setminus\{x,y\}$ joining the two path-connected components of $X \setminus\{x,y\}$---the map~$\gamma$ has to be continuous with respect to~$\sT$, but cannot be continuous with respect to~$\sE$ as $(\bbS^1, \sE)$ is $2$-big-path-flimsy.

In fact, we will construct the topology~$\sT$ such that a stronger property holds: there will exist an ``\emph{everywhere dense Hamiltonian big path}'' $f \colon L \to \bbS^1$, in the sense that~$f$ is a bijective map between some big interval~$L$ and $\bbS^1$ mapping every non-empty open interval $I \subseteq L$ to a dense subset $f(I)$ of~$(\bbS^1,\sE)$, such that~$f$ is continuous with respect to the topology~$\sT$.
This implies that $X := (\bbS^1, \sT)$ meets our criteria: for any distinct $x,y \in \bbS^1$, if $I \subseteq L$ is the open interval between $f^{-1}(x)$ and~$f^{-1}(y)$, then its image $f(I)$ is a big-path-connected subset of $X \setminus \{x,y\}$, and as a dense subset of $(\bbS^1, \sE)$ it intersects the two path-connected components of $\bbS^1 \setminus \{x,y\}$, as desired.

Now, assume that we have a topology $\sT\subseteq\sE$ and a bijective map $f \colon L \to \bbS^1$ as above, and let $\sF(f) := \suchthat{U\subseteq \bbS^1}{f^{-1}(U)\text{ open in }L}$ be the final topology induced by~$f$.
By hypothesis, $f$ is continuous with respect to~$\sT$, so $\sT \subseteq \sF(f)$.
Since we also have $\sT \subseteq \sE$, the topology $\sT$ is coarser than the topology $\sT' := \sF(f) \cap \sE$.
The coarser a topology, the more paths and path-connected subsets it has.
Since $\sT \subseteq \sT' \subseteq \sE$, and since~$(\bbS^1,\sT)$ and~$(\bbS^1,\sE)$ share the same paths and the same path-connected subsets, so does~$(\bbS^1,\sT')$.
In particular, $(\bbS^1,\sT')$ is $2$-path-flimsy, and since $f$ is continuous with respect to $\sT'$ it is not $2$-big-path-flimsy, as explained above.
Therefore, we can without loss of generality look for examples satisfying $\sT=\sT'=\sE\cap\mathscr{F}(f)$: the construction of the topology~$\sT$ will then directly follow from the construction of an adequate map~$f$.

Thanks to the previous reasoning, we have set ourselves a new goal: construct an everywhere dense Hamiltonian big path $f \colon L \to \bbS^1$ such that $\bbS^1$ equipped with the topology~$\sE\cap\mathscr{F}(f)$ does not admit more paths than with the Euclidean topology~$\sE$.
To avoid creating new paths, we must ensure that no non-trivial path can ``follow the trail mapped out by~$f$'', informally speaking.
To prevent that from happening, the map $f \colon L \to \bbS^1$ shall be chosen to be very chaotic.
Note also that we must have $\card L=\card {\bbS^1} = 2^{\aleph_0}$ as we want~$f$ to be bijective.

\paragraph{Formal construction.}
In the whole section, we see $\bbS^1$ as the quotient of the interval $[0,3]$ obtained by identifying the points~$0$ and~$3$, and we let $\sE$ (the \emph{Euclidean topology} on $\bbS^1$) be the quotient topology inherited from the usual Euclidean topology.
As mere sets, we will also identify $\bbS^1$ with $[0,3)$.

We equip $L := [0,1]^\N$ with the lexicographic order $<_{\mathrm{lex}}$ defined as follows: for any $x=(x_n)_{n\in \N}\in L$ and any $y=(y_n)_{n\in\N}\in L$, set $\ttN(x,y):=\min\suchthat{n\in\N}{x_n\neq y_n}$ with the convention $\min\emptyset=\infty$, and write
\begin{equation}
  \label{super_lexico}
  x\leq_{\mathrm{lex}} y
  \iff
  x=y\text{ or }x_{\ttN(x,y)}<y_{\ttN(x,y)}.
\end{equation}
The following proposition, whose proof is postponed to \Cref{sn:specific_big-interval}, gathers the key properties of the ordered set $(L,<_{\mathrm{lex}})$.

\begin{proposition}
  \label{lexico_order_prop}
  $(L,<_{\mathrm{lex}})$ is a dense order-complete linear order with ends, of cardinality~$2^{\aleph_0}$.
  Equipped with the order topology, the big interval~$L$ is regular Hausdorff, compact (so locally compact and $\sigma$-compact), and first-countable.
  Moreover, its path-connected components are the singletons.
\end{proposition}

In all this section, fix $\ell_1=(\tfrac{1}{3},\tfrac{1}{3},\ldots)$ and $\ell_2=(\tfrac{2}{3},\tfrac{2}{3},\ldots)$, so that $\min L<_\lex \ell_1 <_\lex \ell_2 <_\lex \max L$.
Then, we define the following partition $\{L_1,L_2,L_3\}$ of $L$ into linear continua:
\[
  L_1=[\min L,\ell_1],
  \quad
  L_2=(\ell_1,\ell_2),
  \quad
  L_3=[\ell_2,\max L].
\]

We now present the key tool for the construction of the counterexample.
This is an adaptation of the existence of Bernstein sets (cf.~\cite[Theorem~5.3]{oxtoby2013measure} for background) for products of topological spaces: it provides a partition into rectangles that are ``as pathological as possible''.
Its proof is postponed to \Cref{sn:bernstein}.

\begin{theorem}[Partition by Bernstein rectangles]
  \label{bernstein_rectangle}
  Let $Y$ and $Z$ be two topological spaces that are Hausdorff, locally compact and $\sigma$-compact, first-countable, and have cardinality $2^{\aleph_0}$.
  Then, there exist partitions $\{A_1,A_2,A_3\}$ of~$Y$ and $\{B_1,B_2,B_3\}$ of~$Z$ such that for every closed subset~$F \subseteq Y\times Z$, at least one of the following holds:
  \begin{itemize}
    \item
      $F$ can be covered by countably many \emph{lines}, i.e., subsets of the form $\{y\}\times Z$ or $Y\times\{z\}$ with $y\in Y$ and $z\in Z$;
    \item
      $F$ meets each of the nine subsets $A_i\times B_j$ for $i,j\in\{1,2,3\}$ (the ``rectangles'').
  \end{itemize}
\end{theorem}

In what follows, we fix partitions $\{A_1,A_2,A_3\}$ of~$L$ and $\{B_1,B_2,B_3\}$ of~$\bbS^1$ as provided by \Cref{bernstein_rectangle} for $Y=L$ and $Z=\bbS^1$, in which case the hypotheses follow from \Cref{lexico_order_prop} and from the usual properties of the standard circle.

\begin{lemma}
  \label{bernstein-sets}
  For any $i \in \{1,2,3\}$, the set~$A_i$ (resp.~$B_i$) is a Bernstein set in~$L$ (resp.~in~$\bbS^1$).
\end{lemma}

\begin{proof}
  For any uncountable closed subset $F \subseteq L$, the product $F \times \bbS^1$ is a closed subset of $L\times\bbS^1$ that cannot be covered by countably many lines, so it intersects $A_i \times B_1$, which means that~$F$ intersects~$A_i$.
  The case of $B_i$ is analogous.
\end{proof}

\begin{lemma}
  \label{lem:same-card-rectangles}
  For any $(i,j) \in \{1,2,3\}^2$, we have $\card {A_i \cap L_j} = \card {[i-1,i) \cap B_j} = 2^{\aleph_0}$.
\end{lemma}

\begin{proof}
  Observe that for any $s\in(\tfrac{j-1}{3},\tfrac{j}{3})$, the set $\{s\}\times[0,1]^{\N^*}$---which is the closed interval of~$L$ between $(s,0,0,\ldots)$ and $(s,1,1,\ldots)$---is an uncountable closed subset of~$L_j$, and hence intersects $A_i$ by \Cref{bernstein-sets}.
  Since these sets are disjoint for various values of $s$, since there are $\card{(\tfrac{j-1}{3},\tfrac{j}{3})}=2^{\aleph_0}$ possible values of $s$, and since we also have an obvious upper bound, this implies $|A_i\cap L_j|=2^{\aleph_0}$.

  Similarly, any non-empty open interval of~$(0,3)$ contains $2^{\aleph_0}$-many disjoint copies of the Cantor set: indeed, it contains a dilated and translated copy of the Cantor ternary set, which is homeomorphic to $\{0,1\}^{\N}$ and so to $\{0,1\}^{\N}\times \{0,1\}^{\N}=\bigsqcup_{\omega\in \{0,1\}^{\N}}\, \{\omega\}\times \{0,1\}^{\N}$ (see \cite[Chapter~XVI.8]{kuratowski} for details).
  Each one of these copies is an uncountable closed subset of $\bbS^1$ and so meets $B_j$ by \Cref{bernstein-sets}.
  Thus, $[i-1,i)\cap B_j$ also has cardinality $2^{\aleph_0}$.
\end{proof}

For all $(i,j) \in \{1,2,3\}^2$, pick arbitrary bijections $f_{i,j} \colon A_i\cap L_j \to [i-1,i)\cap B_j$, which is possible as these two sets have the same cardinality by \Cref{lem:same-card-rectangles}.
Then, define a map $f \colon L \to \bbS^1$ by
\[
  \forall\ell\in L, \quad
  f(\ell) := f_{i,j}(\ell) \quad
  \text{if }\ell\in A_i\cap L_j,
\]
and observe that~$f$ is a bijection from~$L$ to~$\bbS^1$, satisfying $f(L_j) = B_j$ and $f(A_i) = [i-1,i)$.
Finally, as explained above, we define the topology $\sT := \sE \cap \sF(f)$ on~$\bbS^1$.
More concretely:
\begin{equation}
  \label{topo_counter-ex}
  \sT
  :=
  \suchthat{
    U \subseteq \bbS^1
  }{
    U \textnormal { is open in } (\bbS^1, \sE), \textnormal{ and }
    f^{-1}(U) \textnormal{ is open in } L
  }.
\end{equation}
Let~$X$ be the topological space $(\bbS^1,\sT)$.
In what follows, we will be careful to only use~$\bbS^1$ to denote the topological space~$(\bbS^1, \sE)$ equipped with the Euclidean topology (or its underlying set), and to systematically use~$X$ when the topology is~$\sT$.
When talking about a map $[0,1] \to X$ (which is also a map $[0,1] \to \bbS^1$), we say that it is \emph{$\sE$-continuous} or \emph{$\sT$-continuous} when it is continuous with respect to the topologies~$\sE$ or~$\sT$, respectively.
Note that~$f$ is a ($\sT$-continuous) big path on~$X$, as $L$ is a big interval by \Cref{lexico_order_prop}.
The space~$X$ will be our desired counterexample:

\begin{theorem}
  \label{thm:counter-ex}
  $X$ is $2$-path-flimsy, but $X\setminus\{x,y\}$ is big-path-connected for all distinct $x,y\in X$.
\end{theorem}

We prove the two assertions of \Cref{thm:counter-ex} separately, starting with the second, easier one.

\begin{proposition}
  \label{prop:counter-ex_big-path-compo}
  For any $x,y\in X$, the space $X\setminus \{x,y\}$ is big-path-connected.
\end{proposition}

\begin{proof}
  Because $L_1,L_2,L_3$ are disjoint, there is $i\in\{1,2,3\}$ such that~$L_i$ contains neither~$f^{-1}(x)$ nor~$f^{-1}(y)$.
  By construction, $f(L_i)=B_i$, so~$B_i$ is a big-path-connected subset of~$X$ as the $\sT$-continuous image of a big-path-connected space (cf.~\Cref{cor:bigpathconn-is-conn}), and avoids $\{x,y\}$.
  Let~$C_1$ and~$C_2$ be the two path-connected components of $\bbS^1\setminus \{x,y\}$ with respect to the Euclidean topology~$\sE$, which are in particular big-path-connected subsets of~$X$.
  For any $j \in \{1,2\}$, $C_j \cup \{x,y\}$ is an uncountable closed subset of $(\bbS^1, \sE)$, so the sets $C_j \cup \{x,y\}$ intersect~$B_i$ by \Cref{bernstein-sets}.
  Since $x,y \notin B_i$, this means that~$C_j \cap B_i \neq \emptyset$, so $C_j \cup B_i$ is big-path-connected by \Cref{prop:continuum-is-cspace} (\axref{ax:union-connected}), and then the union $C_1 \cup C_2 \cup B_i = X \setminus \{x,y\}$ is also big-path-connected.
\end{proof}

In particular, \Cref{prop:counter-ex_big-path-compo} implies that~$X$ is not $2$-big-path-flimsy.
In order to prove \Cref{thm:counter-ex}, it remains to prove that~$X$ is $2$-path-flimsy, which follows from the following result:

\begin{proposition}
  \label{prop:path_on_counter-ex}
  The space~$X$ has the same paths as the standard circle $(\bbS^1,\sE)$.
\end{proposition}

The proof of \Cref{prop:path_on_counter-ex} represents most of the remaining work and is divided across the two following subsections.
In \Cref{sn:reparametrization}, we show that any path can be reparametrized to prevent it from ``halting'' (\Cref{reparametrization}).
In \Cref{sn:non-euclidean_path_stay_constant}, we prove that any path on~$X$ that is not $\sE$-continuous must ``halt'', which together with \Cref{reparametrization} implies \Cref{prop:path_on_counter-ex}.
This will conclude the proof of \Cref{thm:counter-ex}.

\subsection{Reparametrization of paths into light paths}
\label{sn:reparametrization}

In this subsection, we prove \Cref{reparametrization}, which is a reparametrization result in the spirit of the ``monotone-light factorization'' (compare to \cite[Theorem~6]{collins}).

\begin{definition}
  Let~$Y$ be a topological space.
  We say that a continuous map $g \colon [0,1] \to Y$ is \emph{light} if it is not constant on any non-empty open interval, i.e., if its fibers are all totally disconnected.
\end{definition}

For us, the importance of this notion comes from the fact that proving \Cref{prop:path_on_counter-ex} will be easier for light paths.
The following fact allows us to reduce the general case to the light case:

\begin{proposition}
  \label{reparametrization}
  Let~$Y$ be a $\rT_1$ topological space, and let $g \colon [0,1] \to Y$ be a non-constant continuous map on~$Y$.
  Then, there exist two continuous maps $\tilde{g} \colon [0,1] \to Y$ and $\phi \colon [0,1] \to [0,1]$ such that $g = \tilde{g}\circ\phi$ and such that $\tilde g$ is light.
\end{proposition}

A key tool is the following lemma:

\begin{lemma}
  \label{lem:quotient-is-reparam}
  Let $\sim$ be an equivalence relation on $[0,1]$ whose equivalence classes are closed intervals.
  Let $I$ be the quotient $[0,1]/\sim$, and let $q \colon [0,1] \twoheadrightarrow I$ be the natural surjection.
  Then:
  \begin{itemize}
    \item
      $I$ is equipped with a natural strict linear order~$\prec$, making $q$ into an order-preserving surjection.
    \item
      The order topology of $(I, \prec)$ coincides with the quotient topology on~$I$ inherited from the Euclidean topology on $[0,1]$.
    \item
      If $\sim$ has at least two equivalence classes, then $(I, \prec)$ is order-isomorphic to $[0,1]$.
  \end{itemize}
\end{lemma}

\begin{proof}
  An element of $I$ is a closed interval $[a,b]$, and these intervals do not overlap: if $[a,b] \neq [a',b']$ are two elements of $I$, then either $b < a'$ or $b' < a$.
  Hence, we can define a linear order $\preceq$ on $I$ by:
  \[
    [a,b] \preceq [a',b']
    \iff
    b \leq a'
    \quad
    \textnormal{(equivalently, $a \leq b'$)}.
  \]
  The surjection $q$ is order-preserving, because if $s,t \in [0,1]$ belong respectively to the equivalence classes $[a,b]$ and $[a',b']$ and satisfy $s \leq t$, then we have $a \leq s \leq t \leq b'$, so $[a,b] \preceq [a',b']$.
  We let $\prec$ be the corresponding strict linear order.
  
  Let~$\tau_q$ be the quotient topology on $I$ 
   and let~$\tau_o$ be the order topology of~$(I,\prec)$.
  As a linearly ordered topological space, $(I, \tau_o)$ is Hausdorff.
  As a quotient of a compact space, $(I, \tau_q)$ is compact.
  For any open interval $U$ of $(I, \prec)$ of the form $\bigl([a,b],[c,d]\bigr)$, its preimage $q^{-1}(U) = (b,c)$ is open in~$[0,1]$.
  Similarly, using the notation of \Cref{sn:prelim_linear}, the preimages of the open intervals $\bigl(-\infty,[a,b]\bigr)$ and $\bigl([a,b],+\infty\bigr)$ are respectively $(-\infty,a)$ and $(b,+\infty)$, which are open in~$[0,1]$.
  Therefore, $\tau_o \subseteq \tau_q$, and by \Cref{lattice_Hausdorff-compact} we then have $\tau_o = \tau_q$.

  Since $\tau_o = \tau_q$, the space $(I, \tau_o)$ is a quotient of $[0,1]$, so it is compact, connected, and separable just like $[0,1]$.
  Assuming that $\card I \geq 2$, this implies that the order~$\prec$ is dense and order-complete (corresponding to connectedness, cf.~\cite[26G]{willard}), has a minimum and a maximum (corresponding to compactness, cf.~\iref{prop:continuum-facts}{item:continuum-compact-subsets}), and admits a dense countable subset.
  By Cantor's isomorphism theorem \cite[\S11]{cantor}, this implies that~$(I, \prec)$ is order-isomorphic to $[0,1]$.
\end{proof}

\begin{proof}[Proof of \Cref{reparametrization}]
  Define an equivalence relation $\sim$ as follows: for any $s,t\in [0,1]$ with $s\leq t$,
  \[
    s \sim t
    \iff
    t\sim s
    \iff
    g|_{[s,t]}
    \textnormal{ is constant.}
  \]
  Since the connected subsets of $[0,1]$ are exactly the intervals, we see that $s\sim t$ if and only if there is $y\in Y$ such that $s$ and $t$ belong to the same connected component of $g^{-1}(\{y\})$.
  In other words, the equivalence classes of~$\sim$ are the connected components of the fibers of $g$.
  In particular, each equivalence class is an interval and a closed subset of $g^{-1}(\{y\})$ for some $y\in Y$, which is itself closed in~$[0,1]$ since $Y$ is $\rT_1$ and $g$ is continuous.
  Therefore, all equivalence classes of~$\sim$ are closed intervals.

  Define the quotient space $I := [0,1]/\sim$.
  Since~$g$ is non-constant, we have $\card I \geq 2$.
  By \Cref{lem:quotient-is-reparam}, $I$ admits a homeomorphism $\eta \colon I \to [0,1]$ such that its composition with the projection $q \colon [0,1] \twoheadrightarrow I$ is a non-decreasing continuous map $\varphi := \eta \circ q \colon [0,1] \to [0,1]$.
  Moreover, since $g$ is constant on every equivalence class of $\sim$ by construction, $g$ factors through the projection $q$ via a continuous map $\bar g \colon I \to Y$.
  Define $\tilde g := \bar g \circ \eta^{-1}$, which is a continuous map $[0,1] \to Y$ satisfying $\tilde g \circ \varphi = \bar g \circ \eta^{-1} \circ \eta \circ q = \bar g \circ q = g$.
  Checking that the path~$\tilde g$ is light amounts to checking that~$\bar g$ is not constant on any non-empty open interval of~$I$, which follows by maximality of the equivalence classes.
  Indeed, if~$\bar g$ were constant on a non-empty open interval~$J$ of~$I$, then~$q^{-1}(J)$ would be an open interval of~$[0,1]$ (since~$q$ is continuous and order-preserving) that is not mapped to a single equivalence class of~$\sim$ but on which~$g$ is constant, contradicting the definition of~$\sim$.
\end{proof}

\subsection{Proof of \Cref{prop:path_on_counter-ex}}
\label{sn:non-euclidean_path_stay_constant}

The goal of this subsection is to prove \Cref{prop:path_on_counter-ex}: we show that the space $X = (\bbS^1, \sT)$ (the circle equipped with the topology defined in \Cref{topo_counter-ex}) has the same paths as the standard circle~$(\bbS^1, \sE)$.
In order to apply \Cref{reparametrization}, which lets us reduce the proof of \Cref{prop:path_on_counter-ex} to the case of light paths, we first prove that~$X$ is $\rT_1$:

\begin{proposition}
  \label{counter-ex_T1}
  The space~$X$ is $\rT_1$.
\end{proposition}

\begin{proof}
  Recall that~$(\bbS^1,\sE)$ and~$L$ are~$\rT_1$, as they are Hausdorff.
  By definition of~$\sT$ (\Cref{topo_counter-ex}), checking that a singleton $\{x\}$ is closed in $X$ means checking that $\{x\}$ is closed in $(\bbS^1, \sE)$ (which is true) and that $f^{-1}(\{x\})$, which is a singleton as $f$ is bijective, is closed in~$L$ (which is again true).
\end{proof}

We first prove three lemmas, gathering more and more information about the topology~$\sT$ and about the local behavior of paths on~$X$.
We will use the notations $x_n \convE x$ and $x_n \convT x$ to denote convergence of sequences in~$\bbS^1 = (\bbS^1, \sE)$ and in $X = (\bbS^1, \sT)$, respectively.

\begin{lemma}
  \label{alternative_sequence}
  Let $x,y\in \bbS^1$, let $\ell\in L$, and let $x_n\in \bbS^1$ for all $n\in \N$.
  If $x_n \convT x$, if $x_n \convE y$, and if $f^{-1}(x_n) \longrightarrow \ell$ in $L$, then $x=y$ or $x=f(\ell)$.
\end{lemma}

\begin{proof}
  Let $N\in\N$.
  The spaces~$\bbS^1$ and~$L$ are Hausdorff, so their respective subsets $\{y\}\cup\suchthat{x_n}{n\geq N}$ and $\{\ell\}\cup\suchthat{f^{-1}(x_n)}{n\geq N}$ are closed.
  Since the singletons of~$\bbS^1$ and $L$ are closed, it follows that \[S_N := \{f(\ell)\}\cup \{y\}\cup\suchthat{x_n}{n\geq N}\] is closed in $\bbS^1$, and that $f^{-1}(S_N)=\{f^{-1}(y)\}\cup \{\ell\}\cup\suchthat{f^{-1}(x_n)}{n\geq N}$ is closed in $L$.
  By definition of $\sT$ (\Cref{topo_counter-ex}), this means that~$S_N$ is a closed subset of~$X$.
  Therefore, the convergence $x_n \convT x$ implies that $x \in S_N$.
  This holds for any $N\in\N$, therefore either $x=f(\ell)$, or $x=y$, or the sequence $(x_n)$ takes the value $x$ infinitely many times.
  However, in the last case, $x$ is also the limit of a subsequence of $(x_n)$ for the Euclidean topology, and then $x = y$ since~$\bbS^1$ is Hausdorff.
\end{proof}

\begin{lemma}
  \label{alternative-continuity_pointwise}
  Let $g \colon [0,1] \to X$ be a $\sT$-continuous map, and let $s\in[0,1]$.
  If $U\subseteq \bbS^1$ and $V\subseteq L$ are open neighborhoods of~$g(s)$ and of~$f^{-1}(g(s))$ respectively, then there exists an open neighborhood $W\subseteq [0,1]$ of~$s$ such that, for all $t\in W$, we have $g(t)\in U$ or $f^{-1}(g(t))\in V$.
\end{lemma}

\begin{proof}
  Assume, by way of contradiction, that there is a sequence $(s_n)\in[0,1]^\N$ such that $s_n \longrightarrow s$, $g(s_n)\notin U$ and $f^{-1}(g(s_n))\notin V$ for all $n \in \N$.
  Since~$\bbS^1$ and~$L$ are sequentially compact (this is well-known for~$\bbS^1$, cf.~\Cref{bigint-seqcomp} or \Cref{lexico_order_prop} for~$L$), replacing~$(s_n)$ by a subsequence if needed, we can assume that $g(s_n) \convE x$ and that $f^{-1}(g(s_n)) \longrightarrow \ell$ in~$L$.
  Since $U$ and $V$ are open, we have $x\notin U$ and $\ell\notin V$.
  Furthermore, by $\sT$-continuity of $g$, we have $g(s_n) \convT g(s)$.
  By \Cref{alternative_sequence}, we have $g(s)=x$ or $f^{-1}(g(s))=\ell$.
  The first case is impossible because $g(s)\in U$, and the second is impossible because $f^{-1}(g(s))\in V$.
  This is a contradiction.
\end{proof}

\begin{lemma}
  \label{alternative-continuity_uniform}
  Let $g \colon [0,1] \to X$ be a $\sT$-continuous map, and let $s\in[0,1]$.
  Let $U \subseteq \bbS^1$, $V \subseteq L$, and $W \subseteq [0,1]$ be open neighborhoods of $g(s)$, $f^{-1}(g(s))$, and $s$, respectively.
  Then:
  \begin{enumroman}
    \item
      \label{alternative-continuity_uniform1}
      If~$g$ is not $\sE$-continuous at~$s$, then there exists a non-empty open interval $I\subseteq W$ such that $f^{-1}(g(I))\subseteq V$.
    \item
      \label{alternative-continuity_uniform2}
      If the map $f^{-1}\circ g \colon [0,1] \to L$ is not continuous at~$s$, then there exists a non-empty open interval $I\subseteq W$ such that $g(I)\subseteq U$.
  \end{enumroman}
\end{lemma}

\begin{proof}
  First, we assume that~$g$ is not $\sE$-continuous at~$s$, and we prove \ref{alternative-continuity_uniform1}.
  Then, there is an open neighborhood~$U'$ of~$g(s)$ in $\bbS^1$ such that~$s$ is not in the interior of~$g^{-1}(U')$.
  Since the Euclidean topology on $\bbS^1$ is regular, we can choose two disjoint open subsets~$U_{\mathrm{in}}$ and~$U_{\mathrm{out}}$ of~$\bbS^1$ such that $g(s)\in U_{\mathrm{in}}$ and $\bbS^1\setminus U'\subseteq U_{\mathrm{out}}$.
  By \Cref{alternative-continuity_pointwise}, there is an open neighborhood~$W_1$ of~$s$ in $[0,1]$ such that $g(t)\in U_{\mathrm{in}}$ or $f^{-1}(g(t))\in V$ for all $t\in W_1$.
  
  Note that~$W\cap W_1$ is an open neighborhood of~$s$.
  Because~$s$ is not in the interior of $g^{-1}(U')$, there exists $\sigma\in W\cap W_1$ such that $g(\sigma)\notin U'$.
  In particular, $U_{\mathrm{out}}$ is an open neighborhood of $g(\sigma)$ in~$\bbS^1$.
  Moreover, $g(\sigma)\notin U_{\mathrm{in}}$ because $U_{\mathrm{in}}$ and $U_{\mathrm{out}}$ are disjoint, so $f^{-1}(g(\sigma))\in V$ by the properties of~$W_1$.
  Thus, $V$ is an open neighborhood of $f^{-1}(g(\sigma))$ in $L$.
  By \Cref{alternative-continuity_pointwise}, there is then an open neighborhood~$W_2$ of~$\sigma$ in~$[0,1]$ such that $g(t)\in U_{\mathrm{out}}$ or $f^{-1}(g(t))\in V$ for all $t\in W_2$.
  
  The set~$W\cap W_1\cap W_2$ is an open neighborhood of~$\sigma$ in~$[0,1]$, so there exists an open interval $I\subseteq W\cap W_1\cap W_2\subseteq W$ that contains $\sigma$.
  We claim that $f^{-1}(g(t))\in V$ for any $t\in I$.
  Indeed, $U_{\mathrm{in}}$ and~$U_{\mathrm{out}}$ are disjoint, so $g(t)\notin U_{\mathrm{in}}$ or $g(t)\notin U_{\mathrm{out}}$.
  But since $t\in W_1\cap W_2$, we can use at least one of the properties defining~$W_1$ and~$W_2$ to prove that~$f^{-1}(g(t))\in V$.
  This concludes the proof of \ref{alternative-continuity_uniform1}.

  The proof of \ref{alternative-continuity_uniform2} is done in the exact same way (using the regularity of~$L$ given by \Cref{lexico_order_prop} instead of that of~$\bbS^1$).
  We omit the details.
\end{proof}

\begin{proof}[Proof of \Cref{prop:path_on_counter-ex}]
  Since $\sT\subseteq \sE$ by construction, any path on $(\bbS^1,\sE)$ is $\sT$-continuous.
  To prove the converse, we fix a $\sT$-continuous map $g \colon [0,1] \to X$, and our goal is to prove that it is $\sE$-continuous.

  If~$g$ is constant, then this is clear.
  Otherwise, since $X$ is $\rT_1$ by \Cref{counter-ex_T1}, we can fix a light path~$\tilde g$ and a map~$\phi$ as given by \Cref{reparametrization}.
  If $\tilde g$ is $\sE$-continuous, then so is $g = \tilde g \circ \phi$ as a composition of continuous maps.
  This shows that it suffices to deal with light paths.
  Without loss of generality, we now assume that~$g$ is light.
  
  We assume, by way of contradiction, that there exists an $s \in [0,1]$ such that~$g$ is not $\sE$-continuous at~$s$.
  Because~$L_1$ and~$L_3$ are disjoint, there exists $i\in\{1,3\}$ such that $f^{-1}(g(s))\notin L_i$.
  Since~$L_i$ is closed, \Cref{alternative-continuity_uniform} yields a non-empty open interval~$I \subseteq [0,1]$ such that $f^{-1}(g(I))\cap L_i=\emptyset$.
  If $f^{-1}\circ g$ is continuous on~$I$, then $f^{-1}(g(I))$ is a non-empty path-connected subset of~$L$, hence a singleton by \Cref{lexico_order_prop}, so~$g$ is constant on~$I$, contradicting its lightness.
  For the rest of the proof, we therefore assume that $f^{-1}\circ g$ is not continuous on~$I$.
  
  By this new assumption, there is $t \in I$ such that $f^{-1} \circ g$ is not continuous at~$t$.
  Even when they are viewed as subsets of $\bbS^1=[0,3]/\{0,3\}$ via the identification of the points~$0$ and~$3$, no point belongs to all three sets $[0,1]$, $[1,2]$, and $[2,3]$, so there exists $j\in\{1,2,3\}$ such that $g(s)\notin [j-1,j]\subseteq \bbS^1$.
  Since~$[j-1,j]$ is closed in~$\bbS^1$, \Cref{alternative-continuity_uniform} yields a closed interval $J\subseteq I$ with positive length such that $g(J)\cap [j-1,j]=\emptyset$.
  Note that we also have $f^{-1}(g(J))\cap L_i =\emptyset$.

  Now, define $F$ as the closure in $L\times\bbS^1$ of $\suchthat{\big( f^{-1}(g(r)) ,\, g(r) \big)}{r\in J}=(f^{-1}\circ g, \, g)(J)$.
  Our goal is to show that~$F$ does not meet~$A_j\times B_i$.
  By contradiction, let $(\ell,x)\in F\cap (A_j\times B_i)$.
  Observe from \Cref{lexico_order_prop} that $L \times \bbS^1$ is first-countable as a product of two first-countable topological spaces.
  Thus, there is a sequence $(r_n)_{n\in\N}$ of points of~$J$ such that $g(r_n)\convE x$ and such that $f^{-1}(g(r_n)) \longrightarrow \ell$ in~$L$.
  By sequential compactness of the bounded closed real interval~$J$, we can assume without loss of generality that $r_n \longrightarrow r\in J$, so that $g(r_n) \convT g(r)$ by $\sT$-continuity of~$g$.
  By \Cref{alternative_sequence}, we then have $g(r)=x$ or $f^{-1}(g(r))=\ell$.
  If $g(r)=x$, then $g(r)\in B_i$ and so $f^{-1}(g(r))\in L_i$ by definition of~$f$, contradicting $f^{-1}(g(J))\cap L_i=\emptyset$.
  If $f^{-1}(g(r))=\ell$, then $f^{-1}(g(r))\in A_j$ and so $f(\, f^{-1}(g(r))\, )=g(r)\in [j-1,j)$, contradicting $g(J)\cap [j-1,j]=\emptyset$.
  Since both cases lead to a contradiction, we have shown that~$F$ does not meet $A_i\times B_j$.
  
  By choice of~$A_i$ and~$B_j$ (cf.~\Cref{bernstein_rectangle}), this implies that $F$, and $(f^{-1}\circ g,g)(J)$ \emph{a fortiori}, can be covered by countably many lines.
  In other words, there are two countable sets $D\subseteq L$ and $D'\subseteq \bbS^1$ such that $g(J)\subseteq f(D)\cup D'$, so~$g(J)$ is at most countable.
  Now, recall that $g \colon [0,1] \to X$ is $\sT$-continuous and that~$X$ is $\rT_1$ (\Cref{counter-ex_T1}), so $g^{-1}(\{x\})$ is closed in $[0,1]$ for all $x\in X$.
  Hence, the decomposition
  \[
    J
    =
    \bigcup_{x\in g(J)}
      J \cap g^{-1}(\{x\})
  \]
  expresses the compact real interval $J$ as an at-most-countable union of disjoint non-empty closed sets.
  By a classic theorem of Sierpiński~\cite{sierpinski} (see also~\cite[Theorem 6.1.27]{engelking}), this is only possible when $|g(J)|=1$, so~$g$ is constant on~$J$, contradicting its lightness.
\end{proof}

\subsection{The lexicographic order on sequences of real numbers}
\label{sn:specific_big-interval}

Recall from \Cref{super_lexico} that $L=[0,1]^{\N}$ is endowed with the lexicographic order $<_{\mathrm{lex}}$.
Also recall that for any $x,y\in L$, we write $\ttN(x,y)=\min\suchthat{n\in\N}{x_n\neq y_n}$.
Here, we prove \Cref{lexico_order_prop}.

\begin{proof}[Proof of \Cref{lexico_order_prop}]
  That the lexicographic order is a total order is standard.
  The cardinality of~$L$ is $\card L = \card{[0,1]}^{\card \N} = (2^{\aleph_0})^{\aleph_0} = 2^{\aleph_0^2} = 2^{\aleph_0}$.
  The elements $(0,0,\ldots)$ and $(1,1,\ldots)$ are respectively the minimum and maximum of $(L,<_{\mathrm{lex}})$.

  In what follows, if $x \in L$, we denote by~$x_n \in [0,1]$ its $n$-th coordinate.

  Let $x,y \in L$ be such that $x<_\lex y$.
  Then, we have $x<_\lex z<_\lex y$ where $z=(\frac{x_n+y_n}{2})_{n\in \N}$.
  Indeed, $\ttN(x,z)=\ttN(y,z)=\ttN(x,y)$ and $x_{\ttN(x,y)}<z_{\ttN(x,y)}<y_{\ttN(x,y)}$.
  Thus, $(L,<_{\mathrm{lex}})$ is dense.

  Let $A \subseteq L$.
  Using the convention $\sup\emptyset = 0$, recursively define $(m_n)_{n\in\N}$ as follows:
  \begin{align}
      \nonumber
      m_0 &:=  \sup \suchthat{ x_0 }{ x \in A },\\
      \label{formula_sup-lex_step}
      m_{k+1} &: = \sup\suchthat{ x_{k+1} }{ x \in A\textnormal{ such that }x_i=m_i\text{ for all }0\leq i\leq k }.
  \end{align}
  We claim that~$m = (m_n)_{n\in\N} \in L$ is the least upper bound of~$A$.
  Indeed, if $x \in A$ is distinct from~$m$, then we have $x_i=m_i$ for all $0\leq i\leq \ttN(x,m)-1$ by definition of $\ttN(x,m)$, and so $x_{\ttN(x,m)}\leq m_{\ttN(x,m)}$ by definition of $m$.
  Since $x_{\ttN(x,m)}\neq m_{\ttN(x,m)}$, we get that $x<_\lex m$.
  Conversely, let $y \in L$ be an upper bound of $A$ distinct from $m$.
  For all $x \in A$ such that $\ttN(x,m)\geq\ttN(y,m)$, we have $\ttN(x,y)\geq \min\bigl(\ttN(x,m),\ttN(y,m)\bigr)=\ttN(y,m)$, and so $x_{\ttN(y,m)}\leq y_{\ttN(y,m)}$ because $x\leq_\lex y$.
  It follows that $y_{\ttN(y,m)}\geq \sup \suchthat{ x_{\ttN(y,m)} }{ x \in A\textnormal{ such that }\ttN(x,m)\geq \ttN(y,m) }=m_{\ttN(y,m)}$ and so $m<_\lex y$.
  We have proved that $(L,<_{\mathrm{lex}})$ is order-complete and that
  \begin{equation}
    \label{formula_sup-lex}
    m=\sup A.
  \end{equation}

  By \Cref{prop:continuum-facts}, $L$ is then compact, connected, and Hausdorff.
  In particular, $L$ is locally compact and $\sigma$-compact, but also regular (and even normal, see~\cite[Theorem~32.3]{munkres}).

  \medskip

  To prove that~$L$ is first-countable, it suffices to prove that for any~$x\in L$, if $x\neq \min L$ (resp.~if $x\neq\max L$), then there are sequences $(x^k)_{k \geq 1}$ (resp.~$(y^k)_{k\geq 1}$) of points of~$L \setminus \{x\}$ such that $\sup_{k\geq 1} x^k=x$ (resp.~$\inf_{k\geq 1} y^k=y$).
  Indeed, the open intervals $(x^k,y^k)$ for all $k \geq 1$ then form a countable local base of~$x$, where we have let $x^k=-\infty$ if $x = \min L$ and $y^k=+\infty$ if $x=\max L$.
  In fact, since $L = (L, <_\lex)$ is order-isomorphic to its opposite $(L,>_\lex)$ via $(z_n)_{n\in\N} \mapsto (1-z_n)_{n\in\N}$, the two cases are symmetrical and we only need to find~$(x^k)_{k\geq 1}$ for $x\neq\min L$.

  If there is a $\mathtt J \in \N$ such that $x_{\mathtt J} > 0$ but $x_n = 0$ for all $n > \mathtt J$, then we let $x^k_n := x_n$ if $n \neq \mathtt J$, and $x^k_{\mathtt J} := (1-\tfrac1k)x_{\mathtt J}$.
  In that case, the sequence $(x^k)_{k \in \N}$ is strictly increasing.

  Otherwise, we let $x_n^k := x_n$ if $n < k$, and $x_n^k := 0$ if $n \geq k$.
  This time, the sequence~$(x^k)$ is non-decreasing, and $x^k <_\lex x$ for all $k \geq 1$ as~$x$ has infinitely many non-zero coordinates.

  In both cases, using the explicit formula for the supremum given by \Cref{formula_sup-lex_step,formula_sup-lex}, we see that $\sup x^k = x$, concluding the proof that~$L$ is first-countable.

  \medskip

  Let $\gamma \colon [0,1] \to L$ be continuous.
  As a continuous image of a compact connected set, $\gamma\bigl([0,1]\bigr)$ is a compact and connected subset of $L$.
  Thus, by \iref{prop:continuum-facts}{item:continuum-compact-subsets}, there are two elements $x=(x_n)_{n\in\N}$ and $y=(y_n)_{n\in\N}$ of $L$ such that $x\leq_\lex y$ and $\gamma\bigl([0,1]\bigr)=\suchthat{ z\in L }{ x\leq_\lex z\leq_\lex y }$.
  Assume, by way of contradiction, that $x_n<y_n$ for some $n\in \N$.
  Then, for any $s\in(x_n,y_n)$, the non-empty open set 
  \[
    I_n(s)
    :=
    \suchthat{
      z\in L
    }{
      (x_0,\ldots,x_{n-1},s,0,0,0,\ldots)
      <_\lex
      z
      <_\lex
      (x_0,\ldots,x_{n-1},s,1,1,1,\ldots)
    }
  \]
  is contained in $\gamma\bigl([0,1]\bigr)$.
  Moreover, the subsets $I_n(s)$ are disjoint for distinct values of $s\in(x_n,y_n)$, so $\suchthat{ \gamma^{-1}(I_n(s)) } { s \in(x_n,y_n) }$ is an uncountable family of disjoint non-empty open subsets of~$[0,1]$, but this cannot exist since~$[0,1]$ is second-countable.
  By contradiction, we have shown that $x=y$.
  Thus, all paths on~$L$ are constant and the path-connected components of~$L$ are the singletons.
\end{proof}

\appendix
\crefalias{section}{appendix}

\section{Proof of \axref{ax:button-zip} for big-path-connected subsets}
\label{sec:proof_C4}

In this appendix, we complete the proof of \Cref{prop:continuum-is-cspace}.
Let~$X$ be a topological space, and let~$\cC$ be the set of all its big-path-connected subsets.
We are going to show that $(X,\cC)$ satisfies \axref{ax:button-zip}.

\begin{proof}
  Pick non-empty sets $A,B \in \cC$ such that $A \cup B \in \cC$.
  If there is a big path $f\colon J \to A \cup B$ such that $f(J\setminus\{\max J\})\subseteq A$ and $f(\max J)\in B$, then both $A\cup\{f(\max J)\}$ and $B\cup\{f(\max J)\} = B$ are big-path-connected, completing the proof (case~\ref{ax:button-zip-button} holds).
  Therefore, we can exclude this case and assume the following:
  \begin{equation}
    \label{right-end_in_A}
    \text{for any big path $f\colon J \to A\cup B$, if $f\bigl(J\setminus\{\max J\}\bigr)\subseteq A$, then $f(\max J)\in A$.}
  \end{equation}
  Symmetrically, swapping the roles of~$A$ and~$B$, we can assume:
  \begin{equation}
    \label{right-end_in_B}
    \text{for any big path $f\colon J \to A\cup B$, if $f\bigl(J\setminus\{\max J\}\bigr)\subseteq B$, then $f(\max J)\in B$.}
  \end{equation}
  If case~\ref{ax:button-zip-seam} of \axref{ax:button-zip} ever happens, then we are done, so we can also assume:
  \begin{equation}
    \label{no-seam}
    \text{if $S \subseteq A\cup B$ is big-path-connected, then $S \cap A$ and $S \cap B$ are big-path-connected.}
  \end{equation}

  Fix arbitrary points $a\in A$ and $b\in B$, and let $\gamma \colon L \to A \cup B$ be a big path from~$a$ to~$b$ in~$A\cup B$.
  Let $s_A := \sup \gamma^{-1}(A)$.
  We have $\gamma\bigl((s_A, \max L]\bigr) \subseteq B$, so \eqref{right-end_in_B} implies that $\gamma(s_A) \in B$.
  Without loss of generality, replacing $b$ with $\gamma(s_A)$, $L$ with $[\min L, s_A]$, and $\gamma$ with the (reparametrized) restricted big path~$\gamma|_{[\min L, s_A]}$, we can assume that $s_A = \max L$.
  Hence,
  \begin{equation}
    \label{A-cofinal}
    \text{for any $\ell \in L \setminus \{\max L\}$, there exists $\ell' > \ell$ such that $\gamma(\ell') \in A$.}
  \end{equation}

  Fix an ordinal number~$\kappa$ of cardinality strictly larger than that of~$L$ (e.g., the smallest ordinal number of cardinality $2^{\card L}$, by Cantor's theorem).
  For the rest of the proof, we are going to construct by transfinite recursion a list~$(\ell_\alpha)_{\alpha\in\kappa}$ of points of $L$, a list~$(L_\alpha,<_\alpha)_{\alpha\in\kappa}$ of big intervals, and a list~$(h_\alpha)_{\alpha\in\kappa}$ of big paths $h_\alpha \colon L_\alpha \to A\cup B$ satisfying the following properties for all $\xi\in \kappa$:
  \begin{enumerate}[label=($\text{\roman*}_\xi$)]
      \item
      \label{recursion1}
        for all $\alpha<\beta\leq \xi$, we have $\ell_\alpha<\ell_\beta$;
      \item
      \label{recursion2}
        for all $\alpha\leq \xi$, $h_\alpha(\min L_\alpha)=a$ and $h_\alpha(\max L_\alpha)=\gamma(\ell_\alpha)$;
      \item
      \label{recursion3}
        for all $\alpha < \beta \leq \xi$, $L_\alpha$ is a proper initial segment of $L_\beta$;
      \item
      \label{recursion4}
        for all $\alpha < \beta \leq \xi$, the restriction of $h_\beta$ to $L_\alpha$ is $h_\alpha$, and $h_\beta\bigl([\max L_\alpha,\max L_\beta]\bigr)\subseteq \gamma\bigl([\ell_\alpha,\ell_\beta]\bigr)$;
      \item
      \label{recursion5}
        for all $\alpha\leq \xi$, $h_\alpha(L_\alpha)\subseteq A$.
  \end{enumerate}

  By point (i$_\xi$), the map $\alpha\in\kappa \longmapsto \ell_\alpha\in L$ will then be injective, contradicting the assumption on the cardinality of $\kappa$ and hence concluding the proof.

  First, we initialize the recursion as follows: let $(L_0,<_0)$ be the real segment $[0,1]$ equipped with its usual order, let $\ell_0=\min L$, and let $h_0 \colon s\in[0,1]\mapsto a\in A \cup B$ be a constant path.
  For the induction step, fix an ordinal number $0<\xi<\kappa$ and assume that we have constructed $(\ell_\alpha)_{\alpha<\xi}$, $(L_\alpha,<_\alpha)_{\alpha<\xi}$, and $(h_\alpha)_{\alpha<\xi}$.
  Our goal is to construct $\ell_\xi$, $(L_\xi,<_\xi)$, and $h_\xi$.
  We treat the cases of a successor ordinal and of a limit ordinal separately:
  \begin{itemize}
    \item
      Assume that $\xi$ is a successor ordinal, and write $\xi=\delta+1$.
      Since $\gamma(\ell_\delta)\in h_\delta(L_\delta)\subseteq A$ by (ii$_\delta$) and (v$_\delta$), we have $\ell_\delta\neq \max L$, so by \eqref{A-cofinal} we can pick $\ell_{\delta+1}\in L$ with $\ell_{\delta+1}>\ell_\delta$ such that $\gamma(\ell_{\delta+1})\in A$.
      The subset $S := \gamma\bigl([\ell_\delta,\ell_{\delta+1}]\bigr)$ of $A\cup B$ is big-path-connected, so $S\cap A$ is big-path-connected by \eqref{no-seam}.
      Now, we would like to choose a big path from $\gamma(\ell_{\delta})$ to $\gamma(\ell_{\delta+1})$ on~$S\cap A$.
      However, the (non-empty) collection of big paths on any space forms a proper class and we will have to make infinitely many such choices throughout the recursion.
      Hence, to apply the axiom of choice properly, we need to determine an explicit non-empty \emph{set} contained in that class.
      To do that, let 
      \[
        \mathfrak{m}
        :=
        \min\suchthat{
          |J|
        }{
          f \colon J \to S\cap A \text{ big path from $\gamma(\ell_\delta)$ to $\gamma(\ell_{\delta+1})$}
        },
      \]
      which is a well-defined cardinal number, let $\mathcal J$ be a set formed of at least one representative from each homeomorphism class of big intervals $J$ of cardinality~$\mathfrak m$ (for example, $\mathcal J \subseteq \{\mathfrak m\} \times \mathcal P(\mathfrak m)$ is the set of all topological spaces $(\mathfrak m, \tau)$ obtained by equipped $\mathfrak m$ with some topology~$\tau$ making it into a big interval), and define the non-empty set
      \[
        E
        :=
        \suchthat{
          f \colon J \to S\cap A
        }{
          \text{big path from $\gamma(\ell_\delta)$ to $\gamma(\ell_{\delta+1})$ with }  J  \in \mathcal J
        }.
      \]
      This technical issue out of the way, pick an arbitrary element of~$E$, which is a big path $f_\delta\colon J_\delta \to S\cap A$ from $\gamma(\ell_\delta)$ to $\gamma(\ell_{\delta+1})$ as desired.
      
      Now, let $(L_{\delta+1},<_{\delta+1})$ be the big interval obtained by concatenating $L_\delta$ and $J_\delta$ and by identifying $\max L_{\delta}$ and $\min J_\delta$.
      Since $h_\delta(\max L_\delta)=f_\delta(\min J_\delta)$, we can define a big path $h_{\delta+1}\colon L_{\delta+1} \to A \cup B$ by setting for all $t\in L_{\delta+1}$,
      \[
        h_{\delta+1}(t)
        :=
        \begin{cases}
          h_\delta(t) &\text{ if }x\in L_\delta,\\
          f_\delta(t) &\text{ if }x\in J_\delta.
        \end{cases}
      \]
      By recursion hypothesis, ($\text{i}_{\delta+1}$--$\text{iii}_{\delta+1}$) and ($\text{v}_{\delta+1}$) are clear.
      For ($\text{iv}_{\delta+1}$), observe that for any $\alpha\leq \delta$ we have 
      $h_{\delta+1}\bigl([\max L_\alpha,\max L_{\delta+1}]\bigr) = h_\delta\bigl([\max L_\alpha,\max L_\delta]\bigr) \cup f_\delta(J_\delta)$ which is contained in $\gamma\bigl([\ell_\alpha,\ell_{\delta}]\bigr)\cup S=\gamma\bigl([\ell_\alpha,\ell_{\delta+1}]\bigr)$.
      This concludes the induction step for successor ordinals.

  \item
    Assume now that $\xi$ is a limit ordinal.
    Let $\ell_\xi := \sup_{\alpha<\xi}\ell_\alpha$, and let $L_\xi$ be the union of the~$(L_\alpha)_{\alpha<\xi}$ supplemented with an additional point, i.e.,
    \[
      L_\xi
      :=
      \left(
        \bigcup_{\alpha<\xi}L_\alpha
      \right)
      \sqcup
      \{\xi\}.
    \]
    For $s,t \in L_\xi$, write $s\leq_\xi t$ if $t=\xi$ or if there is $\alpha<\xi$ such that $s,t\in L_\alpha$ and $s\leq_\alpha t$.
    As the non-strict orders $\leq_\alpha$ are compatible for all $\alpha < \xi$ by ($\text{iii}_\alpha$), the binary relation~$\leq_\xi$ is a well-defined (non-strict) total order on $L_\xi$, making $(L_\xi\setminus\{\xi\}, <_\xi)$ into a dense order-complete linear order.
    Moreover, using ($\text{iii}_\alpha$) for all $\alpha<\xi$ together with the fact that~$\xi$ is a limit ordinal, we see that $\xi=\sup_{\alpha<\xi}\max L_\alpha =\sup (L_\xi\setminus\{\xi\})$.
    Hence, $(L_\xi,<_\xi)$ is a big interval.
    Furthermore, for any $\alpha < \xi$, $L_\alpha$ is a proper initial segment of $L_\xi$, so (iii$_\xi$) holds.
    
    Define the map $h_\xi \colon L_\xi \to A\cup B$ as follows:
    \[
      h_\xi(t)
      :=
      \begin{cases}
          h_\alpha(t) &\text{ if }t\in L_\alpha\text{ for some }\alpha<\xi,\\
          \gamma(\ell_\xi) &\text{ if }t=\xi.
      \end{cases}
    \]
    Thanks to ($\text{iv}_\alpha$) for $\alpha<\xi$, $h_\xi$ is well-defined.
    Before proving that~$h_\xi$ is continuous, let us show~($\text{iv}_\xi$).
    By definition, the restriction of~$h_\xi$ on~$L_\alpha$ is~$h_\alpha$ for any $\alpha < \xi$.
    Moreover, we have
    \begin{align*}
      h_\xi\bigl([\max L_\alpha,\max L_\xi]\bigr)=h_\xi\bigl([\max L_\alpha,\xi]\bigr)
      &=
      \bigcup_{\alpha< \beta<\xi}
        h_\xi\bigl(
          [\max L_\alpha,\max L_\beta]
        \bigr)
      \cup \{h_\xi(\xi)\}
      \\
      &=
      \bigcup_{\alpha< \beta<\xi}
        h_\beta\bigl(
          [\max L_\alpha,\max L_\beta]
        \bigr)
      \cup \{\gamma(\ell_\xi)\}
      \\
      &
      \subseteq
      \bigcup_{\alpha< \beta<\xi}
        \gamma\bigl(
          [\ell_\alpha,\ell_\beta]
        \bigr)
      \cup \{\gamma(\ell_\xi)\}
      \subseteq
      \gamma\bigl([\ell_\alpha,\ell_\xi]\bigr).
    \end{align*}
    The map $h_\xi$ is continuous on $L_\alpha$ for all $\alpha<\xi$ as it coincides there with the continuous map~$h_\alpha$, and it is thus continuous on $L_\xi\setminus\{\xi\} = \bigcup_{\alpha < \xi} L_\alpha$.
    Furthermore, for any neighborhood~$V$ of $h_\xi(\xi)=\gamma(\ell_\xi)$, by continuity of $\gamma$ and because $\ell_\xi = \sup_{\alpha < \xi} \ell_\alpha$, there exists $\alpha<\xi$ such that $\gamma\bigl([\ell_\alpha,\ell_\xi]\bigr)\subseteq V$, and then $h_\xi\bigl([\max L_\alpha,\max L_\xi]\bigr)\subseteq V$ by (iv$_\xi$).
    This proves that $h_\xi$ is also continuous at $\xi$, so $h_\xi$ is indeed a big path.

    It remains to show ($\text{i}_\xi$), ($\text{ii}_\xi$), and (v$_\xi$).
    Point~($\text{i}_\xi$) immediately follows from~($\text{i}_\beta$) for $\beta < \xi$, and from the fact that~$\xi$ is a limit ordinal (for any $\alpha < \xi$, we also have $\alpha + 1 < \xi$, so $\ell_\alpha < \ell_{\alpha+1} \leq \ell_\xi$ by ($\text{i}_{\alpha+1}$) and the definition of $\ell_\xi$).
    Point~(ii$_\xi$) follows from~(ii$_\alpha$) when $\alpha < \xi$, and by construction of $h_\xi$ when $\alpha=\xi$.
    To prove ($\text{v}_\xi$), observe from ($\text{v}_\alpha$) and~(iv$_\xi$) that $h_\xi(L_\alpha) = h_\alpha(L_\alpha) \subseteq A$ for all $\alpha<\xi$, which implies that $h_\xi(L_\xi\setminus\{\xi\})\subseteq A$, and then apply \eqref{right-end_in_A} to deduce that $h_\xi(L_\xi)\subseteq A$.
    This concludes the induction step for limit ordinals.
  \end{itemize}
  As previously explained, this concludes the proof.
\end{proof}

\begin{remark}
  In the previous proof, in order to properly apply the axiom of choice, we have had to restrict to those big paths whose domain has minimal cardinality among those joining two given points.
  Explicit upper bounds are known about this minimal cardinality, cf.~\cite{cannon2} (for Hausdorff spaces) and~\cite[Corollary~4.5]{penrod}.
\end{remark}

\section{Bernstein rectangles}
\label{sn:bernstein}

This section is devoted to the proof of \Cref{bernstein_rectangle}.
Let $Y$ and $Z$ be topological spaces that are Hausdorff, locally compact and $\sigma$-compact, and first-countable.
We endow the space $Y\times Z$ with the product topology.
We call \emph{line} a subset of~$Y\times Z$ of the form $\{y\}\times Z$ or $Y\times\{z\}$, where $y\in Y$ and $z\in Z$.
For any $S_Y\subseteq Y$ and $S_Z\subseteq Z$, we set
\[
  S_Y\oplus S_Z
  :=
  (S_Y\times Z)\cup (Y\times S_Z).
\]
To prove \Cref{bernstein_rectangle}, we will mimic the usual construction of a Bernstein set in the real line by transfinite recursion, by choosing an enumeration of all closed subsets of $Y \times Z$ and, for each such subset, adding some of its points to~$A_i$ and~$B_j$.
This approach requires us to control the number of such subsets, and to ensure that at each step of the recursion, the next subset is not covered by the lines induced by the points already selected.
Both of these tasks will be taken care of by the following theorem (which is of independent interest).

\begin{theorem}
  \label{thm:cantor_for_product}
  If a closed subset $F \subseteq Y\times Z$ cannot be covered by countably many lines, then it contains a separable closed subset that cannot be covered by fewer than $2^{\aleph_0}$ lines.
\end{theorem}

\Cref{thm:cantor_for_product} establishes a weak form of the continuum hypothesis, in the spirit of the Cantor--Bendixson theorem (\cite[Corollary 6.5]{kechris1995classical}): any closed subset of~$Y\times Z$ covered by fewer than $2^{\aleph_0}$ lines can be covered by countably many lines.
Our proof is mainly based on an idea that was privately communicated to us by Rolf Suabedissen.
Just like for the Cantor--Bendixson theorem, our proof relies on the construction of a Cantor scheme (cf.~\cite[Definition 6.1]{kechris1995classical}) whose blocks are in disjoint horizontal and vertical strips, and are centered on the \emph{condensation points} of~$F$, defined as follows:

\begin{definition}
  Let $S$ be a topological space, and let $A\subseteq S$.
  A point $x\in S$ is a \emph{condensation point} of $A$ if every neighborhood of $x$ contains uncountably many points of $A$.
  (Note that $x$ does not need to belong to $A$.)
\end{definition}

\begin{lemma}
  \label{lemma:acc_point}
  Let $S$ be a $\sigma$-compact topological space.
  If $A\subseteq S$ is uncountable, then there exists a condensation point of~$A$.
\end{lemma}

\begin{proof}
  Let $(K_n)_{n\in\N}$ be a sequence of compact subsets of $S$ such that $S=\bigcup_{n\in\N} K_n$.
  Assume that~$A$ has no condensation point, i.e., that
  $
    \bigcup
      \suchthat{
        U \textnormal{ open in }S
      }{
        U\cap A \textnormal{ is countable}
      }
    =
    S
  $.
  Then, each compact subspace~$K_n$ is covered by finitely many open sets $U_1,\ldots,U_m$ for which $U_i \cap A$ is countable, so $K_n\cap A\subseteq \bigcup_{i=1}^m (U_i\cap A)$ is countable.
  But then~$A = \bigcup_{n \in \N} (K_n \cap A)$ is countable, which is a contradiction.
\end{proof}

\begin{proof}[Proof of \Cref{thm:cantor_for_product}]
  First, using the hypothesis on $F$ and by transfinite recursion on the first uncountable ordinal number, we construct an uncountable subset $A$ of~$F$ such that distinct points of~$A$ have both coordinates distinct.
  Next, since $Y$ is first-countable, Hausdorff, and locally compact, \cite[Theorem 29.2]{munkres} ensures that we can fix for every $y\in Y$ and for every $z \in Z$ a countable local base $(N_k^y)_{k\geq 1}$ of compact neighborhoods of $y$ and a countable local base $(N_k^z)_{k\geq 1}$ of compact neighborhoods of $z$.

  The main idea of the proof (due to Rolf Suabedissen) is to construct the desired subset of~$F$ via a Cantor scheme.
  To do so, let us introduce some notation.
  We denote the natural projections by $\pi_Y\colon Y\times Z \to Y$ and by $\pi_Z\colon Y\times Z \to Z$.
  Denote by $\words$ the set of all finite binary words (that is, whose terms belong to $\{0,1\}$), and by $\varnothing \in \words$ the empty word.
  For each word $u\in \words$, denote by $u\zero$ (resp.~$u\one$) the word obtained by appending~$0$ (resp.~$1$) at the end of~$u$.
  We will recursively construct a list $(C_u)_{u\in \words}$ of closed subsets of $F$ such that, for all $u\in \words$:
  \begin{enumroman}
    \item
      if $u\neq \varnothing$, then $C_u$ is compact;
    \item
      $\pi_Y(C_{u\zero})\cap \pi_Y(C_{u\one})=\pi_Z(C_{u\zero})\cap \pi_Z(C_{u\one})=\emptyset$;
    \item
      $C_{u\zero}\cup C_{u\one}\subseteq C_u$;
    \item
      $C_u\cap A$ is uncountable.
  \end{enumroman}

  We begin by setting $C_\varnothing=F$, which is closed, and then $C_\varnothing \cap A = A$ is uncountable.
  Let $u\in \words$ and assume that we have already constructed $C_u$.
  Since~$Y$ and~$Z$ are $\sigma$-compact, $Y\times Z$ is also $\sigma$-compact, so by \Cref{lemma:acc_point} there is a condensation point~$(y_0, z_0)$ of~$C_u\cap A$.
  The set
  \[
    \bigcap_{k\geq 1}
      (N_k^{y_0}\oplus N_k^{z_0})
      \cap
      (C_u\cap A)
    =
    (\{y_0\}\oplus \{z_0\})
    \cap
    C_u
    \cap
    A
  \]
  cannot have more than two elements because any two distinct points of~$A$ have both of their coordinates distinct.
  Hence, the sets $(C_u\cap A)\setminus (N_{k}^{y_0}\oplus N_{k}^{z_0})$ for $k\geq 1$ cover an uncountable set by (iv).
  It follows that there is $k_0\geq 1$ such that $(C_u\cap A)\setminus (N_{k_0}^{y_0}\oplus N_{k_0}^{z_0})$ is uncountable.
  Using \Cref{lemma:acc_point} again, we fix a condensation point $(y_1,z_1)$ of $(C_u\cap A)\setminus (N_{k_0}^{y_0}\oplus N_{k_0}^{z_0})$.
  Observe that $y_1\neq y_0$ and $z_1\neq z_0$, because otherwise $N_{k_0}^{y_0}\oplus N_{k_0}^{z_0}$ would be a neighborhood of $(y_1,z_1)$ that does not meet $(C_u\cap A)\setminus (N_{k_0}^{y_0}\oplus N_{k_0}^{z_0})$.
  Thus, since $Y$ and $Z$ are Hausdorff, there exists $k\geq 1$ such that $N_{k}^{y_0}\cap N_{k}^{y_1}=N_{k}^{z_0}\cap N_{k}^{z_1}=\emptyset$.
  Then, let
  \[
    C_{u\zero}:=C_u\cap (N_k^{y_0}\times N_k^{z_0})
    \andd
    C_{u\one}:=C_u\cap (N_k^{y_1}\times N_k^{z_1}).
  \]
  Points (ii) and (iii) are clear.
  By definition of condensation points, $N_{k}^{y_0}\times N_{k}^{z_0}$ and $N_{k}^{y_1}\times N_{k}^{z_1}$ both have uncountable intersections with $C_u\cap A$: this entails (iv).
  Finally, $C_{u\zero}$ and $C_{u\one}$ are compact as the intersections of the closed subset $C_u$ with the compact sets $N_{k}^{y_0}\times N_{k}^{z_0}$ and $N_{k}^{y_1}\times N_{k}^{z_1}$.
  Since $Y\times Z$ is Hausdorff, $C_{u\zero}$ and $C_{u\one}$ are also closed.
  This concludes the construction of $(C_u)_{u\in \words}$.

  We are now ready to build the desired separable subset of $F$.
  Thanks to (iv), we can pick arbitrarily $c_u \in C_u$ for all $u \in \words$.
  Next, define~$G$ as the closure of $\suchthat{c_u}{u\in \words}$.
  By definition, $G$ is closed and separable (since $\words$ is countable).
  Moreover, we have $G\subseteq F$ because $F$ is closed and $c_u\in C_u\subseteq F$ for all $u\in \words$.
  It only remains to show that $G$ cannot be covered by fewer than~$2^{\aleph_0}$ lines.
  For any $v=(v_n)_{n\in \N}\in \{0,1\}^{\N}$, let
  \[G_v := \bigcap_{n\geq 0}G\cap C_{(v_0,\ldots,v_n)}.\]
  Since $Y\times Z$ is Hausdorff, Cantor's intersection theorem ensures that $G_v$ is non-empty.
  Furthermore, if $w\in \{0,1\}^{\N}$ is distinct from $v$, then $\pi_Y(G_v)\cap \pi_Y(G_w)=\pi_Z(G_v)\cap \pi_Z(G_w)=\emptyset$ by (ii).
  It follows that each line can only cover $G_v$ for at most one $v \in \{0,1\}^\N$, so $G$ cannot be covered by fewer than $|\{0,1\}^{\N}|=2^{\aleph_0}$ lines.
\end{proof}

\begin{proof}[Proof of \Cref{bernstein_rectangle}]
  The closure operator induces a surjection of $(Y\times Z)^{\N}$ into the set of separable closed subsets of $Y\times Z$, so the latter has cardinality at most $|Y\times Z|^{\aleph_0}=(2^{\aleph_0}\cdot 2^{\aleph_0})^{\aleph_0}=2^{2\cdot \aleph_0^2}=2^{\aleph_0}$.
  This allows us to enumerate all separable closed subsets of $Y\times Z$ that cannot be covered by fewer than $2^{\aleph_0}$ lines as $(F^{(\alpha)})_{\alpha\in \kappa}$, where $\kappa$ is an ordinal number not bigger than the smallest ordinal number with cardinality $2^{\aleph_0}$.
  By transfinite recursion, we are going to construct two injective maps
  \begin{align*}
    (i,j,\alpha)\in \{1,2,3\}\times \{1,2,3\}\times \kappa
    &\longmapsto
    a_{i,j}^{(\alpha)} \in Y,\\
    (i,j,\alpha)\in \{1,2,3\}\times \{1,2,3\}\times \kappa
    &\longmapsto
    b_{i,j}^{(\alpha)} \in Z,
  \end{align*}
  such that $\big(a_{i,j}^{(\alpha)},b_{i,j}^{(\alpha)}\big)\in F^{(\alpha)}$ for all $i,j\in\{1,2,3\}$ and $\alpha\in\kappa$.

  Let $\alpha\in \kappa$ and assume that we have already constructed $a_{i,j}^{(\beta)}$ and $b_{i,j}^{(\beta)}$ for all $i,j\in\{1,2,3\}$ and $\beta<\alpha$.
  Since $|\alpha|<2^{\aleph_0}$, $F^{(\alpha)}$ cannot be covered by $18 \cdot |\alpha| +18$ lines by assumption.
  Therefore,
  \[
    F^{(\alpha)}
    \setminus
    \Big(
      \suchthat{
        a_{k,l}^{(\beta)}
      }{
        k,l\in \{1,2,3\}, \, \beta<\alpha
      }
      \oplus
      \suchthat{
        b_{k,l}^{(\beta)}
      }{
        k,l\in \{1,2,3\}, \, \beta<\alpha
      }
    \Big)
  \]
  contains at least $9$ points whose both coordinates are pairwise distinct.
  Denote them by $\big(a_{i,j}^{(\alpha)},b_{i,j}^{(\alpha)}\big)$ for $i,j\in\{1,2,3\}$, thereby defining $a_{i,j}^{(\alpha)}$ and $b_{i,j}^{(\alpha)}$ as desired.
  Note that $a_{i,j}^{(\alpha)}$ (resp.~$b_{i,j}^{(\alpha)}$) is indeed distinct from the $a_{k,l}^{(\beta)}$ (resp.~$b_{k,l}^{(\beta)}$) for $k,l\in \{1,2,3\}$ and $\beta<\alpha$.

  Now, for all $i,j\in\{1,2,3\}$, let us set
  \[
    A^{\circ}_i
    :=
    \suchthat{
      a_{i,k}^{(\alpha)}
    }{
      \alpha\in\kappa, \, k\in\{1,2,3\}
    }
    \andd
    B_j^\circ :=
    \suchthat{
      b_{k,j}^{(\alpha)}
    }{
      \alpha\in\kappa, \, k\in\{1,2,3\} 
    }.
  \]
  By injectivity, the sets $A_1^\circ,A_2^\circ,A_3^\circ$ (resp.~$B_1^\circ,B_2^\circ,B_3^\circ$) are pairwise disjoint.
  Thus, setting
  \begin{align*}
    A_1 := A_1^\circ, \quad
    A_2 := A_2^\circ, &\andd
    A_3 := Y\setminus (A_1^\circ\cup A_2^\circ)
    = A_3^\circ \cup \comp{(A_1^\circ \cup A_2^\circ \cup A_3^\circ)},
    \\
    B_1 := B_1^\circ, \quad
    B_2 := B_2^\circ, &\andd
    B_3 := Z\setminus (B_1^\circ\cup B_2^\circ)
    = B_3^\circ \cup \comp{(B_1^\circ \cup B_2^\circ \cup B_3^\circ)},
  \end{align*}
  we obtain partitions $\{A_1,A_2,A_3\}$ and $\{B_1,B_2,B_3\}$ of $Y$ and $Z$ respectively.
  Let $F$ be a closed subset of $Y\times Z$ that cannot be covered by countably many lines.
  By \Cref{thm:cantor_for_product}, there is an ordinal number $\alpha\in\kappa$ such that $F^{(\alpha)}\subseteq F$.
  Then, for any $i,j\in\{1,2,3\}$, $F\cap (A_i\times B_j)$ contains the point $\big(a_{i,j}^{(\alpha)},b_{i,j}^{(\alpha)}\big)$ by construction.
\end{proof}

\printbibliography

\end{document}